\documentclass[english,nohyperref]{article}
\usepackage[T1]{fontenc}
\usepackage[latin9]{inputenc}
\usepackage{geometry}
\usepackage{array}
\usepackage{float}
\usepackage{booktabs}
\usepackage{mathtools}
\usepackage{multirow}
\usepackage{amsmath}
\usepackage{amsthm}
\usepackage{amssymb}
\usepackage{graphicx}

\makeatletter

\providecommand{\tabularnewline}{\\}
\floatstyle{ruled}
\newfloat{algorithm}{tbp}{loa}
\providecommand{\algorithmname}{Algorithm}
\floatname{algorithm}{\protect\algorithmname}

\theoremstyle{plain}
\newtheorem{thm}{\protect\theoremname}[section]
\theoremstyle{plain}
\newtheorem{lem}[thm]{\protect\lemmaname}
\ifx\proof\undefined
\newenvironment{proof}[1][\protect\proofname]{\par
	\normalfont\topsep6\p@\@plus6\p@\relax
	\trivlist
	\itemindent\parindent
	\item[\hskip\labelsep\scshape #1]\ignorespaces
}{%
	\endtrivlist\@endpefalse
}
\providecommand{\proofname}{Proof}
\fi

\@ifundefined{date}{}{\date{}}

\@ifundefined{showcaptionsetup}{}{%
 \PassOptionsToPackage{caption=false}{subfig}}
\usepackage{subfig}
\makeatother

\usepackage{babel}
\providecommand{\lemmaname}{Lemma}
\providecommand{\theoremname}{Theorem}

\begin{document}
\global\long\def\E{\mathbb{E}}%
\global\long\def\Var{\mathrm{Var}}%
\global\long\def\Cov{\mathrm{Cov}}%
\global\long\def\dom{\mathcal{X}}%
\global\long\def\R{\mathbb{R}}%
\global\long\def\avx{\overline{x}}%

\allowdisplaybreaks

\title{Adaptive Accelerated (Extra-)Gradient Methods with Variance Reduction}
\author{Zijian Liu\thanks{Equal contribution. Department of Computer Science, Boston University,
\texttt{{zl3067@bu.edu}.}}\and Ta Duy Nguyen\thanks{Equal contribution. Department of Computer Science, Boston University,
\texttt{{taduy@bu.edu}.}} \and Alina Ene\thanks{Department of Computer Science, Boston University, \texttt{{aene@bu.edu}}}
\and Huy L. Nguy\~{ê}n\thanks{Khoury College of Computer and Information Science, Northeastern University,
\texttt{{hu.nguyen@northeastern.edu}.}}}

\maketitle

\begin{abstract}
In this paper, we study the finite-sum convex optimization problem
focusing on the general convex case. Recently, the study of variance
reduced (VR) methods and their accelerated variants has made exciting
progress. However, the step size used in the existing VR algorithms
typically depends on the smoothness parameter, which is often unknown
and requires tuning in practice. To address this problem, we propose
two novel adaptive VR algorithms: \textit{Adaptive Variance Reduced
Accelerated Extra-Gradient} (AdaVRAE) and\textit{ Adaptive Variance
Reduced Accelerated Gradient} (AdaVRAG). Our algorithms do not require
knowledge of the smoothness parameter. AdaVRAE uses $\mathcal{O}\left(n\log\log n+\sqrt{\frac{n\beta}{\epsilon}}\right)$
gradient evaluations and AdaVRAG uses $\mathcal{O}\left(n\log\log n+\sqrt{\frac{n\beta\log\beta}{\epsilon}}\right)$
gradient evaluations to attain an $\mathcal{O}(\epsilon)$-suboptimal
solution, where $n$ is the number of functions in the finite sum
and $\beta$ is the smoothness parameter. This result matches the
best-known convergence rate of non-adaptive VR methods and it improves
upon the convergence of the state of the art adaptive VR method, AdaSVRG.
We demonstrate the superior performance of our algorithms compared
with previous methods in experiments on real-world datasets.
\end{abstract}

\section{Introduction}

\begin{table*}[t]
\begin{centering}
{\small{}\caption{Our results and comparison with prior works.}
\label{tb:alg-comparison}\vspace{0.1in}
}{\small\par}
\par\end{centering}
\centering{}{\small{}}%
\begin{tabular}{ccc}
\toprule 
{\small{}Algorithm} & {\small{}General convex} & {\small{}Adaptive}\tabularnewline
\midrule
\midrule 
{\small{}SVRG \cite{johnson2013accelerating}} & {\small{}-} & {\small{}No}\tabularnewline
\midrule 
{\small{}$\text{SVRG}^{++}$ \cite{allen2016improved}} & {\small{}$\mathcal{O}\left(n\log\frac{\beta}{\epsilon}+\frac{\beta}{\epsilon}\right)$} & {\small{}No}\tabularnewline
\midrule 
{\small{}Katyusha \cite{allen2017katyusha}} & {\small{}$\mathcal{O}\left(n\log\frac{\beta}{\epsilon}+\sqrt{\frac{n\beta}{\epsilon}}\right)$} & {\small{}No}\tabularnewline
\midrule 
{\small{}VARAG \cite{lan2019unified}} & {\small{}$\mathcal{O}\left(n\min\left\{ \log\frac{\beta}{\epsilon},\log n\right\} +\sqrt{\frac{n\beta}{\epsilon}}\right)$} & {\small{}No}\tabularnewline
\midrule 
{\small{}VRADA \cite{song2020variance}} & {\small{}$\mathcal{O}\left(n\min\left\{ \log\log\frac{\beta}{\epsilon},\log\log n\right\} +\sqrt{\frac{n\beta}{\epsilon}}\right)$} & {\small{}No}\tabularnewline
\midrule 
\multirow{2}{*}{{\small{}AdaSVRG \cite{dubois2021svrg}}} & {\small{}$\mathcal{O}\left(\frac{n\beta}{\epsilon}\right)$ (fixed
sized inner loop, only if $\epsilon=\Omega(\frac{\beta}{n})$)} & \multirow{2}{*}{{\small{}Yes}}\tabularnewline
\cmidrule{2-2} 
 & {\small{}$\mathcal{O}\left(n\log\frac{\beta}{\epsilon}+\frac{\beta}{\epsilon}\right)$
(multi-stage)} & \tabularnewline
\midrule 
{\small{}AdaVRAE (unknown $\beta$) }\textbf{\small{}(This Paper)} & {\small{}$\mathcal{O}\left(n\min\left\{ \log\log\frac{\beta}{\epsilon},\log\log n\right\} +\sqrt{\frac{n\beta}{\epsilon}}\right)$} & {\small{}Yes}\tabularnewline
\cmidrule{2-3} \cmidrule{3-3} 
{\small{}AdaVRAE (known $\beta$) }\textbf{\small{}(This Paper)} & {\small{}$\mathcal{O}\left(n\min\left\{ \log\log\frac{\beta}{\epsilon},\log\log n\right\} +\sqrt{\frac{n\beta}{\epsilon}}\right)$} & {\small{}No}\tabularnewline
\midrule 
{\small{}AdaVRAG (unknown $\beta$)}\textbf{\small{} (This Paper)} & {\small{}$\mathcal{O}\left(n\min\left\{ \log\log\frac{\beta\log\beta}{\epsilon},\log\log n\right\} +\sqrt{\frac{n\beta\log\beta}{\epsilon}}\right)$} & {\small{}Yes}\tabularnewline
\cmidrule{2-3} \cmidrule{3-3} 
{\small{}AdaVRAG (known $\beta$)}\textbf{\small{} (This Paper)} & {\small{}$\mathcal{O}\left(n\min\left\{ \log\log\frac{\beta}{\epsilon},\log\log n\right\} +\sqrt{\frac{n\beta}{\epsilon}}\right)$} & {\small{}No}\tabularnewline
\midrule
{\small{}Lower Bound \cite{woodworth2016tight}} & {\small{}$\Omega\left(n+\sqrt{\frac{n\beta}{\epsilon}}\right)$} & {\small{}-}\tabularnewline
\bottomrule
\end{tabular}{\small\par}
\end{table*}

In this paper, we consider the finite-sum optimization problem in
the form of 
\begin{equation}
\min_{x\in\mathcal{X}}\left\{ \frac{1}{n}\sum_{i=1}^{n}f_{i}(x)+h(x)\right\} \label{eq:problem}
\end{equation}
where each function $f_{i}$ is convex and $\beta$-smooth, $h$ is
convex and potentially nonsmooth but admitting an efficient proximal
operator, and $\dom\subseteq\R^{d}$ is a closed convex set. Additionally,
we further assume that $\dom$ is compact when $\beta$ is unknown.
Problem (\ref{eq:problem}) has found a wide range of applications
in machine learning, typically in empirical risk minimization problems,
and has been extensively studied in the past few years. 

Among existing approaches to solve this problem, variance reduced
(VR) methods \cite{johnson2013accelerating,defazio2014saga,schmidt2017minimizing,roux2012stochastic}
have recently shown significant improvement over the classic stochastic
gradient methods such as stochastic gradient descent (SGD) and its
variants. For example, in strongly convex problems, VR methods such
as \cite{allen2017katyusha,lan2019unified,lin2015universal} can achieve
the optimal number of gradient evaluations of $\mathcal{O}\left((n+\sqrt{n\kappa})\log\frac{1}{\epsilon}\right)$
to attain an $\mathcal{O}(\epsilon)$-suboptimal solution, where $\kappa$
is the condition number, which improves over full-batch gradient descent
($\mathcal{O}\left(n\kappa\log\frac{1}{\epsilon}\right)$) and Nesterov's
accelerated gradient descent \cite{nesterov1983method,nesterov2003introductory}
($\mathcal{O}\left(n\sqrt{\kappa}\log\frac{1}{\epsilon}\right)$).
For general convex problems, the current state-of-the-art VR methods,
namely VRADA \cite{song2020variance} can find an $\mathcal{O}(\epsilon)$-suboptimal
solution using $\mathcal{O}\left(n\log\log n+\sqrt{\frac{n\beta}{\epsilon}}\right)$
gradient evaluations, which nearly-matches the lower bound of $\Omega\left(n+\sqrt{\frac{n\beta}{\epsilon}}\right)$\cite{woodworth2016tight}. 

However, most of existing VR gradient methods have the same limitation
as classic gradient methods; that is, they require the prior knowledge
of the smoothness parameter in order to set the step size. Lacking
this information, one may have to carefully perform hyper-parameter
tuning to avoid the situation that the algorithm divergences or converges
too slowly due to too large or too small step size. This limitation
of gradient methods motivates the development of methods that aim
to adapt to unknown problem structures. A notable line of work starting
with the influential AdaGrad algorithm\textbf{ }has designed a family
of gradient descent based methods that set the step size based on
the gradients or iterates observed in previous iterations \cite{McMahanS10,duchi2011adaptive,kingma2014adam,Levy17,levy2018online,bach2019universal,Cutkosky19,KavisLBC19,joulani2020simpler,ene2021adaptive,antonakopoulos2021adaptive,ene2021variational}.
Remarkably, these works have shown that, in the setting where we have
access to the exact full gradient in each iteration, it is possible
to match the convergence rates of both unaccelerated and accelerated
gradient descent methods without any prior knowledge of the smoothness
parameter. These methods have also been analyzed in the stochastic
setting under a bounded variance assumption, and they achieve a convergence
rate that is comparable to that of SGD. 

Given the theoretical and practical success of adaptive methods, it
is natural to ask whether one can design VR methods that achieve state
of the art convergence guarantees without any prior knowledge of the
smoothness parameter. The recent work of \cite{dubois2021svrg} gives
the first adaptive VR method --- AdaSVRG --- with the gradient complexity
of $\mathcal{O}\left(n\log\frac{\beta}{\epsilon}+\frac{\beta}{\epsilon}\right)$.
AdaSVRG builds on the AdaGrad \cite{duchi2011adaptive} and SVRG algorithms
\cite{johnson2013accelerating}, both of which are not accelerated.

\textbf{Our contributions:} In this work, we take this line of work
further and design the first accelerated VR methods that do not require
any prior knowledge of the smoothness parameter. Our algorithms\textit{,
Adaptive Variance Reduced Accelerated Extra-Gradient} (AdaVRAE) and\textit{
Adaptive Variance Reduced Accelerated Gradient} (AdaVRAG), only use
$\mathcal{O}\left(n\log\log n+\sqrt{\frac{n\beta}{\epsilon}}\right)$
and $\mathcal{O}\left(n\log\log n+\sqrt{\frac{n\beta\log\beta}{\epsilon}}\right)$
gradient evaluations respectively to attain an $\mathcal{O}(\epsilon)$-suboptimal
solution when $\beta$ is unknown, both of which significantly improve
the convergence rate of AdaSVRG. Table \ref{tb:alg-comparison} compares
our algorithms and prior VR methods and Section \ref{sec:algorithms}
discusses our algorithmic approaches and techniques. The convergence
rate of AdaVRAE matches up to constant factors the best-known convergence
rate of non-adaptive VR methods \cite{song2020variance,joulani2020simpler}.
Both of our algorithms follow a different approach from these methods
that is based on extra-gradient and mirror descent, instead of dual
averaging.

We demonstrate the efficiency of our algorithms in practice on multiple
real-world datasets. We show that AdaVRAG and AdaVRAE are competitive
with existing standard and adaptive VR methods while having the advantage
of not requiring hyperparameter tuning, and in many cases AdaVRAG
outperforms these benchmarks. 

\subsection{Related work}

\textbf{Variance reduced gradient methods:} Variance reduction technique
\cite{roux2012stochastic,schmidt2017minimizing,shalev2013stochastic,mairal2013optimization,johnson2013accelerating,defazio2014saga}
has been proposed to improve the convergence rate of stochastic gradient
descent algorithms in the finite sum problem and has since become
widely-used in many successful algorithms. Notable improvements can
be seen in strongly convex optimization problems where earliest algorithms
such as SVRG \cite{johnson2013accelerating} or SAGA \cite{defazio2014saga}
obtain $\mathcal{O}\left((n+\kappa)\log\frac{1}{\epsilon}\right)$
convergence rate compared with $\mathcal{O}\left(\frac{\sigma^{2}\kappa}{\beta\epsilon}\right)$
of plain SGD, with the latter requiring an additional assumption on
the $\sigma^{2}$-boundedness of the variance term, i.e., $\E_{i}\left[\left\Vert \nabla f_{i}(x)-\nabla f(x)\right\Vert ^{2}\right]\le\sigma^{2}$.
However, these non-accelerated methods do not achieve the optimal
convergence rate. Recent works such as \cite{lin2015universal,allen2017katyusha,lan2019unified}
focus on designing accelerated methods and successfully match the
optimal lower bound for strongly convex optimization of $\Omega\left((n+\sqrt{n\kappa})\log\frac{1}{\epsilon}\right)$
given by \cite{lan2018random}. 

In non-strongly convex problems, however, existing works do not yet
match the lower bound of $\Omega\left(n+\sqrt{\frac{\beta n}{\epsilon}}\right)$
shown in \cite{woodworth2016tight}. The best effort so far can be
found in the line of accelerated methods started by \cite{allen2017katyusha}
and followed by \cite{allen2018katyusha,lan2019unified,li2021anita}
that rely on incorporating the checkpoint in each update. AdaVRAG
follows the same idea but offers simpler update and more efficient
choice of coefficients that results in a better convergence rate,
equivalent to VRADA \cite{song2020variance}. By comparison, while
VRADA is a dual-averaging scheme, AdaVRAG is a mirror descent method
and AdaVRAE is an extra-gradient algorithm.

In a different line of research \cite{allen2016variance,fang2018spider,zhou2018stochastic},
variance reduction has been applied to non-convex optimization to
find critical points with much better convergence rate. 

\textbf{Adaptive methods with variance reduction:} There has been
extensive research on adaptive methods \cite{duchi2011adaptive,kingma2014adam,reddi2018convergence,tieleman2012lecture,dozat2016incorporating}
in the setting where we compute a full gradient in each iteration.
However, there are only few works combining adaptive methods with
VR techniques in the finite sum setup. Most relevant for our work
is AdaSVRG \cite{dubois2021svrg}. This algorithm is built upon SVRG
which as mentioned earlier is a non-accelerated method and has a slower
convergence rate. AdaSVRG uses the gradient norm to update the step
size, similar to \cite{duchi2011adaptive} and the step is reset in
every epoch, which could lead to step sizes that are too large in
later stages. In contrast, both AdaVRAG and AdaVRAE are accelerated
VR methods and use a cumulative step size. AdaVRAG uses the iterate
movement to update the step size, as in \cite{bach2019universal,ene2021adaptive}.
AdaVRAE improves the convergence rate by a $\sqrt{\log\beta}$ factor
by using the gradient difference similarly to \cite{MohriYang16,joulani2020simpler,ene2021variational}. 

A different line of work considers VR methods that set the step size
using stochastic line search \cite{schmidt2017minimizing,mairal2013optimization}
or Barzilai-Borwein step size \cite{tan2016barzilai,li2020almost}.
The former methods do not have theoretical guarantees, and the latter
methods require knowledge of the smoothness parameter in order to
obtain theoretical bounds.

Recent works design variance-reduced methods for non-convex optimization.
STORM \cite{cutkosky2019momentum} and $\text{STORM}^{+}$\cite{levy2021storm+}
design an adaptive step size, though the former still requires the
smoothness parameter in the step size. Super-Adam \cite{huang2021super}
also requires their parameters to satisfy some inequality involving
the smoothness parameter like STORM.

\subsection{Notation and problem setup }

Let $\left[n\right]$ denote the set $\left\{ 1,2,\cdots,n\right\} $.
For simplicity, we only consider the Euclidean norm $\left\Vert \cdot\right\Vert \mathrel{\coloneqq}\left\Vert \cdot\right\Vert _{2}$
(Our work can be extended to $\left\Vert x\right\Vert _{A}\coloneqq\sqrt{x^{\top}Ax}$
for any $A\succ0$ with almost no change). $x^{+}$ represents $\max\left\{ x,0\right\} $.

We are interested in solving the following problem 
\[
\min_{x\in\mathcal{X}}\left\{ F(x)=f(x)+h(x)\right\} 
\]
where $f(x)\coloneqq\frac{1}{n}\sum_{i=1}^{n}f_{i}(x)$ and for $i\in\left[n\right]$,
$f_{i}:\R^{d}\to\R$ and $h:\mathcal{X}\to\R$ are convex functions
with a closed convex set $\mathcal{X}\subseteq\R^{d}$. Let $x^{*}=\arg\min_{x\in\dom}F(x)$.
We say a function $G$ is $\beta$-smooth if $\left\Vert \nabla G(x)-\nabla G(y)\right\Vert \leq\beta\left\Vert x-y\right\Vert $
for all $x,y\in\R^{d}$. Equivalently, we have $G(y)\leq G(x)+\left\langle \nabla G(x),y-x\right\rangle +\frac{\beta}{2}\left\Vert y-x\right\Vert ^{2}$.
In this paper we always assume that each $f_{i}$ is $\beta$-smooth,
which implies that $f$ is also $\beta$-smooth. We assume that we
can efficiently solve optimization problems of the form $\arg\min_{x\in\dom}\left(\gamma h(x)+\frac{1}{2}\left\Vert x-v\right\Vert ^{2}\right)$
where $\gamma\geq0$ and $v\in\R^{d}$. When the smoothness parameter
$\beta$ is unknown, we additionally assume that $\dom$ is compact
with diameter $D$, i.e., $\sup_{x,y\in\dom}\left\Vert x-y\right\Vert \le D$.

\section{Our algorithms and convergence guarantees}

\label{sec:algorithms}

\begin{algorithm}
\caption{AdaVRAE\label{alg:AdaVRAE}}

\textbf{Input:} initial point $u^{(0)}$, domain diameter $D$.

\textbf{Parameters:} $\{a^{(s)}\}$,$\left\{ T_{s}\right\} $, $A_{T_{0}}^{(0)}>0$,
$\eta>0$.

$\avx_{0}^{(1)}=z_{0}^{(1)}=u^{(0)}$, compute $\nabla f(u^{(0)})$

Initialize $\gamma_{0}^{(1)}=\gamma$, where $\gamma$ is any small
constant

\textbf{for $s=1$ to $S$:}

$\quad$$A_{0}^{(s)}=A_{T_{s-1}}^{(s-1)}-T_{s}\left(a^{(s)}\right)^{2}$

\textbf{$\quad$for $t=1$ to $T_{s}$:}

$\quad\quad$$x_{t}^{(s)}=\arg\min_{x\in\dom}\left\{ a^{(s)}\left\langle g_{t-1}^{(s)},x\right\rangle +a^{(s)}h(x)+\frac{\gamma_{t-1}^{(s)}}{2}\left\Vert x-z_{t-1}^{(s)}\right\Vert ^{2}\right\} $

$\quad\quad$Let $A_{t}^{(s)}=A_{t-1}^{(s)}+a^{(s)}+\left(a^{(s)}\right)^{2}$

$\quad\quad$$\avx_{t}^{(s)}=\frac{1}{A_{t}^{(s)}}\left(A_{t-1}^{(s)}\avx_{t-1}^{(s)}+a^{(s)}x_{t}^{(s)}+\left(a^{(s)}\right)^{2}u^{(s-1)}\right)$

$\quad\quad$ \textbf{if $t\neq T_{s}$:}

$\quad\quad\quad$Pick $i_{t}^{(s)}\sim\mathrm{Uniform}\left(\left[n\right]\right)$

$\quad\quad\quad$$g_{t}^{(s)}=\nabla f_{i_{t}^{(s)}}(\avx_{t}^{(s)})-\nabla f_{i_{t}^{(s)}}(u^{(s-1)})+\nabla f(u^{(s-1)})$

$\quad\quad$ \textbf{else:}

$\quad\quad\quad$$g_{t}^{(s)}=\nabla f(\avx_{t}^{(s)})$

$\quad\quad$$\gamma_{t}^{(s)}=\frac{1}{\eta}\sqrt{\eta^{2}\left(\gamma_{t-1}^{(s)}\right)^{2}+\left(a^{(s)}\right)^{2}\left\Vert g_{t}^{(s)}-g_{t-1}^{(s)}\right\Vert ^{2}}$

$\quad\quad$$z_{t}^{(s)}=\arg\min_{z\in\dom}\left\{ a^{(s)}\left\langle g_{t}^{(s)},z\right\rangle +a^{(s)}h(z)+\frac{\gamma_{t-1}^{(s)}}{2}\left\Vert z-z_{t-1}^{(s)}\right\Vert ^{2}+\frac{\gamma_{t}^{(s)}-\gamma_{t-1}^{(s)}}{2}\left\Vert z-x_{t}^{(s)}\right\Vert ^{2}\right\} $

$\quad$$u^{(s)}=\avx_{0}^{(s+1)}=\avx_{T_{s}}^{(s)}$, $z_{0}^{(s+1)}=z_{T_{s}}^{(s)}$,
$g_{0}^{(s+1)}=g_{T_{s}}^{(s)}$, $\gamma_{0}^{(s+1)}=\gamma_{T_{s}}^{(s)}$

\textbf{return} $u^{(S)}$
\end{algorithm}

\begin{algorithm}
\caption{AdaVRAG\label{alg:AdaVRAG}}

\textbf{Input:} initial point $u^{(0)}$, domain diameter $D$.

\textbf{Parameters:} $\{a^{(s)}\}$, $a^{(s)}\in(0,1)$, $\{q^{(s)}\}$,
$\{T_{s}\}$, $\eta>0$.

$x_{0}^{(1)}=u^{(0)}$

Initialize $\gamma_{0}^{(1)}=\gamma$, where $\gamma$ is any small
constant

\textbf{for $s=1$ to $S$:}

$\quad$$\avx_{0}^{(s)}=a^{(s)}x_{0}^{(s)}+(1-a^{(s)})u^{(s-1)},$
compute $\nabla f(u^{(s-1)})$

\textbf{$\quad$for $t=1$ to $T_{s}$:}

$\quad\quad$Pick $i_{t}^{(s)}\sim\mathrm{Uniform}\left(\left[n\right]\right)$

$\quad\quad$$g_{t}^{(s)}=\nabla f_{i_{t}^{(s)}}(\avx_{t-1}^{(s)})-\nabla f_{i_{t}^{(s)}}(u^{(s-1)})+\nabla f(u^{(s-1)})$

$\quad\quad$$x_{t}^{(s)}=\arg\min_{x\in\dom}\left\{ \left\langle g_{t}^{(s)},x\right\rangle +h(x)+\frac{\gamma_{t-1}^{(s)}q^{(s)}}{2}\left\Vert x-x_{t-1}^{(s)}\right\Vert ^{2}\right\} $

$\quad\quad$$\avx_{t}^{(s)}=a^{(s)}x_{t}^{(s)}+(1-a^{(s)})u^{(s-1)}$

$\quad\quad$Option I: $\gamma_{t}^{(s)}=\gamma_{t-1}^{(s)}\sqrt{1+\frac{\left\Vert x_{t}^{(s)}-x_{t-1}^{(s)}\right\Vert ^{2}}{\eta^{2}}}$

$\quad\quad$Option II: $\gamma_{t}^{(s)}=\gamma_{t-1}^{(s)}+\frac{\left\Vert x_{t}^{(s)}-x_{t-1}^{(s)}\right\Vert ^{2}}{\eta^{2}}$

$\quad$$u^{(s)}=\frac{1}{T_{s}}\sum_{t=1}^{T_{s}}\avx_{t}^{(s)}$,
$x_{0}^{(s+1)}=x_{T_{s}}^{(s)}$ $,\gamma_{0}^{(s+1)}=\gamma_{T_{s}}^{(s)}$

\textbf{return} $u^{(S)}$
\end{algorithm}

In this section, we describe our algorithms and state their convergence
guarantees. Our algorithm AdaVRAE shown in Algorithm \ref{alg:AdaVRAE}
is a novel accelerated scheme that uses past extra-gradient update
steps in the inner loop and novel averaging to achieve acceleration.
AdaVRAE adaptively sets the step sizes based on the stochastic gradient
difference. Our choice of step sizes is a novel adaptation to the
VR setting of the step sizes used by the works \cite{MohriYang16,KavisLBC19,joulani2020simpler,ene2021variational}
in the batch/full-gradient setting. Our algorithm builds on the work
\cite{ene2021variational}, which provides an unaccelerated past extra-gradient
algorithm in the batch/full-gradient setting.

Theorem \ref{thm:AdaVRAE-main-body} states the parameter choices
 and the convergence guarantee for AdaVRAE, and we give its proof
in Section \ref{sec:AdaVRAE} in the appendix. The convergence rate
of AdaVRAE matches up to constant factors the rate of the state of
the art non-adaptive VR methods \cite{joulani2020simpler,song2020variance}.
The initial step size $\gamma_{0}^{(1)}$ can be set to any small
constant $\gamma$, which in practice we choose $\gamma=0.01$. Similarly
to AdaGrad, setting $\eta=\Theta(D)$ gives us the optimal dependence
of the convergence rate in the domain diameter. For simplicity, we
state the convergence in Theorem \ref{thm:AdaVRAE-main-body} and
\ref{thm:AdaVRAG-main-body} when $\eta=\Theta(D)$. We refer the
reader to Theorems \ref{thm:AdaVRAE-Convergence-A} and \ref{thm:AdaVRAG-Convergence-A}
in the appendix for the precise choice of parameters as well as the
full dependence of the convergence rate on arbitrary choices of $\gamma$
and $\eta$. In both Theorem \ref{thm:AdaVRAE-main-body} and \ref{thm:AdaVRAG-main-body},
we measure convergence using the number of individual gradient evaluations
$\nabla f_{i}$, assuming that the exact computation of $\nabla f$
takes $n$ gradient evaluations.
\begin{thm}
\label{thm:AdaVRAE-main-body}(Convergence of AdaVRAE) Define $s_{0}=\left\lceil \log_{2}\log_{2}4n\right\rceil $,
$c=\frac{3}{2}$. Suppose we set the parameters of Algorithm \ref{alg:AdaVRAE}
as follows: 
\begin{align*}
a^{(s)} & =\begin{cases}
(4n)^{-0.5^{s}} & 1\leq s\leq s_{0}\\
\frac{s-s_{0}-1+c}{2c} & s_{0}<s
\end{cases},\\
T_{s} & =n,\\
A_{T_{0}}^{(0)} & =\frac{5}{4}.
\end{align*}
Suppose that $\dom$ is a compact convex set with diameter $D$ and
we set $\eta=\Theta(D)$. The number of individual gradient evaluations
to achieve a solution $u^{(S)}$ such that $\E\left[F(u^{(S)})-F(x^{*})\right]\le\epsilon$
for Algorithm \ref{alg:AdaVRAE} is
\begin{align*}
\#grads & =\begin{cases}
\mathcal{O}\text{\ensuremath{\left(n\log\log\frac{V_{1}}{\epsilon}\right)}} & \mbox{if }\epsilon\geq\frac{V_{1}}{n}\\
\mathcal{O}\left(n\log\log n+\sqrt{\frac{nV_{1}}{\epsilon}}\right) & \mbox{if }\epsilon<\frac{V_{1}}{n}
\end{cases}
\end{align*}
where $V_{1}=\mathcal{O}\left(F(u^{(0)})-F(x^{*})+\left(\text{\ensuremath{\gamma}}+\beta\right)D^{2}\right)$.
\end{thm}
Our algorithm AdaVRAG is shown in Algorithm \ref{alg:AdaVRAG}. Compared
with AdaVRAE, AdaVRAG has a worse dependence on the smoothness parameter
$\beta$ but it performs only one projection onto $\dom$ in each
inner iteration. Additionally, as we discuss in more detail below,
it uses adaptive step sizes based on the iterate movement.

AdaVRAG follows a similar framework to existing VR methods such as
VARAG \cite{lan2019unified} and VRADA \cite{song2020variance}. Similarly
to VRADA, the algorithm achieves acceleration at the epoch level,
where an epoch is an iteration of the outer loop. The iterations in
an epoch update the main iterates via mirror descent with novel choices
of step sizes and coefficients. The stochastic gradient is computed
at a point that is a convex combination between the current iterate
and the checkpoint; the coefficients of this combination remain fixed
throughout the epoch. The step sizes are adaptively set based on the
iterate movement.

The structure of the inner iterations of our algorithm differs from
both VARAG and VRADA in several notable aspects. VARAG also uses mirror
descent to update the main iterates and it computes the stochastic
gradient at suitable combinations of the iterates and the checkpoint.
AdaVARAG uses a different averaging of the iterates to compute the
snapshots. Moreover, it uses a very different and simpler choice for
the coefficient used to combine the main iterates and the checkpoint
in order to obtain the points at which the stochastic gradients are
evaluated. In VARAG, this coefficient is set to a constant (namely,
$1/2$) in the initial iterations, whereas in AdaVRAG, it starts from
a small number and is increased gradually. This choice is critical
for improving the first term in the convergence from $\mathcal{O}(n\log n)$
to $\mathcal{O}(n\log\log n)$. In a similar manner, VRADA attains
the same convergence by a new choice of coefficient. However, this
is achieved via a very different approach based on dual-averaging. 

The step sizes used by AdaVRAG have two components: the step $\gamma_{t}^{(s)}$
that is updated based on the iterate movement and the per-epoch coefficient
$q^{(s)}$ to achieve acceleration at the epoch level. Our analysis
is flexible and allows the use of several approaches for updating
the steps $\gamma_{t}^{(s)}$. One approach, shown as option I in
Algorithm \ref{alg:AdaVRAG}, is based on the multiplicative update
rule of AdaGrad+ \cite{ene2021adaptive} which generalizes the AdaGrad
update to the constrained setting. We also propose a different variant,
shown as option II, that updates the steps in an additive manner.
Our analysis shows a similar convergence guarantee for both options,
with the main difference being in the dependence on the smoothness:
option I incurs a dependence of $\sqrt{\beta\log\beta}$, whereas
option II has a worse dependence of $\beta$. Option II achieved improved
performance in our experiments.

Theorem \ref{thm:AdaVRAG-main-body} states the parameter choices
and the convergence guarantee for AdaVRAG, and we give its proof in
Section \ref{sec:AdaVRAG} in the appendix. Analogously to AdaVRAE,
the initial step size $\gamma$ can be set to any small constant.
\begin{thm}
\label{thm:AdaVRAG-main-body}(Convergence of AdaVRAG) Define $s_{0}=\lceil\log_{2}\log_{2}4n\rceil$,
$c=\frac{3+\sqrt{33}}{4}$. Suppose we set the parameters of Algorithm
\ref{alg:AdaVRAG} as follows:
\begin{align*}
a^{(s)} & =\begin{cases}
1-\left(4n\right)^{-0.5^{s}} & 1\leq s\leq s_{0}\\
\frac{c}{s-s_{0}+2c} & s_{0}<s
\end{cases},\\
q^{(s)} & =\begin{cases}
\frac{1}{\left(1-a^{(s)}\right)a^{(s)}} & 1\leq s\leq s_{0}\\
\frac{8\left(2-a^{(s)}\right)a^{(s)}}{3(1-a^{(s)})} & s_{0}<s
\end{cases},\\
T_{s} & =n.
\end{align*}
Suppose that $\dom$ is a compact convex set with diameter $D$ and
we set $\eta=\Theta(D)$. Additionally, we assume that $2\eta^{2}>D^{2}$
if $\text{Option I}$ is used for setting the step size. The number
of individual gradient evaluations to achieve a solution $u^{(S)}$
such that $\E\left[F(u^{(S)})-F(x^{*})\right]\leq\epsilon$ for Algorithm
\ref{alg:AdaVRAG} is 
\[
\#grads=\begin{cases}
\mathcal{O}\left(n\log\log\frac{V_{2}}{\epsilon}\right) & \epsilon\geq\frac{V_{2}}{n}\\
\mathcal{O}\left(n\log\log n+\sqrt{\frac{nV_{2}}{\epsilon}}\right) & \epsilon<\frac{V_{2}}{n}
\end{cases},
\]
where
\[
V_{2}=\begin{cases}
\mathcal{O}\left(F(u^{(0)})-F(x^{*})+\left(\text{\ensuremath{\gamma}}+\beta\log\left(\frac{\beta}{\gamma}\right)\right)D^{2}\right) & \text{for Option I}\\
\mathcal{O}\left(F(u^{(0)})-F(x^{*})+\left(\text{\ensuremath{\gamma}}+\beta^{2}\right)D^{2}\right) & \text{for Option II}
\end{cases}.
\]
\end{thm}
\textbf{Comparison to AdaSVRG:} As noted in the introduction, the
state of the art adaptive VR method is the AdaSVRG algorithm \cite{dubois2021svrg},
which is a non-accelerated method. Both of our algorithms achieve
a faster convergence using different approaches and step sizes. AdaSVRG
resets the step sizes in each epoch, whereas our algorithms use a
cumulative update approach for the step sizes. In our experimental
evaluation, the resetting of the step sizes led to slower convergence.
AdaSVRG (multi-stage variant) uses varying epoch lengths similarly
to $\text{SVRG}^{++}$ \cite{allen2016improved}, whereas our algorithms
use epoch lengths that are set to $n$. Using an epoch of length $n$
allows for implementing the random sampling via a random permutation
of $[n]$ and is the preferred approach in practice.

Both our algorithms and AdaSVRG require that the domain $\dom$ has
bounded diameter. This is a restriction that is shared by almost all
existing adaptive methods. Recent work \cite{antonakopoulos2021adaptive,ene2021variational}
in the batch/full-gradient setting have proposed unaccelerated methods
that are suitable for unbounded domains, at a loss of additional factors
in the convergence. All of the existing accelerated methods require
that the domain is bounded, even in the batch/full-gradient setting.
We note that our analysis holds for arbitrary compact domains, whereas
the analysis of AdaSVRG only applies to domains that contain the global
optimum. Similarly to AdaGrad, both our algorithms and AdaSVRG can
be used in the unconstrained setting under the promise that the iterates
do not move too far from the optimum.

\textbf{Non-adaptive variants of our algorithms:} In the setting where
the smoothness parameter is known, we can set the step sizes of our
algorithms based on the smoothness, as shown in Algorithms \ref{alg:VRAE}
and \ref{alg:VRAG} (Sections \ref{sec:VRAE} and \ref{sec:VRAG}
in the appendix). Both algorithms match the convergence rates of the
state of the art VR methods \cite{joulani2020simpler,song2020variance}
using different algorithmic approaches based on mirror descent and
extra-gradient instead of dual-averaging. We experimentally compare
the non-adaptive algorithms to existing methods in Section \ref{sec:Additional-experiment-details}
of the appendix.

\subsection{Analysis outline}

We outline some of the key steps in the analysis of AdaVRAE. For the
purpose of simplicity, we assume $h=0$ and $\eta=D$. By building
on the standard analysis of the stochastic regret for extra-gradient
methods, we obtain the following result for the progress of one iteration:
\begin{align}
 & \E\left[\left(A_{t}^{(s)}-\left(a^{(s)}\right)^{2}\right)\left(f(\avx_{t}^{(s)})-f(x^{*})\right)-A_{t-1}^{(s)}\left(f(\avx_{t-1}^{(s)})-f(x^{*})\right)\right]\nonumber \\
 & \leq\E\left[\frac{\gamma_{t-1}^{(s)}}{2}\left\Vert z_{t-1}^{(s)}-x^{*}\right\Vert ^{2}-\frac{\gamma_{t}^{(s)}}{2}\left\Vert z_{t}^{(s)}-x^{*}\right\Vert ^{2}\right]\nonumber \\
 & +\E\left[\frac{\gamma_{t}^{(s)}-\gamma_{t-1}^{(s)}}{2}\left\Vert x_{t}^{(s)}-x^{*}\right\Vert ^{2}\right]\nonumber \\
 & +\E\left[\left(a^{(s)}\right)^{2}\left\langle \nabla f(\avx_{t}^{(s)}),u^{(s-1)}-\avx_{t}^{(s)}\right\rangle \right]\nonumber \\
 & +\E\left[\frac{\left(a^{(s)}\right)^{2}}{2\gamma_{t}^{(s)}}\left\Vert g_{t}^{(s)}-g_{t-1}^{(s)}\right\Vert ^{2}\right]\nonumber \\
 & -\underbrace{\E\left[\frac{A_{t-1}^{(s)}}{2\beta}\left\Vert \nabla f(\avx_{t}^{(s)})-\nabla f(\avx_{t-1}^{(s)})\right\Vert ^{2}\right]}_{\mbox{gain}}.\label{eq:iteration-progress}
\end{align}
In comparison to the standard analysis, the coefficient for the checkpoint
appears in the coefficient of $f(\avx_{t}^{(s)})-f(x^{*})$, which
becomes $\left(A_{t}^{(s)}-\left(a^{(s)}\right)^{2}\right)$ instead
of the usual $A_{t}^{(s)}$, making the sum not telescope immediately.
To resolve this, we first turn our attention to the analysis of the
stochastic gradient difference $\left\Vert g_{t}^{(s)}-g_{t-1}^{(s)}\right\Vert ^{2}$.
The key idea is to split $\frac{\left(a^{(s)}\right)^{2}}{2\gamma_{t}^{(s)}}\left\Vert g_{t}^{(s)}-g_{t-1}^{(s)}\right\Vert ^{2}$
into $\left(\frac{1}{2\gamma_{t}^{(s)}}-\frac{1}{16\beta}\right)\left(a^{(s)}\right)^{2}\left\Vert g_{t}^{(s)}-g_{t-1}^{(s)}\right\Vert ^{2}+\frac{\left(a^{(s)}\right)^{2}}{16\beta}\left\Vert g_{t}^{(s)}-g_{t-1}^{(s)}\right\Vert ^{2}$,
and bound each term in turn. For the first term, we build on the techniques
from prior work in the batch/full-gradient setting \cite{ene2021variational}.
For the second term, we use Young's inequality to write $\E\left[\left\Vert g_{t}^{(s)}-g_{t-1}^{(s)}\right\Vert ^{2}\right]\leq\E\left[4\left\Vert \nabla f(\avx_{t}^{(s)})-g_{t}^{(s)}\right\Vert ^{2}+4\left\Vert \nabla f(\avx_{t-1}^{(s)})-g_{t-1}^{(s)}\right\Vert ^{2}\right]+\E\left[2\left\Vert \nabla f(\avx_{t}^{(s)})-\nabla f(\avx_{t-1}^{(s)})\right\Vert ^{2}\right]$.
The gradient difference loss term is cancelled by the gain term in
(\ref{eq:iteration-progress}), and thus we can focus on the first
two variance terms. We apply the usual variance reduction technique
put forward by \cite{lan2019unified} (see Lemma \ref{lem:Variance-Reduction})
to bound the two variance terms, as follows:
\begin{align*}
\E\left[\left\Vert g_{t}^{(s)}-\nabla f(\avx_{t}^{(s)})\right\Vert ^{2}\right]\le\E\left[2\beta\left(f(u^{(s-1)})-f(\avx_{t}^{(s)})-\left\langle \nabla f(\avx_{t}^{(s)}),u^{(s-1)}-\avx_{t}^{(s)}\right\rangle \right)\right].
\end{align*}
Thus we obtain an upper bound on $\frac{\left(a^{(s)}\right)^{2}}{16\beta}\left\Vert g_{t}^{(s)}-g_{t-1}^{(s)}\right\Vert ^{2}$
in terms of $\left(a^{(s)}\right)^{2}\left(f(u^{(s-1)})-f(\avx_{t}^{(s)})\right)$.
This is the reason for setting the coefficient for the checkpoint
to $\left(a^{(s)}\right)^{2}$, so that the LHS of (\ref{eq:iteration-progress})
can become the usual telescoping sum $A_{t}^{(s)}\left(f(\avx_{t}^{(s)})-f(x^{*})\right)-A_{t-1}^{(s)}\left(f(\avx_{t-1}^{(s)})-f(x^{*})\right)$.
Using the convexity of $f$, we obtain the following key result for
the progress of each epoch:
\begin{align*}
 & \E\left[A_{T_{s}}^{(s)}\left(f(\avx_{T_{s}}^{(s)})-f(x^{*})\right)-A_{0}^{(s)}\left(f(\avx_{0}^{(s)})-f(x^{*})\right)\right]\\
 & \leq\E\left[\frac{\gamma_{0}^{(s)}}{2}\left\Vert z_{0}^{(s)}-x^{*}\right\Vert ^{2}-\frac{\gamma_{T_{s}}^{(s)}}{2}\left\Vert z_{T_{s}}^{(s)}-x^{*}\right\Vert ^{2}\right]\\
 & +\E\left[\sum_{t=1}^{T_{s}}\frac{\gamma_{t}^{(s)}-\gamma_{t-1}^{(s)}}{2}\left\Vert x_{t}^{(s)}-x^{*}\right\Vert ^{2}\right]\\
 & +\E\left[T_{s}\left(a^{(s)}\right)^{2}\left(f(u^{(s-1)})-f(x^{*})\right)\right]\\
 & +\E\left[\sum_{t=1}^{T_{s}}\left(\frac{1}{2\gamma_{t}^{(s)}}-\frac{1}{16\beta}\right)\left(a^{(s)}\right)^{2}\left\Vert g_{t}^{(s)}-g_{t-1}^{(s)}\right\Vert ^{2}\right].
\end{align*}
Intuitively, we want to have another telescoping sum when summing
up the above inequality across all epochs $s$. To do so, we can set
the starting points of the next epoch to be the ending points of the
previous one, i.e., $\avx_{T_{s}}^{(s)}=\avx_{0}^{(s+1)}=u^{(s)}$,
$\gamma_{T_{s}}^{(s)}=\gamma_{0}^{(s+1)}$, $z_{T_{s}}^{(s)}=z_{0}^{(s+1)}$.
However, an extra term $T_{s}\left(a^{(s)}\right)^{2}\left(f(u^{(s-1)})-f(x^{*})\right)$
appears on the RHS. We need to reset the new starting coefficient
in the new epoch $A_{0}^{(s)}$ to $A_{T_{s-1}}^{(s-1)}-T_{s}\left(a^{(s)}\right)^{2}$
so that we can telescope the LHS.

To bound the term $\sum_{s=1}^{S}\sum_{t=1}^{T_{s}}\frac{\gamma_{t}^{(s)}-\gamma_{t-1}^{(s)}}{2}\left\Vert x_{t}^{(s)}-x^{*}\right\Vert ^{2}+\left(\frac{1}{2\gamma_{t}^{(s)}}-\frac{1}{16\beta}\right)\left(a^{(s)}\right)^{2}\left\Vert g_{t}^{(s)}-g_{t-1}^{(s)}\right\Vert ^{2}$,
since $\gamma_{T_{s}}^{(s)}=\gamma_{0}^{(s+1)}$ and $g_{T_{s}}^{(s)}=g_{0}^{(s+1)}$,
we can consider the doubly indexed sequences $\left(\gamma_{t}^{(s)}\right)$
and $\left(g_{t}^{(s)}\right)$ as two singly indexed sequences $\left(\gamma_{k}\right)$
and $\left(g_{k}\right)$ and the coefficient $a^{(s)}$ to be another
sequence $\left(a_{k}\right)$. Then we can employ the following two
inequalities:
\begin{align*}
\frac{D^{2}}{2}\left(\gamma_{K}-\gamma_{0}\right)-\frac{1}{48\beta}\sum_{k=1}^{K}a_{k}^{2}\left\Vert g_{k}-g_{k-1}\right\Vert ^{2} & \le12\beta D^{2}\\
\sum_{k=1}^{K}\text{\ensuremath{\left(\frac{1}{2\gamma_{k}}-\frac{1}{24\beta}\right)}}a_{k}^{2}\left\Vert g_{k}-g_{k-1}\right\Vert ^{2} & \le12\beta D^{2}
\end{align*}
Finally, we need to choose the parameters $a^{(s)}$ so that the conditions
needed for our analysis are satisfied and $A_{T_{s}}^{(s)}$ is sufficiently
large, so that we attain a fast convergence. We have to choose $a^{(s)}$
such that $\left(a^{(s)}\right)^{2}\leq4A_{t-1}^{(s)}$ for all $s,t\ge1$
and that $A_{0}^{(s)}=A_{T_{s-1}}^{(s-1)}-T_{s}\left(a^{(s)}\right)^{2}\ge0$.
The main idea is to divide the epochs into two phases: in the first
phase, $A_{T_{s}}^{(s)}$ quickly rises to $\Omega(n)$ and in the
second phase, to achieve the optimal $\sqrt{\frac{n\beta}{\epsilon}}$
rate, $A_{T_{s}}^{(s)}=\Omega(n^{2})$. The nearly-optimal choice
of $a^{(s)}$ in the first phase is $(4n)^{-0.5^{s}}$, stopping at
$s=s_{0}=\left\lceil \log_{2}\log_{2}4n\right\rceil $, while in the
second phase, we have to be more conservative and choose $a^{(s)}=\frac{s-s_{0}+\frac{1}{2}}{3}$.
With this we can obtain the convergence rate of $\mathcal{O}\left(n\min\left\{ \log\log\frac{\beta}{\epsilon},\log\log n\right\} +\sqrt{\frac{n\beta}{\epsilon}}\right)$.

\section{Experiments}
\label{sec:experiments}

\begin{figure*}
\subfloat[Logistic loss]{\includegraphics[width=0.33\textwidth]{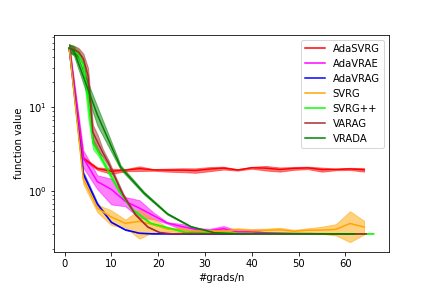}}\subfloat[Squared loss]{\includegraphics[width=0.33\textwidth]{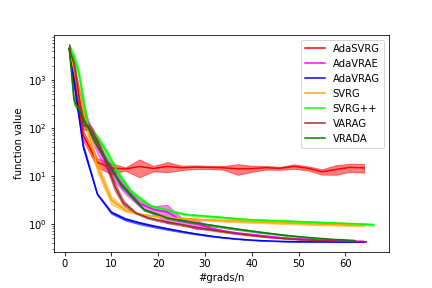}}\subfloat[Huber loss]{\includegraphics[width=0.33\textwidth]{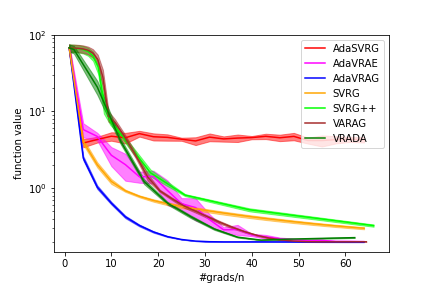}}\hfill{}\caption{a1a}

\label{fig:experimental-results-a1a}
\end{figure*}
\begin{figure*}
\subfloat[Logistic loss]{\includegraphics[width=0.33\textwidth]{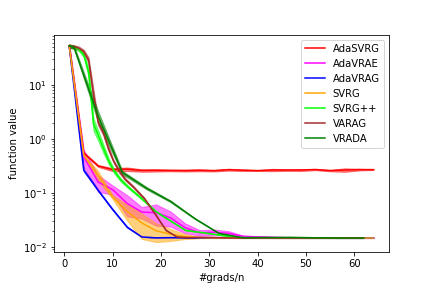}}\subfloat[Squared loss]{\includegraphics[width=0.33\textwidth]{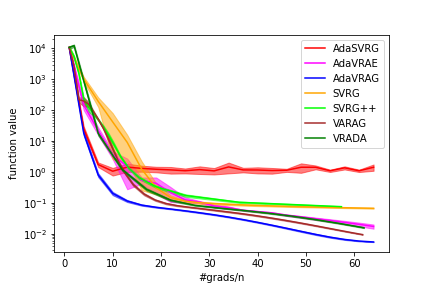}}\subfloat[Huber loss]{\includegraphics[width=0.33\textwidth]{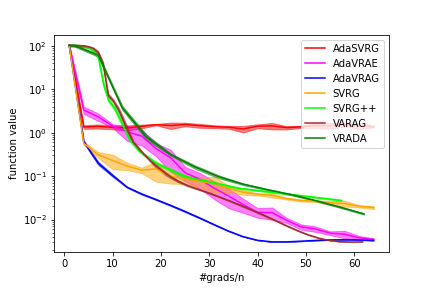}}\hfill{}\caption{mushrooms}

\label{fig:experimental-results-mushrooms}
\end{figure*}
\begin{figure*}
\subfloat[Logistic loss]{\includegraphics[width=0.33\textwidth]{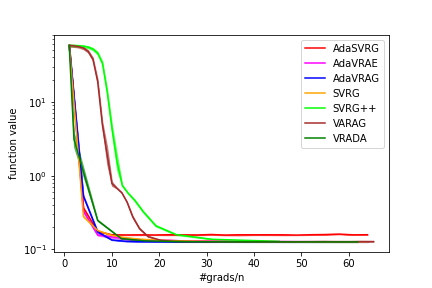}}\subfloat[Squared loss]{\includegraphics[width=0.33\textwidth]{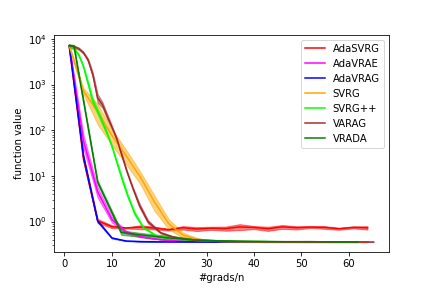}}\subfloat[Huber loss]{\includegraphics[width=0.33\textwidth]{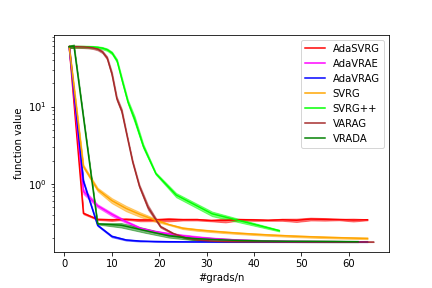}}\hfill{}\caption{w8a}

\label{fig:experimental-results-w8a}
\end{figure*}
\begin{figure*}
\subfloat[Logistic loss]{\includegraphics[width=0.33\textwidth]{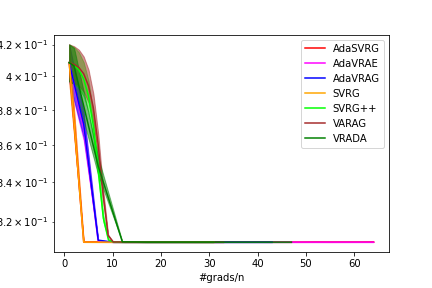}}\subfloat[Squared loss]{\includegraphics[width=0.33\textwidth]{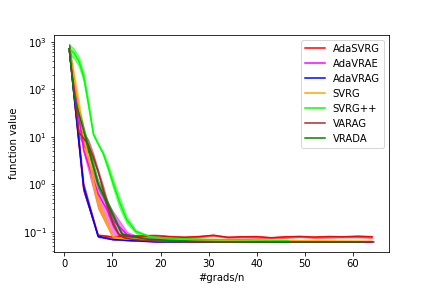}}\subfloat[Huber loss]{\includegraphics[width=0.33\textwidth]{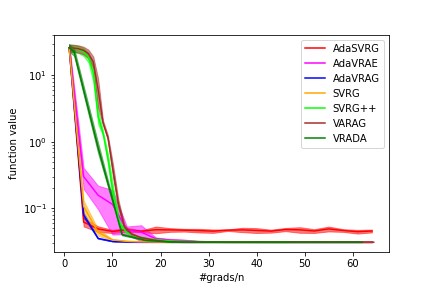}}\hfill{}\caption{phishing}

\label{fig:experimental-results-phishing}
\end{figure*}

In this section we demonstrate the performances of AdaVRAG and AdaVRAE
in comparison with the existing standard and adaptive VR methods.
We use the experimental setup and the code base of \cite{dubois2021svrg}\footnote{Their code can be found at https://github.com/bpauld/AdaSVRG}.

\textbf{Datasets and loss functions:} We experiment with binary classification
on four standard LIBSVM datasets: a1a, mushrooms, w8a and phishing
\cite{chang2011libsvm}. For each dataset, we show the results for
three different objective functions: logistic, squared and huber loss.
Following the setting in \cite{dubois2021svrg} we add a $\ell_{2}$-regularization
term to the loss function, with regularization set to $1/n$. 

\textbf{Constraint:} In all experiments, we evaluate the algorithms
under a ball constraint. That is, the domain of each problem in our
experiment is a ball of radius $R=100$ around the initial point,
which means for every algorithm, in the update step, we need to do
a projection onto this ball.

\textbf{Algorithms and hyperparameter selection:} We compare AdaVRAE
and AdaVRAG with the common VR algorithms: SVRG \cite{johnson2013accelerating},
$\text{SVRG}^{++}$\cite{allen2016improved}, VARAG \cite{lan2019unified},
VRADA \cite{song2020variance}, and AdaSVRG \cite{dubois2021svrg}
(in the experiment the multi-stage variant performs worse than the
fixed-sized inner loop variant, and we omit it from the plots). Among
these, only AdaSVRG is an adaptive VR method, which does not require
parameter tuning. For the non-adaptive methods we chose the step size
(or equivalently, the inverse of the smoothness parameter ($1/\beta$)
for VRADA) via hyperparameter search over $\left\{ 0.01,0.05,0.1,0.5,1,5,10,100\right\} $.
For each experiment, we used the choice that led to the best performance,
and we report the parameters used in Table \ref{tab:Hyperparameters}.
The adaptive methods --- AdaSVRG, AdaVRAE, AdaVRAG --- do not require
any hyperparameter tuning and we set their parameters as prescribed
by the theoretical analysis. For AdaSVRG, we used $\eta=D/\sqrt{2}=\sqrt{2}R$
as recommended in the original paper. For AdaVRAE and AdaVRAG, we
used $\gamma=0.01$ and $\eta=D/2=R$.

\textbf{Implementation and initialization:} For all algorithms, in
the inner loop, we use a random permutation to select a function.
We also fix the batch size to $1$ in all cases to match the theoretical
setting. We initialize $u^{(0)}$ to be a random point in $[0,10]^{d}$
where each dimension is uniformly chosen in $[0,10]$. Each experiment
is repeated five times with different initial point, which is kept
the same across all algorithms. 

\textbf{Results:} The results are shown in Figures \ref{fig:experimental-results-a1a},
\ref{fig:experimental-results-mushrooms}, \ref{fig:experimental-results-w8a},
\ref{fig:experimental-results-phishing}. For each experiment, we
plot the mean value and 95\% confidence interval of the training objective
against the number of gradient evaluations normalized by the number
of examples.

\textbf{Discussion: }We observe that, in all experiments, AdaVRAG
consistently performs competitively with all methods and generally
have the best performances. The non-accelerated methods in general
converge more slowly compared with accelerated methods, especially
in the later epochs. In some cases, VARAG suffers from a slow convergence
rate in the first phase. This is possibly due to the fact that it
sets to $1/2$ the coefficient for the checkpoint in the first phase.
VRADA sometimes exhibits similar behavior but to a lesser extent.
In AdaVRAG and AdaVRAE, the coefficient for the checkpoint is set
to be small in the beginning and gradually increased over time when
the quality of the checkpoint is improved. The other adaptive method,
AdaSVRG, exhibits slow convergence in many cases. One reason might
be that AdaSVRG resets the step size in every epoch and, in later
epochs, the step size may be too large for the algorithm to converge.
In contrast, AdaVRAG and AdaVRAE use cumulative step sizes.

\bibliographystyle{plain}
\bibliography{ref}

\appendix

\section{Analysis of algorithm \ref{alg:AdaVRAE}}

\label{sec:AdaVRAE}

In this section, we analyze Algorithm \ref{alg:AdaVRAE} and prove
the following convergence guarantee:
\begin{thm}[Convergence of AdaVRAE]
\label{thm:AdaVRAE-Convergence-A} Define $s_{0}=\left\lceil \log_{2}\log_{2}4n\right\rceil $,
$c=\frac{3}{2}$. If we choose parameters as follows
\begin{align*}
a^{(s)} & =\begin{cases}
(4n)^{-0.5^{s}} & 1\leq s\leq s_{0}\\
\frac{s-s_{0}-1+c}{2c} & s_{0}<s
\end{cases},\\
T_{s} & =n,\\
A_{T_{0}}^{(0)} & =\frac{5}{4}.
\end{align*}
Assuming $\dom$ is a compact convex set with diameter $D$, the number
of individual gradient evaluations to achieve a solution $u^{(S)}$
such that $\E\left[F(u^{(S)})-F(x^{*})\right]\le\epsilon$ for Algorithm
\ref{alg:AdaVRAE} is

\begin{align*}
\#grads & =\begin{cases}
\mathcal{O}\text{\ensuremath{\left(n\log\log\frac{V}{\epsilon}\right)}} & \mbox{if }\epsilon\geq\frac{V}{n}\\
\mathcal{O}\left(n\log\log n+\sqrt{\frac{Vn}{\epsilon}}\right) & \mbox{if }\epsilon<\frac{V}{n}
\end{cases}
\end{align*}

where $V=\frac{5}{2}\left(F(u^{(0)})-F(x^{*})\right)+\text{\ensuremath{\gamma}}\left\Vert u^{(0)}-x^{*}\right\Vert ^{2}+\frac{16\beta\left(D^{4}+2\eta^{4}\right)}{\eta^{2}}$.
\end{thm}
To start with, we state and prove the following variance reduction
lemma commonly used in accelerated methods:
\begin{lem}
\label{lem:Variance-Reduction}(Variance Reduction) Let $i\sim\mathrm{Uniform}([n])$
and $g=\nabla f_{i}(x)-\nabla f_{i}(u)+\nabla f(u)$ be an estimate
of the gradient of $f$ at $x$. We have
\[
\E_{i}\left[\left\Vert g-\nabla f(x)\right\Vert ^{2}\right]\leq2\beta\left(f(u)-f(x)-\left\langle \nabla f(x),u-x\right\rangle \right).
\]
\end{lem}
\begin{proof}
By the definition of $g$,
\begin{align*}
\E_{i}\left[\left\Vert g-\nabla f(x)\right\Vert ^{2}\right] & =\E_{i}\left[\left\Vert \nabla f_{i}(x)-\nabla f_{i}(u)+\nabla f(u)-\nabla f(x)\right\Vert ^{2}\right]\\
 & \overset{(a)}{\leq}\E_{i}\left[\left\Vert \nabla f_{i}(u)-\nabla f_{i}(x)\right\Vert ^{2}\right]\\
 & \overset{(b)}{\leq}\E_{i}\left[2\beta\left(f_{i}(u)-f_{i}(x)-\left\langle \nabla f_{i}(x),u-x\right\rangle \right)\right]\\
 & \overset{(c)}{=}2\beta\left(f(u)-f(x)-\left\langle \nabla f(x),u-x\right\rangle \right),
\end{align*}
where $(a)$ is because $\E_{i}\left[\nabla f_{i}(u)-\nabla f_{i}(x)\right]=\nabla f(u)-\nabla f(x)$
and $\E\left[\left\Vert X-\E\left[X\right]\right\Vert ^{2}\right]\leq\E\left[\left\Vert X\right\Vert ^{2}\right]$,
$(b)$ is by the convexity and $\beta$-smoothness of $f_{i}$, $(c)$
is by $i\sim\mathrm{Uniform}([n])$ and the definition of $f$.
\end{proof}

\subsection{Single iteration progress}

We first analyze the progress in function value made in a single iteration
of an epoch. The analysis follows the standard method as in \cite{ene2021variational};
however, we need to pay attention to the extra term for the checkpoint
that appears in the convex combination for $\overline{x}_{t}^{(s)}$.
We start off by the following observation
\begin{lem}
\label{lem:AdaVRAE-Iterate-Diff}For any $s\ge1$ and $t\in\left[T_{s}\right]$,
\[
\overline{x}_{t}^{(s)}-\avx_{t-1}^{(s)}=\frac{a^{(s)}}{A_{t-1}^{(s)}}\left(x_{t}^{(s)}-\overline{x}_{t}^{(s)}\right)+\frac{\left(a^{(s)}\right)^{2}}{A_{t-1}^{(s)}}\left(u^{(s-1)}-\overline{x}_{t}^{(s)}\right).
\]
\end{lem}
\begin{proof}
We note that the definition $\avx_{t}^{(s)}=\frac{1}{A_{t}^{(s)}}\left(A_{t-1}^{(s)}\avx_{t-1}^{(s)}+a^{(s)}x_{t}^{(s)}+\left(a^{(s)}\right)^{2}u^{(s-1)}\right)$
implies

\begin{align*}
A_{t}^{(s)}\overline{x}_{t}^{(s)} & =A_{t-1}^{(s)}\avx_{t-1}^{(s)}+a^{(s)}x_{t}^{(s)}+\left(a^{(s)}\right)^{2}u^{(s-1)}\\
\overset{(a)}{\Leftrightarrow}A_{t-1}^{(s)}(\avx_{t}^{(s)}-\avx_{t-1}^{(s)}) & =a^{(s)}(x_{t}^{(s)}-\overline{x}_{t}^{(s)})+\left(a^{(s)}\right)^{2}(u^{(s-1)}-\overline{x}_{t}^{(s)})\\
\Leftrightarrow\avx_{t}^{(s)}-\avx_{t-1}^{(s)} & =\frac{a^{(s)}}{A_{t-1}^{(s)}}\left(x_{t}^{(s)}-\overline{x}_{t}^{(s)}\right)+\frac{\left(a^{(s)}\right)^{2}}{A_{t-1}^{(s)}}\left(u^{(s-1)}-\overline{x}_{t}^{(s)}\right),
\end{align*}
where $(a)$ is by $A_{t}^{(s)}=A_{t-1}^{(s)}+a^{(s)}+\left(a^{(s)}\right)^{2}$.
\end{proof}

Next, we bound the function progress in a single epoch via the stochastic
regret. Note that, this lemma is somewhat weaker than we would desire,
due to the appearance the coefficient of the checkpoint, making the
LHS not immediately telescope. We will account for this factor later
in the analysis.
\begin{lem}
\label{lem:AdaVRAE-One-Step-Descent-I}For all epochs $s\geq1$ and
all iterations $t\in\left[T_{s}\right]$
\begin{align*}
 & \E\left[\left(A_{t}^{(s)}-\left(a^{(s)}\right)^{2}\right)\left(F(\avx_{t}^{(s)})-F(x^{*})\right)-A_{t-1}^{(s)}\left(F(\avx_{t-1}^{(s)})-F(x^{*})\right)\right]\\
\leq & \E\left[a^{(s)}\underbrace{\left\langle g_{t}^{(s)},x_{t}^{(s)}-x^{*}\right\rangle }_{\text{stochastic regret}}+\left(a^{(s)}\right)^{2}\left\langle \nabla f(\avx_{t}^{(s)}),u^{(s-1)}-\avx_{t}^{(s)}\right\rangle \right]\\
 & -\E\left[\frac{A_{t-1}^{(s)}}{2\beta}\left\Vert \nabla f(\avx_{t}^{(s)})-\nabla f(\avx_{t-1}^{(s)})\right\Vert ^{2}\right]\\
 & +\E\left[\left(A_{t}^{(s)}-\left(a^{(s)}\right)^{2}\right)\left(h(\avx_{t}^{(s)})-h(x^{*})\right)-A_{t-1}^{(s)}\left(h(\avx_{t-1}^{(s)})-h(x^{*})\right)\right].
\end{align*}
\end{lem}
\begin{proof}
Using the observation in Lemma \ref{lem:AdaVRAE-Iterate-Diff}, we
have
\begin{align*}
 & F(\avx_{t}^{(s)})-F(\avx_{t-1}^{(s)})\\
= & f(\avx_{t}^{(s)})-f(\avx_{t-1}^{(s)})+h(\avx_{t}^{(s)})-h(\avx_{t-1}^{(s)})\\
\stackrel{(a)}{\le} & \left\langle \nabla f(\avx_{t}^{(s)}),\avx_{t}^{(s)}-\avx_{t-1}^{(s)}\right\rangle -\frac{1}{2\beta}\left\Vert \nabla f(\avx_{t}^{(s)})-\nabla f(\avx_{t-1}^{(s)})\right\Vert ^{2}+h(\avx_{t}^{(s)})-h(\avx_{t-1}^{(s)})\\
\stackrel{(b)}{=} & \frac{a^{(s)}}{A_{t-1}}\left\langle \nabla f(\avx_{t}^{(s)}),x_{t}^{(s)}-\avx_{t}^{(s)}\right\rangle +\frac{\left(a^{(s)}\right)^{2}}{A_{t-1}^{(s)}}\left\langle \nabla f(\avx_{t}^{(s)}),u^{(s-1)}-\avx_{t}^{(s)}\right\rangle \\
 & -\frac{1}{2\beta}\left\Vert \nabla f(\avx_{t}^{(s)})-\nabla f(\avx_{t-1}^{(s)})\right\Vert ^{2}+h(\avx_{t}^{(s)})-h(\avx_{t-1}^{(s)})
\end{align*}
where $(a)$ is due to the smoothness of $f$ and $(b)$ comes from
Lemma \ref{lem:AdaVRAE-Iterate-Diff}. By the convexity of $f$, we
also have
\begin{align*}
 & F(\avx_{t}^{(s)})-F(x^{*})\\
= & f(\avx_{t}^{(s)})-f(x^{*})+h(\avx_{t}^{(s)})-h(x^{*})\\
\leq & \left\langle \nabla f(\avx_{t}^{(s)}),\avx_{t}^{(s)}-x^{*}\right\rangle +h(\avx_{t}^{(s)})-h(x^{*})
\end{align*}
We combine the two inequalities and obtain
\begin{align*}
 & A_{t-1}^{(s)}\left(F(\avx_{t}^{(s)})-F(\avx_{t-1}^{(s)})\right)+a^{(s)}\left(F(\avx_{t}^{(s)})-F(x^{*})\right)\\
\leq & a^{(s)}\left\langle \nabla f(\avx_{t}^{(s)}),x_{t}^{(s)}-x^{*}\right\rangle +\left(a^{(s)}\right)^{2}\left\langle \nabla f(\avx_{t}^{(s)}),u^{(s-1)}-\avx_{t}^{(s)}\right\rangle \\
 & -\frac{A_{t-1}}{2\beta}\left\Vert \nabla f(\avx_{t}^{(s)})-\nabla f(\avx_{t-1}^{(s)})\right\Vert ^{2}\\
 & +A_{t-1}^{(s)}\left(h(\avx_{t}^{(s)})-h(\avx_{t-1}^{(s)})\right)+a^{(s)}\left(h(\avx_{t}^{(s)})-h(x^{*})\right)\\
= & a^{(s)}\left\langle g_{t}^{(s)},x_{t}^{(s)}-x^{*}\right\rangle +\left(a^{(s)}\right)^{2}\left\langle \nabla f(\avx_{t}^{(s)}),u^{(s-1)}-\avx_{t}^{(s)}\right\rangle \\
 & +a^{(s)}\left\langle \nabla f(\avx_{t}^{(s)})-g_{t}^{(s)},x_{t}^{(s)}-x^{*}\right\rangle -\frac{A_{t-1}^{(s)}}{2\beta}\left\Vert \nabla f(\avx_{t}^{(s)})-\nabla f(\avx_{t-1}^{(s)})\right\Vert ^{2}\\
 & +A_{t-1}^{(s)}\left(h(\avx_{t}^{(s)})-h(\avx_{t-1}^{(s)})\right)+a^{(s)}\left(h(\avx_{t}^{(s)})-h(x^{*})\right).
\end{align*}
Note that we can rearrange the terms
\begin{align*}
 & A_{t-1}^{(s)}\left(F(\avx_{t}^{(s)})-F(\avx_{t-1}^{(s)})\right)+a^{(s)}\left(F(\avx_{t}^{(s)})-F(x^{*})\right)\\
= & \left(A_{t}^{(s)}-\left(a^{(s)}\right)^{2}\right)\left(F(\avx_{t}^{(s)})-F(x^{*})\right)-A_{t-1}^{(s)}\left(F(\avx_{t-1}^{(s)})-F(x^{*})\right),\\
 & A_{t-1}^{(s)}\left(h(\avx_{t}^{(s)})-h(\avx_{t-1}^{(s)})\right)+a^{(s)}\left(h(\avx_{t}^{(s)})-h(x^{*})\right)\\
= & \left(A_{t}^{(s)}-\left(a^{(s)}\right)^{2}\right)\left(h(\avx_{t}^{(s)})-h(x^{*})\right)-A_{t-1}^{(s)}\left(h(\avx_{t-1}^{(s)})-h(x^{*})\right).
\end{align*}
Thus we obtain
\begin{align}
 & \left(A_{t}^{(s)}-\left(a^{(s)}\right)^{2}\right)\left(F(\avx_{t}^{(s)})-F(x^{*})\right)-A_{t-1}^{(s)}\left(F(\avx_{t-1}^{(s)})-F(x^{*})\right)\nonumber \\
\le & a^{(s)}\left\langle g_{t}^{(s)},x_{t}^{(s)}-x^{*}\right\rangle +\left(a^{(s)}\right)^{2}\left\langle \nabla f(\avx_{t}^{(s)}),u^{(s-1)}-\avx_{t}^{(s)}\right\rangle \nonumber \\
 & +a^{(s)}\left\langle \nabla f(\avx_{t}^{(s)})-g_{t}^{(s)},x_{t}^{(s)}-x^{*}\right\rangle -\frac{A_{t-1}^{(s)}}{2\beta}\left\Vert \nabla f(\avx_{t}^{(s)})-\nabla f(\avx_{t-1}^{(s)})\right\Vert ^{2}\nonumber \\
 & +\left(A_{t}^{(s)}-\left(a^{(s)}\right)^{2}\right)\left(h(\avx_{t}^{(s)})-h(x^{*})\right)-A_{t-1}^{(s)}\left(h(\avx_{t-1}^{(s)})-h(x^{*})\right).\label{eq:AdaVRAE-1}
\end{align}
Observe that for $t<T_{s}$
\begin{align*}
\E\left[a^{(s)}\left\langle \nabla f(\avx_{t}^{(s)})-g_{t}^{(s)},x_{t}^{(s)}-x^{*}\right\rangle \right] & =\E\left[\E_{i_{t}^{(s)}}\left[a^{(s)}\left\langle \nabla f(\avx_{t}^{(s)})-g_{t}^{(s)},x_{t}^{(s)}-x^{*}\right\rangle \right]\right]\\
 & =0.
\end{align*}
and for $t=T_{s},$we have $\nabla f(\avx_{t}^{(s)})=g_{t}^{(s)}$
thus $\E\left[a^{(s)}\left\langle \nabla f(\avx_{t}^{(s)})-g_{t}^{(s)},x_{t}^{(s)}-x^{*}\right\rangle \right]=0$.
By taking expectations w.r.t. both sides of (\ref{eq:AdaVRAE-1}),
we get
\begin{align*}
 & \E\left[\left(A_{t}^{(s)}-\left(a^{(s)}\right)^{2}\right)\left(F(\avx_{t}^{(s)})-F(x^{*})\right)-A_{t-1}^{(s)}\left(F(\avx_{t-1}^{(s)})-F(x^{*})\right)\right]\\
\le & \E\left[a^{(s)}\left\langle g_{t}^{(s)},x_{t}^{(s)}-x^{*}\right\rangle +\left(a^{(s)}\right)^{2}\left\langle \nabla f(\avx_{t}^{(s)}),u^{(s-1)}-\avx_{t}^{(s)}\right\rangle \right]\\
 & -\E\left[\frac{A_{t-1}^{(s)}}{2\beta}\left\Vert \nabla f(\avx_{t}^{(s)})-\nabla f(\avx_{t-1}^{(s)})\right\Vert ^{2}\right]\\
 & +\E\left[\left(A_{t}^{(s)}-\left(a^{(s)}\right)^{2}\right)\left(h(\avx_{t}^{(s)})-h(x^{*})\right)-A_{t-1}^{(s)}\left(h(\avx_{t-1}^{(s)})-h(x^{*})\right)\right].
\end{align*}
\end{proof}

To analyze the stochastic regret, we split the inner product as follows

\begin{align*}
\left\langle g_{t}^{(s)},x_{t}^{(s)}-x^{*}\right\rangle  & =\left\langle g_{t}^{(s)},z_{t}^{(s)}-x^{*}\right\rangle +\left\langle g_{t}^{(s)}-g_{t-1}^{(s)},x_{t}^{(s)}-z_{t}^{(s)}\right\rangle +\left\langle g_{t-1}^{(s)},x_{t}^{(s)}-z_{t}^{(s)}\right\rangle .
\end{align*}
For each term we give a bound as stated in Lemma \ref{lem:AdaVRAE-Stochastic-Regret-Bound}.
\begin{lem}
\label{lem:AdaVRAE-Stochastic-Regret-Bound}For any $s\ge1$ all iterations
$t\in\left[T_{s}\right]$, we have

\begin{align*}
a^{(s)}\left\langle g_{t-1}^{(s)},x_{t}^{(s)}-z_{t}^{(s)}\right\rangle  & \leq\frac{\gamma_{t-1}^{(s)}}{2}\left\Vert z_{t-1}^{(s)}-z_{t}^{(s)}\right\Vert ^{2}-\frac{\gamma_{t-1}^{(s)}}{2}\left\Vert z_{t-1}^{(s)}-x_{t}^{(s)}\right\Vert ^{2}-\frac{\gamma_{t-1}^{(s)}}{2}\left\Vert x_{t}^{(s)}-z_{t}^{(s)}\right\Vert ^{2}\\
 & \quad+a^{(s)}\left(h(z_{t}^{(s)})-h(x_{t}^{(s)})\right).
\end{align*}
\begin{align*}
a^{(s)}\left\langle g_{t}^{(s)},z_{t}^{(s)}-x^{*}\right\rangle  & \leq\frac{\gamma_{t}^{(s)}-\gamma_{t-1}^{(s)}}{2}\left\Vert x_{t}^{(s)}-x^{*}\right\Vert ^{2}+\frac{\gamma_{t-1}^{(s)}}{2}\left\Vert z_{t-1}^{(s)}-x^{*}\right\Vert ^{2}-\frac{\gamma_{t}^{(s)}}{2}\left\Vert z_{t}^{(s)}-x^{*}\right\Vert ^{2}\\
 & \quad-\frac{\gamma_{t-1}^{(s)}}{2}\left\Vert z_{t}^{(s)}-z_{t-1}^{(s)}\right\Vert ^{2}-\frac{\gamma_{t}^{(s)}-\gamma_{t-1}^{(s)}}{2}\left\Vert x_{t}^{(s)}-z_{t}^{(s)}\right\Vert ^{2}\\
 & \quad+a^{(s)}\left(h(x^{*})-h(z_{t}^{(s)})\right).
\end{align*}
\begin{align*}
a^{(s)}\left\langle g_{t}^{(s)}-g_{t-1}^{(s)},x_{t}^{(s)}-z_{t}^{(s)}\right\rangle  & \le\frac{\left(a^{(s)}\right)^{2}}{2\gamma_{t}^{(s)}}\left\Vert g_{t}^{(s)}-g_{t-1}^{(s)}\right\Vert ^{2}+\frac{\gamma_{t}^{(s)}}{2}\left\Vert x_{t}^{(s)}-z_{t}^{(s)}\right\Vert ^{2}.
\end{align*}
\end{lem}
\begin{proof}
Since $x_{t}^{(s)}=\arg\min_{x\in\dom}\left\{ a^{(s)}\left\langle g_{t-1}^{(s)},x\right\rangle +a^{(s)}h(x)+\frac{\gamma_{t-1}^{(s)}}{2}\left\Vert x-z_{t-1}^{(s)}\right\Vert ^{2}\right\} $,
by the optimality condition of $x_{t}^{(s)}$, we have
\begin{align*}
\left\langle a^{(s)}g_{t-1}^{(s)}+a^{(s)}h'(x_{t}^{(s)})+\gamma_{t-1}^{(s)}\left(x_{t}^{(s)}-z_{t-1}^{(s)}\right),x_{t}^{(s)}-z_{t}^{(s)}\right\rangle  & \leq0,
\end{align*}
where $h'(x_{t}^{(s)})\in\partial h(x_{t}^{(s)})$ is a subgradient
of $h$ at $x_{t}^{(s)}$. We rearrange the above inequality and obtain
\begin{align*}
a^{(s)}\left\langle g_{t-1}^{(s)},x_{t}^{(s)}-z_{t}^{(s)}\right\rangle  & \le\gamma_{t-1}^{(s)}\left\langle x_{t}^{(s)}-z_{t-1}^{(s)},z_{t}^{(s)}-x_{t}^{(s)}\right\rangle +a^{(s)}\left\langle h'(x_{t}^{(s)}),z_{t}^{(s)}-x_{t}^{(s)}\right\rangle \\
 & \overset{(a)}{\leq}\gamma_{t-1}^{(s)}\left\langle x_{t}^{(s)}-z_{t-1}^{(s)},z_{t}^{(s)}-x_{t}^{(s)}\right\rangle +a^{(s)}\left(h(z_{t}^{(s)})-h(x_{t}^{(s)})\right)\\
 & \overset{(b)}{=}\frac{\gamma_{t-1}^{(s)}}{2}\left\Vert z_{t-1}^{(s)}-z_{t}^{(s)}\right\Vert ^{2}-\frac{\gamma_{t-1}^{(s)}}{2}\left\Vert z_{t-1}^{(s)}-x_{t}^{(s)}\right\Vert ^{2}-\frac{\gamma_{t-1}^{(s)}}{2}\left\Vert x_{t}^{(s)}-z_{t}^{(s)}\right\Vert ^{2}\\
 & \quad+a^{(s)}\left(h(z_{t}^{(s)})-h(x_{t}^{(s)})\right),
\end{align*}
where $(a)$ follows from the convexity of $h$ and the fact that
$h'(x_{t}^{(s)})\in\partial h(x_{t}^{(s)})$, and $(b)$ is due to
the identity $\langle a,b\rangle=\frac{1}{2}\left(\left\Vert a+b\right\Vert ^{2}-\left\Vert a\right\Vert ^{2}-\left\Vert b\right\Vert ^{2}\right)$.

Using the optimality condition of $z_{t}^{(s)}$, we have
\begin{align*}
\left\langle a^{(s)}g_{t}^{(s)}+a^{(s)}h'(z_{t}^{(s)})+\gamma_{t-1}^{(s)}\left(z_{t}^{(s)}-z_{t-1}^{(s)}\right)+\left(\gamma_{t}^{(s)}-\gamma_{t-1}^{(s)}\right)\left(z_{t}^{(s)}-x_{t}^{(s)}\right),z_{t}^{(s)}-x^{*}\right\rangle  & \leq0
\end{align*}
where $h'(z_{t}^{(s)})\in\partial h(z_{t}^{(s)})$ is a subgradient
of $h$ at $z_{t}^{(s)}$. We rearrange the above inequality and obtain
\begin{align*}
a^{(s)}\left\langle g_{t}^{(s)},z_{t}^{(s)}-x^{*}\right\rangle  & \leq\gamma_{t-1}^{(s)}\left\langle z_{t}^{(s)}-z_{t-1}^{(s)},x^{*}-z_{t}^{(s)}\right\rangle +\left(\gamma_{t}^{(s)}-\gamma_{t-1}^{(s)}\right)\left\langle z_{t}^{(s)}-x_{t}^{(s)},x^{*}-z_{t}^{(s)}\right\rangle \\
 & \quad+a^{(s)}\left\langle h'(z_{t}^{(s)}),x^{*}-z_{t}^{(s)}\right\rangle \\
 & \overset{(c)}{\leq}\gamma_{t-1}^{(s)}\left\langle z_{t}^{(s)}-z_{t-1}^{(s)},x^{*}-z_{t}^{(s)}\right\rangle +\left(\gamma_{t}^{(s)}-\gamma_{t-1}^{(s)}\right)\left\langle z_{t}^{(s)}-x_{t}^{(s)},x^{*}-z_{t}^{(s)}\right\rangle \\
 & \quad+a^{(s)}\left(h(x^{*})-h(z_{t}^{(s)})\right)\\
 & \overset{(d)}{=}\frac{\gamma_{t-1}^{(s)}}{2}\left[\left\Vert z_{t-1}^{(s)}-x^{*}\right\Vert ^{2}-\left\Vert z_{t}^{(s)}-x^{*}\right\Vert ^{2}-\left\Vert z_{t}^{(s)}-z_{t-1}^{(s)}\right\Vert ^{2}\right]\\
 & \quad+\frac{\gamma_{t}^{(s)}-\gamma_{t-1}^{(s)}}{2}\left[\left\Vert x_{t}^{(s)}-x^{*}\right\Vert ^{2}-\left\Vert z_{t}^{(s)}-x^{*}\right\Vert ^{2}-\left\Vert x_{t}^{(s)}-z_{t}^{(s)}\right\Vert ^{2}\right]\\
 & \quad+a^{(s)}\left(h(x^{*})-h(z_{t}^{(s)})\right)\\
 & =\frac{\gamma_{t}^{(s)}-\gamma_{t-1}^{(s)}}{2}\left\Vert x_{t}^{(s)}-x^{*}\right\Vert ^{2}+\frac{\gamma_{t-1}^{(s)}}{2}\left\Vert z_{t-1}^{(s)}-x^{*}\right\Vert ^{2}-\frac{\gamma_{t}^{(s)}}{2}\left\Vert z_{t}^{(s)}-x^{*}\right\Vert ^{2}\\
 & \quad-\frac{\gamma_{t-1}^{(s)}}{2}\left\Vert z_{t}^{(s)}-z_{t-1}^{(s)}\right\Vert ^{2}-\frac{\gamma_{t}^{(s)}-\gamma_{t-1}^{(s)}}{2}\left\Vert x_{t}^{(s)}-z_{t}^{(s)}\right\Vert ^{2}\\
 & \quad+a^{(s)}\left(h(x)-h(z_{t}^{(s)})\right),
\end{align*}
where $(c)$ follows from the convexity of $h$ and the fact that
$h'(z_{t}^{(s)})\in\partial h(z_{t}^{(s)})$, and $(d)$ is due to
the identity $\langle a,b\rangle=\frac{1}{2}\left(\left\Vert a+b\right\Vert ^{2}-\left\Vert a\right\Vert ^{2}-\left\Vert b\right\Vert ^{2}\right)$.

For the third inequality, we have
\begin{align*}
a^{(s)}\left\langle g_{t}^{(s)}-g_{t-1}^{(s)},x_{t}^{(s)}-z_{t}^{(s)}\right\rangle  & \overset{(e)}{\leq}a^{(s)}\left\Vert g_{t}^{(s)}-g_{t-1}^{(s)}\right\Vert \left\Vert x_{t}^{(s)}-z_{t}^{(s)}\right\Vert \\
 & \overset{(f)}{\leq}\frac{\left(a^{(s)}\right)^{2}}{2\gamma_{t}^{(s)}}\left\Vert g_{t}^{(s)}-g_{t-1}^{(s)}\right\Vert ^{2}+\frac{\gamma_{t}^{(s)}}{2}\left\Vert x_{t}^{(s)}-z_{t}^{(s)}\right\Vert ^{2}.
\end{align*}
 where $(e)$ is by the Cauchy--Schwarz inequality, $(f)$ is by
Young's inequality.
\end{proof}

With above results, we obtain the descent lemma for one iteration.
A key idea to remove $\left(a^{(s)}\right)^{2}$ from the coefficient
of $\left(F(\avx_{t}^{(s)})-F(x^{*})\right)$ is to split the term
$\frac{\left(a^{(s)}\right)^{2}}{2\gamma_{t}^{(s)}}\left\Vert g_{t}^{(s)}-g_{t-1}^{(s)}\right\Vert ^{2}$
into $\left(\frac{\left(a^{(s)}\right)^{2}}{2\gamma_{t}^{(s)}}-\frac{\left(a^{(s)}\right)^{2}}{16\beta}\right)\left\Vert g_{t}^{(s)}-g_{t-1}^{(s)}\right\Vert ^{2}+\frac{\left(a^{(s)}\right)^{2}}{16\beta}\left\Vert g_{t}^{(s)}-g_{t-1}^{(s)}\right\Vert ^{2}$
and apply the VR lemma for the second term. 
\begin{lem}
\label{lem:AdaVRAE-One-Step-Descent-II}For all epochs $s\geq1$ and
all iterations $t\in\left[T_{s}\right]$, we have 
\begin{align*}
 & \E\left[\left(A_{t}^{(s)}-\left(a^{(s)}\right)^{2}\right)\left(F(\avx_{t}^{(s)})-F(x^{*})\right)-A_{t-1}^{(s)}\left(F(\avx_{t-1}^{(s)})-F(x^{*})\right)\right]\\
\le & \E\left[\frac{\gamma_{t-1}^{(s)}}{2}\left\Vert z_{t-1}^{(s)}-x^{*}\right\Vert ^{2}-\frac{\gamma_{t}^{(s)}}{2}\left\Vert z_{t}^{(s)}-x^{*}\right\Vert ^{2}+\frac{\gamma_{t}^{(s)}-\gamma_{t-1}^{(s)}}{2}\left\Vert x_{t}^{(s)}-x^{*}\right\Vert ^{2}\right]\\
 & +\E\left[\left(a^{(s)}\right)^{2}\left\langle \nabla f(\avx_{t}^{(s)}),u^{(s-1)}-\avx_{t}^{(s)}\right\rangle \right]\\
 & +\E\left[\frac{\left(a^{(s)}\right)^{2}}{2}\left(f(u^{(s-1)})-f(\avx_{t}^{(s)})-\left\langle \nabla f(\avx_{t}^{(s)}),u^{(s-1)}-\avx_{t}^{(s)}\right\rangle \right)\right]\\
 & +\E\left[\frac{\left(a^{(s)}\right)^{2}}{2}\left(f(u^{(s-1)})-f(\avx_{t-1}^{(s)})-\left\langle \nabla f(\avx_{t-1}^{(s)}),u^{(s-1)}-\avx_{t-1}^{(s)}\right\rangle \right)\right]\\
 & +\E\left[\left(\frac{\left(a^{(s)}\right)^{2}}{2\gamma_{t}^{(s)}}-\frac{\left(a^{(s)}\right)^{2}}{16\beta}\right)\left\Vert g_{t}^{(s)}-g_{t-1}^{(s)}\right\Vert ^{2}+\left(\frac{\left(a^{(s)}\right)^{2}}{8\beta}-\frac{A_{t-1}^{(s)}}{2\beta}\right)\left\Vert \nabla f(\avx_{t}^{(s)})-\nabla f(\avx_{t-1}^{(s)})\right\Vert ^{2}\right]\\
 & +\E\left[\left(a^{(s)}\right)^{2}\left(h(u^{(s-1)})-h(\avx_{t}^{(s)})\right)\right].
\end{align*}
\end{lem}
\begin{proof}
By Lemma \ref{lem:AdaVRAE-Stochastic-Regret-Bound}, we can bound
$a^{(s)}\left\langle g_{t}^{(s)},x_{t}^{(s)}-x^{*}\right\rangle $
as follows
\begin{align*}
a^{(s)}\left\langle g_{t}^{(s)},x_{t}^{(s)}-x^{*}\right\rangle  & \leq\frac{\gamma_{t-1}^{(s)}}{2}\left\Vert z_{t-1}^{(s)}-x^{*}\right\Vert ^{2}-\frac{\gamma_{t}^{(s)}}{2}\left\Vert z_{t}^{(s)}-x^{*}\right\Vert ^{2}+\frac{\gamma_{t}^{(s)}-\gamma_{t-1}^{(s)}}{2}\left\Vert x_{t}^{(s)}-x^{*}\right\Vert ^{2}-\frac{\gamma_{t-1}^{(s)}}{2}\left\Vert z_{t-1}^{(s)}-x_{t}^{(s)}\right\Vert ^{2}\\
 & \quad+a^{(s)}\left(h(x^{*})-h(x_{t}^{(s)})\right)+\frac{\left(a^{(s)}\right)^{2}}{2\gamma_{t}^{(s)}}\left\Vert g_{t}^{(s)}-g_{t-1}^{(s)}\right\Vert ^{2}.\\
 & \leq\frac{\gamma_{t-1}^{(s)}}{2}\left\Vert z_{t-1}^{(s)}-x^{*}\right\Vert ^{2}-\frac{\gamma_{t}^{(s)}}{2}\left\Vert z_{t}^{(s)}-x^{*}\right\Vert ^{2}+\frac{\gamma_{t}^{(s)}-\gamma_{t-1}^{(s)}}{2}\left\Vert x_{t}^{(s)}-x^{*}\right\Vert ^{2}\\
 & \quad+a^{(s)}\left(h(x^{*})-h(x_{t}^{(s)})\right)+\frac{\left(a^{(s)}\right)^{2}}{2\gamma_{t}^{(s)}}\left\Vert g_{t}^{(s)}-g_{t-1}^{(s)}\right\Vert ^{2}.
\end{align*}
Combining the above result with Lemma \ref{lem:AdaVRAE-One-Step-Descent-I},
we know
\begin{align}
 & \E\left[\left(A_{t}^{(s)}-\left(a^{(s)}\right)^{2}\right)\left(F(\avx_{t}^{(s)})-F(x^{*})\right)-A_{t-1}^{(s)}\left(F(\avx_{t-1}^{(s)})-F(x^{*})\right)\right]\nonumber \\
\leq & \E\left[\frac{\gamma_{t-1}^{(s)}}{2}\left\Vert z_{t-1}^{(s)}-x^{*}\right\Vert ^{2}-\frac{\gamma_{t}^{(s)}}{2}\left\Vert z_{t}^{(s)}-x^{*}\right\Vert ^{2}+\frac{\gamma_{t}^{(s)}-\gamma_{t-1}^{(s)}}{2}\left\Vert x_{t}^{(s)}-x^{*}\right\Vert ^{2}\right]\nonumber \\
 & +\E\left[\left(a^{(s)}\right)^{2}\left\langle \nabla f(\avx_{t}^{(s)}),u^{(s-1)}-\avx_{t}^{(s)}\right\rangle \right]\nonumber \\
 & +\E\left[\frac{\left(a^{(s)}\right)^{2}}{2\gamma_{t}^{(s)}}\left\Vert g_{t}^{(s)}-g_{t-1}^{(s)}\right\Vert ^{2}-\frac{A_{t-1}^{(s)}}{2\beta}\left\Vert \nabla f(\avx_{t}^{(s)})-\nabla f(\avx_{t-1}^{(s)})\right\Vert ^{2}\right]\nonumber \\
 & +\E\left[\left(A_{t}^{(s)}-\left(a^{(s)}\right)^{2}\right)\left(h(\avx_{t}^{(s)})-h(x^{*})\right)-A_{t-1}^{(s)}\left(h(\avx_{t-1}^{(s)})-h(x^{*})\right)+a^{(s)}\left(h(x^{*})-h(x_{t}^{(s)})\right)\right].\label{eq:AdaVRAE-2}
\end{align}
Note that
\begin{align}
 & \left(A_{t}^{(s)}-\left(a^{(s)}\right)^{2}\right)\left(h(\avx_{t}^{(s)})-h(x^{*})\right)-A_{t-1}^{(s)}\left(h(\avx_{t-1}^{(s)})-h(x^{*})\right)+a^{(s)}\left(h(x^{*})-h(x_{t}^{(s)})\right)\nonumber \\
= & \left(A_{t}^{(s)}-\left(a^{(s)}\right)^{2}\right)\left(h(\avx_{t}^{(s)})-h(x^{*})\right)+\left(a^{(s)}\right)^{2}\left(h(u^{(s-1)})-h(x^{*})\right)\nonumber \\
 & -\left(a^{(s)}\right)^{2}\left(h(u^{(s-1)})-h(x^{*})\right)-A_{t-1}^{(s)}\left(h(\avx_{t-1}^{(s)})-h(x^{*})\right)-a^{(s)}\left(h(x_{t}^{(s)})-h(x^{*})\right)\nonumber \\
\stackrel{(a)}{\le} & \left(A_{t}^{(s)}-\left(a^{(s)}\right)^{2}\right)\left(h(\avx_{t}^{(s)})-h(x^{*})\right)+\left(a^{(s)}\right)^{2}\left(h(u^{(s-1)})-h(x^{*})\right)\nonumber \\
 & -A_{t}^{(s)}(h(\avx_{t}^{(s)})-h(x^{*}))\nonumber \\
= & \left(a^{(s)}\right)^{2}\left(h(u^{(s-1)})-h(\avx_{t}^{(s)})\right),\label{eq:AdaVRAE-3}
\end{align}
where $(a)$ is by the convexity of $h$ and $A_{t}^{(s)}=A_{t-1}^{(s)}+a^{(s)}+\left(a^{(s)}\right)^{2}$.
Plugging in (\ref{eq:AdaVRAE-3}) into (\ref{eq:AdaVRAE-2}), we know
\begin{align*}
 & \E\left[\left(A_{t}^{(s)}-\left(a^{(s)}\right)^{2}\right)\left(F(\avx_{t}^{(s)})-F(x^{*})\right)-A_{t-1}^{(s)}\left(F(\avx_{t-1}^{(s)})-F(x^{*})\right)\right]\\
\le & \E\left[\frac{\gamma_{t-1}^{(s)}}{2}\left\Vert z_{t-1}^{(s)}-x^{*}\right\Vert ^{2}-\frac{\gamma_{t}^{(s)}}{2}\left\Vert z_{t}^{(s)}-x^{*}\right\Vert ^{2}+\frac{\gamma_{t}^{(s)}-\gamma_{t-1}^{(s)}}{2}\left\Vert x_{t}^{(s)}-x^{*}\right\Vert ^{2}\right]\\
 & +\E\left[\left(a^{(s)}\right)^{2}\left\langle \nabla f(\avx_{t}^{(s)}),u^{(s-1)}-\avx_{t}^{(s)}\right\rangle \right]\\
 & +\E\left[\frac{\left(a^{(s)}\right)^{2}}{2\gamma_{t}^{(s)}}\left\Vert g_{t}^{(s)}-g_{t-1}^{(s)}\right\Vert ^{2}-\frac{A_{t-1}^{(s)}}{2\beta}\left\Vert \nabla f(\avx_{t}^{(s)})-\nabla f(\avx_{t-1}^{(s)})\right\Vert ^{2}\right]\\
 & +\E\left[\left(a^{(s)}\right)^{2}\left(h(u^{(s-1)})-h(\avx_{t}^{(s)})\right)\right].
\end{align*}
Now for $\E\left[\left\Vert g_{t}^{(s)}-g_{t-1}^{(s)}\right\Vert ^{2}\right]$,
when $1<t<T_{s}$, we have
\begin{align}
\E\left[\left\Vert g_{t}^{(s)}-g_{t-1}^{(s)}\right\Vert ^{2}\right] & \leq\E\left[4\left\Vert \nabla f(\avx_{t}^{(s)})-g_{t}^{(s)}\right\Vert ^{2}+4\left\Vert \nabla f(\avx_{t-1}^{(s)})-g_{t-1}^{(s)}\right\Vert ^{2}\right]\nonumber \\
 & \quad+\E\left[2\left\Vert \nabla f(\avx_{t}^{(s)})-\nabla f(\avx_{t-1}^{(s)})\right\Vert ^{2}\right]\nonumber \\
 & \overset{(b)}{\leq}\E\left[8\beta\left(f(u^{(s-1)})-f(\avx_{t}^{(s)})-\left\langle \nabla f(\avx_{t}^{(s)}),u^{(s-1)}-\avx_{t}^{(s)}\right\rangle \right)\right]\nonumber \\
 & \quad+\E\left[8\beta\left(f(u^{(s-1)})-f(\avx_{t-1}^{(s)})-\left\langle \nabla f(\avx_{t-1}^{(s)}),u^{(s-1)}-\avx_{t-1}^{(s)}\right\rangle \right)\right]\nonumber \\
 & \quad+\E\left[2\left\Vert \nabla f(\avx_{t}^{(s)})-\nabla f(\avx_{t-1}^{(s)})\right\Vert ^{2}\right],\label{eq:AdaVRAE-4}
\end{align}
where $(b)$ is by Lemma \ref{lem:Variance-Reduction} for all $1<t<T_{s}$.
When $t=1$, note that both $\left\Vert \nabla f(\avx_{t-1}^{(s)})-g_{t-1}^{(s)}\right\Vert ^{2}$
and $f(u^{(s-1)})-f(\avx_{t-1}^{(s)})-\left\langle \nabla f(\avx_{t-1}^{(s)}),u^{(s-1)}-\avx_{t-1}^{(s)}\right\rangle $
are zero by our definition $\avx_{0}^{(s)}=u^{(s-1)}$ and $\nabla f(\avx_{0}^{(s)})=\nabla f(\avx_{T_{s-1}}^{(s-1)})=g_{T_{s-1}}^{(s-1)}=g_{0}^{(s)}$,
which means the above inequality is still true. When $t=T_{s}$, note
that $\left\Vert \nabla f(\avx_{t}^{(s)})-g_{t}^{(s)}\right\Vert ^{2}=0$
and $f(u^{(s-1)})-f(\avx_{t}^{(s)})-\left\langle \nabla f(\avx_{t}^{(s)}),u^{(s-1)}-\avx_{t}^{(s)}\right\rangle $
is always non-negative due to the convexity of $f$. So the above
inequality also holds in this case. Now we conlclude the above inequality
is right for $t\in\left[T_{s}\right]$.

Splitting $\frac{\left(a^{(s)}\right)^{2}}{2\gamma_{t}^{(s)}}\left\Vert g_{t}^{(s)}-g_{t-1}^{(s)}\right\Vert ^{2}$
into $\left(\frac{\left(a^{(s)}\right)^{2}}{2\gamma_{t}^{(s)}}-\frac{\left(a^{(s)}\right)^{2}}{16\beta}\right)\left\Vert g_{t}^{(s)}-g_{t-1}^{(s)}\right\Vert ^{2}+\frac{\left(a^{(s)}\right)^{2}}{16\beta}\left\Vert g_{t}^{(s)}-g_{t-1}^{(s)}\right\Vert ^{2}$
and applying (\ref{eq:AdaVRAE-4}) to $\frac{\left(a^{(s)}\right)^{2}}{16\beta}\left\Vert g_{t}^{(s)}-g_{t-1}^{(s)}\right\Vert ^{2}$,
we have
\begin{align*}
 & \E\left[\left(A_{t}^{(s)}-\left(a^{(s)}\right)^{2}\right)\left(F(\avx_{t}^{(s)})-F(x^{*})\right)-A_{t-1}^{(s)}\left(F(\avx_{t-1}^{(s)})-F(x^{*})\right)\right]\\
\le & \E\left[\frac{\gamma_{t-1}^{(s)}}{2}\left\Vert z_{t-1}^{(s)}-x^{*}\right\Vert ^{2}-\frac{\gamma_{t}^{(s)}}{2}\left\Vert z_{t}^{(s)}-x^{*}\right\Vert ^{2}+\frac{\gamma_{t}^{(s)}-\gamma_{t-1}^{(s)}}{2}\left\Vert x_{t}^{(s)}-x^{*}\right\Vert ^{2}\right]\\
 & +\E\left[\left(a^{(s)}\right)^{2}\left\langle \nabla f(\avx_{t}^{(s)}),u^{(s-1)}-\avx_{t}^{(s)}\right\rangle \right]\\
 & +\E\left[\frac{\left(a^{(s)}\right)^{2}}{2}\left(f(u^{(s-1)})-f(\avx_{t}^{(s)})-\left\langle \nabla f(\avx_{t}^{(s)}),u^{(s-1)}-\avx_{t}^{(s)}\right\rangle \right)\right]\\
 & +\E\left[\frac{\left(a^{(s)}\right)^{2}}{2}\left(f(u^{(s-1)})-f(\avx_{t-1}^{(s)})-\left\langle \nabla f(\avx_{t-1}^{(s)}),u^{(s-1)}-\avx_{t-1}^{(s)}\right\rangle \right)\right]\\
 & +\E\left[\left(\frac{\left(a^{(s)}\right)^{2}}{2\gamma_{t}^{(s)}}-\frac{\left(a^{(s)}\right)^{2}}{16\beta}\right)\left\Vert g_{t}^{(s)}-g_{t-1}^{(s)}\right\Vert ^{2}+\left(\frac{\left(a^{(s)}\right)^{2}}{8\beta}-\frac{A_{t-1}^{(s)}}{2\beta}\right)\left\Vert \nabla f(\avx_{t}^{(s)})-\nabla f(\avx_{t-1}^{(s)})\right\Vert ^{2}\right]\\
 & +\E\left[\left(a^{(s)}\right)^{2}\left(h(u^{(s-1)})-h(\avx_{t}^{(s)})\right)\right].
\end{align*}
\end{proof}

\subsection{Single epoch progress and final output}

Even though Lemma \ref{lem:AdaVRAE-One-Step-Descent-II} looks somewhat
more convoluted, when we sum up over all iterations in one epoch,
many terms are canceled out nicely and we obtain the following lemma
that states the progress of the function value in one epoch. The trick
is to set the value for each term at the end of one epoch equal to
its value in the next one, with an exception for $A_{T_{s-1}}^{(s-1)}$.
Due to the accumulation of the term $\left(F(u^{(s-1)})-F(x^{*})\right)$
throughout the epoch, we will set $A_{0}^{(s)}=A_{T_{s-1}}^{(s-1)}-T_{s}\left(a^{(s)}\right)^{2}$.
\begin{lem}
\label{lem:AdaVRAE-One-Epoch-Descent}For all epochs $s\geq1$, if
\begin{align*}
\left(a^{(s)}\right)^{2} & \leq4A_{t-1}^{(s)},\forall t\in\left[T_{s}\right].
\end{align*}
We have
\begin{align*}
 & \E\left[A_{T_{s}}^{(s)}\left(F(u^{(s)})-F(x^{*})\right)-A_{T_{s-1}}^{(s-1)}\left(F(u^{(s-1)})-F(x^{*})\right)\right]\\
\leq & \E\left[\frac{\gamma_{0}^{(s)}}{2}\left\Vert z_{0}^{(s)}-x^{*}\right\Vert ^{2}-\frac{\gamma_{0}^{(s+1)}}{2}\left\Vert z_{0}^{(s+1)}-x^{*}\right\Vert ^{2}+\sum_{t=1}^{T_{s}}\frac{\gamma_{t}^{(s)}-\gamma_{t-1}^{(s)}}{2}\left\Vert x_{t}^{(s)}-x^{*}\right\Vert ^{2}\right]\\
 & +\E\left[\sum_{t=1}^{T_{s}}\left(\frac{\left(a^{(s)}\right)^{2}}{2\gamma_{t}^{(s)}}-\frac{\left(a^{(s)}\right)^{2}}{16\beta}\right)\left\Vert g_{t}^{(s)}-g_{t-1}^{(s)}\right\Vert ^{2}\right].
\end{align*}
\end{lem}
\begin{proof}
Using Lemma \ref{lem:AdaVRAE-One-Step-Descent-II}, we know
\begin{align}
 & \sum_{t=1}^{T_{s}}\E\left[\left(A_{t}^{(s)}-\left(a^{(s)}\right)^{2}\right)\left(F(\avx_{t}^{(s)})-F(x^{*})\right)-A_{t-1}^{(s)}\left(F(\avx_{t-1}^{(s)})-F(x^{*})\right)\right]\nonumber \\
\le & \sum_{t=1}^{T_{s}}\E\left[\frac{\gamma_{t-1}^{(s)}}{2}\left\Vert z_{t-1}^{(s)}-x^{*}\right\Vert ^{2}-\frac{\gamma_{t}^{(s)}}{2}\left\Vert z_{t}^{(s)}-x^{*}\right\Vert ^{2}+\frac{\gamma_{t}^{(s)}-\gamma_{t-1}^{(s)}}{2}\left\Vert x_{t}^{(s)}-x^{*}\right\Vert ^{2}\right]\nonumber \\
 & +\sum_{t=1}^{T_{s}}\E\left[\left(a^{(s)}\right)^{2}\left\langle \nabla f(\avx_{t}^{(s)}),u^{(s-1)}-\avx_{t}^{(s)}\right\rangle +\left(a^{(s)}\right)^{2}\left(h(u^{(s-1)})-h(\avx_{t}^{(s)})\right)\right]\nonumber \\
 & +\sum_{t=1}^{T_{s}}\E\left[\frac{\left(a^{(s)}\right)^{2}}{2}\left(f(u^{(s-1)})-f(\avx_{t}^{(s)})-\left\langle \nabla f(\avx_{t}^{(s)}),u^{(s-1)}-\avx_{t}^{(s)}\right\rangle \right)\right]\nonumber \\
 & +\sum_{t=1}^{T_{s}}\E\left[\frac{\left(a^{(s)}\right)^{2}}{2}\left(f(u^{(s-1)})-f(\avx_{t-1}^{(s)})-\left\langle \nabla f(\avx_{t-1}^{(s)}),u^{(s-1)}-\avx_{t-1}^{(s)}\right\rangle \right)\right]\nonumber \\
 & +\sum_{t=1}^{T_{s}}\E\left[\left(\frac{\left(a^{(s)}\right)^{2}}{2\gamma_{t}^{(s)}}-\frac{\left(a^{(s)}\right)^{2}}{16\beta}\right)\left\Vert g_{t}^{(s)}-g_{t-1}^{(s)}\right\Vert ^{2}+\left(\frac{\left(a^{(s)}\right)^{2}}{8\beta}-\frac{A_{t-1}^{(s)}}{2\beta}\right)\left\Vert \nabla f(\avx_{t}^{(s)})-\nabla f(\avx_{t-1}^{(s)})\right\Vert ^{2}\right]\nonumber \\
\overset{(a)}{=} & \E\left[\frac{\gamma_{0}^{(s)}}{2}\left\Vert z_{0}^{(s)}-x^{*}\right\Vert ^{2}-\frac{\gamma_{0}^{(s+1)}}{2}\left\Vert z_{0}^{(s+1)}-x^{*}\right\Vert ^{2}+\sum_{t=1}^{T_{s}}\frac{\gamma_{t}^{(s)}-\gamma_{t-1}^{(s)}}{2}\left\Vert x_{t}^{(s)}-x^{*}\right\Vert ^{2}\right]\nonumber \\
 & +\E\left[\sum_{t=1}^{T_{s}-1}\left(a^{(s)}\right)^{2}\left(f(u^{(s-1)})-f(\avx_{t}^{(s)})\right)+\left(a^{(s)}\right)^{2}\left(h(u^{(s-1)})-h(\avx_{t}^{(s)})\right)\right]\nonumber \\
 & +\E\left[\frac{\left(a^{(s)}\right)^{2}}{2}\left(f(u^{(s-1)})-f(\avx_{T_{s}}^{(s)})+\left\langle \nabla f(\avx_{T_{s}}^{(s)}),u^{(s-1)}-\avx_{T_{s}}^{(s)}\right\rangle \right)\right]\nonumber \\
 & +\E\left[\sum_{t=1}^{T_{s}}\left(\frac{\left(a^{(s)}\right)^{2}}{2\gamma_{t}^{(s)}}-\frac{\left(a^{(s)}\right)^{2}}{16\beta}\right)\left\Vert g_{t}^{(s)}-g_{t-1}^{(s)}\right\Vert ^{2}+\left(\frac{\left(a^{(s)}\right)^{2}}{8\beta}-\frac{A_{t-1}^{(s)}}{2\beta}\right)\left\Vert \nabla f(\avx_{t}^{(s)})-\nabla f(\avx_{t-1}^{(s)})\right\Vert ^{2}\right]\nonumber \\
\overset{(b)}{\leq} & \E\left[\frac{\gamma_{0}^{(s)}}{2}\left\Vert z_{0}^{(s)}-x^{*}\right\Vert ^{2}-\frac{\gamma_{0}^{(s+1)}}{2}\left\Vert z_{0}^{(s+1)}-x^{*}\right\Vert ^{2}+\sum_{t=1}^{T_{s}}\frac{\gamma_{t}^{(s)}-\gamma_{t-1}^{(s)}}{2}\left\Vert x_{t}^{(s)}-x^{*}\right\Vert ^{2}\right]\nonumber \\
 & +\E\left[\sum_{t=1}^{T_{s}}\left(a^{(s)}\right)^{2}\left(f(u^{(s-1)})-f(\avx_{t}^{(s)})+\left(a^{(s)}\right)^{2}\left(h(u^{(s-1)})-h(\avx_{t}^{(s)})\right)\right)\right]\nonumber \\
 & +\E\left[\sum_{t=1}^{T_{s}}\left(\frac{\left(a^{(s)}\right)^{2}}{2\gamma_{t}^{(s)}}-\frac{\left(a^{(s)}\right)^{2}}{16\beta}\right)\left\Vert g_{t}^{(s)}-g_{t-1}^{(s)}\right\Vert ^{2}+\left(\frac{\left(a^{(s)}\right)^{2}}{8\beta}-\frac{A_{t-1}^{(s)}}{2\beta}\right)\left\Vert \nabla f(\avx_{t}^{(s)})-\nabla f(\avx_{t-1}^{(s)})\right\Vert ^{2}\right]\nonumber \\
\overset{(c)}{=} & \E\left[\frac{\gamma_{0}^{(s)}}{2}\left\Vert z_{0}^{(s)}-x^{*}\right\Vert ^{2}-\frac{\gamma_{0}^{(s+1)}}{2}\left\Vert z_{0}^{(s+1)}-x^{*}\right\Vert ^{2}+\sum_{t=1}^{T_{s}}\frac{\gamma_{t}^{(s)}-\gamma_{t-1}^{(s)}}{2}\left\Vert x_{t}^{(s)}-x^{*}\right\Vert ^{2}\right]\nonumber \\
 & +\E\left[\sum_{t=1}^{T_{s}}\left(a^{(s)}\right)^{2}\left(F(u^{(s-1)})-F(\avx_{t}^{(s)})\right)\right]\nonumber \\
 & +\E\left[\sum_{t=1}^{T_{s}}\left(\frac{\left(a^{(s)}\right)^{2}}{2\gamma_{t}^{(s)}}-\frac{\left(a^{(s)}\right)^{2}}{16\beta}\right)\left\Vert g_{t}^{(s)}-g_{t-1}^{(s)}\right\Vert ^{2}+\left(\frac{\left(a^{(s)}\right)^{2}}{8\beta}-\frac{A_{t-1}^{(s)}}{2\beta}\right)\left\Vert \nabla f(\avx_{t}^{(s)})-\nabla f(\avx_{t-1}^{(s)})\right\Vert ^{2}\right].\label{eq:AdaVRAE-5}
\end{align}
where $(a)$ is due to $z_{0}^{(s+1)}=z_{T_{s}}^{(s)}$, $\gamma_{0}^{(s+1)}=\gamma_{T_{s}}^{(s)}$,
$\avx_{0}^{(s)}=u^{(s-1)}$, $(b)$ is by the convexity of $f$
\[
\left\langle \nabla f(\avx_{T_{s}}^{(s)}),u^{(s-1)}-\avx_{T_{s}}^{(s)}\right\rangle \leq f(u^{(s-1)})-f(\avx_{T_{s}}^{(s)}),
\]
$(c)$ is by the definition of $F=f+h$. By adding $\sum_{t=1}^{T_{s}}\left(a^{(s)}\right)^{2}\left(F(\avx_{t}^{(s)})-F(x^{*})\right)$
to both sides of \ref{eq:AdaVRAE-5}. we obtain
\begin{align*}
 & \E\left[\sum_{t=1}^{T_{s}}A_{t}^{(s)}\left(F(\avx_{t}^{(s)})-F(x^{*})\right)-A_{t-1}^{(s)}\left(F(\avx_{t-1}^{(s)})-F(x^{*})\right)\right]\\
\leq & \E\left[\frac{\gamma_{0}^{(s)}}{2}\left\Vert z_{0}^{(s)}-x^{*}\right\Vert ^{2}-\frac{\gamma_{0}^{(s+1)}}{2}\left\Vert z_{0}^{(s+1)}-x^{*}\right\Vert ^{2}+\sum_{t=1}^{T_{s}}\frac{\gamma_{t}^{(s)}-\gamma_{t-1}^{(s)}}{2}\left\Vert x_{t}^{(s)}-x^{*}\right\Vert ^{2}\right]\\
 & +\E\left[T_{s}\left(a^{(s)}\right)^{2}\left(F(u^{(s-1)})-F(x^{*})\right)\right]\\
 & +\E\left[\sum_{t=1}^{T_{s}}\left(\frac{\left(a^{(s)}\right)^{2}}{2\gamma_{t}^{(s)}}-\frac{\left(a^{(s)}\right)^{2}}{16\beta}\right)\left\Vert g_{t}^{(s)}-g_{t-1}^{(s)}\right\Vert ^{2}+\left(\frac{\left(a^{(s)}\right)^{2}}{8\beta}-\frac{A_{t-1}^{(s)}}{2\beta}\right)\left\Vert \nabla f(\avx_{t}^{(s)})-\nabla f(\avx_{t-1}^{(s)})\right\Vert ^{2}\right].
\end{align*}
Note that
\begin{align*}
 & \E\left[\sum_{t=1}^{T_{s}}A_{t}^{(s)}\left(F(\avx_{t}^{(s)})-F(x^{*})\right)-A_{t-1}^{(s)}\left(F(\avx_{t-1}^{(s)})-F(x^{*})\right)\right]\\
= & \E\left[A_{T_{s}}^{(s)}\left(F(\avx_{T_{s}}^{(s)})-F(x^{*})\right)-A_{0}^{(s)}\left(F(\avx_{0}^{(s)})-F(x^{*})\right)\right]\\
\overset{(d)}{=} & \E\left[A_{T_{s}}^{(s)}\left(F(u^{(s)})-F(x^{*})\right)-A_{0}^{(s)}\left(F(u^{(s-1)})-F(x^{*})\right)\right],
\end{align*}
where $(d)$ is due to the definition $u^{(s)}=\avx_{T_{s}}^{(s)}$
and $\avx_{0}^{(s)}=u^{(s-1)}$. Finally we have
\begin{align*}
 & \E\left[A_{T_{s}}^{(s)}\left(F(u^{(s)})-F(x^{*})\right)-\left(A_{0}^{(s)}+T_{s}\left(a^{(s)}\right)^{2}\right)\left(F(u^{(s-1)})-F(x^{*})\right)\right]\\
\leq & \E\left[\frac{\gamma_{0}^{(s)}}{2}\left\Vert z_{0}^{(s)}-x^{*}\right\Vert ^{2}-\frac{\gamma_{0}^{(s+1)}}{2}\left\Vert z_{0}^{(s+1)}-x^{*}\right\Vert ^{2}+\sum_{t=1}^{T_{s}}\frac{\gamma_{t}^{(s)}-\gamma_{t-1}^{(s)}}{2}\left\Vert x_{t}^{(s)}-x^{*}\right\Vert ^{2}\right]\\
 & +\E\left[\sum_{t=1}^{T_{s}}\left(\frac{\left(a^{(s)}\right)^{2}}{2\gamma_{t}^{(s)}}-\frac{\left(a^{(s)}\right)^{2}}{16\beta}\right)\left\Vert g_{t}^{(s)}-g_{t-1}^{(s)}\right\Vert ^{2}+\left(\frac{\left(a^{(s)}\right)^{2}}{8\beta}-\frac{A_{t-1}^{(s)}}{2\beta}\right)\left\Vert \nabla f(\avx_{t}^{(s)})-\nabla f(\avx_{t-1}^{(s)})\right\Vert ^{2}\right].
\end{align*}
Combining the fact $A_{0}^{(s)}=A_{T_{s-1}}^{(s-1)}-T_{s}\left(a^{(s)}\right)^{2}$
and our condition $\left(a^{(s)}\right)^{2}\leq4A_{t-1}^{(s)}$, we
get the desired result.
\end{proof}

The telescoping sum on the LSH allows us to obtain the guarantee for
the final output $u^{(S)}$.
\begin{lem}
\label{lem:AdaVRAE-Whole-Descent}For all $S\geq1$, assume we have
\begin{align*}
\left(a^{(s)}\right)^{2} & \leq4A_{t-1}^{(s)},\forall t\in\left[T_{s}\right],\forall s\in\left[S\right].
\end{align*}
Then
\begin{align*}
 & \E\left[A_{T_{S}}^{(S)}\left(F(u^{(S)})-F(x^{*})\right)\right]\\
\leq & A_{T_{0}}^{(0)}\left(F(u^{(0)})-F(x^{*})\right)+\frac{\gamma}{2}\left\Vert u^{(0)}-x^{*}\right\Vert ^{2}\\
 & +\E\left[\sum_{s=1}^{S}\sum_{t=1}^{T_{s}}\frac{\gamma_{t}^{(s)}-\gamma_{t-1}^{(s)}}{2}\left\Vert x_{t}^{(s)}-x^{*}\right\Vert ^{2}+\left(\frac{\left(a^{(s)}\right)^{2}}{2\gamma_{t}^{(s)}}-\frac{\left(a^{(s)}\right)^{2}}{16\beta}\right)\left\Vert g_{t}^{(s)}-g_{t-1}^{(s)}\right\Vert ^{2}\right].
\end{align*}
\end{lem}
\begin{proof}
Note that our assumptions satisfy the requirements for Lemma \ref{lem:AdaVRAE-One-Epoch-Descent},
by Applying Lemma \ref{lem:AdaVRAE-One-Epoch-Descent} and make the
telescoping sum from $s=1$ to $S$, we obtain
\begin{align*}
 & \E\left[A_{T_{S}}^{(S)}\left(F(u^{(S)})-F(x^{*})\right)\right]\\
\le & A_{T_{0}}^{(0)}\left(F(u^{(0)})-F(x^{*})\right)+\E\left[\frac{\gamma_{0}^{(s)}}{2}\left\Vert z_{0}^{(s)}-x^{*}\right\Vert ^{2}-\frac{\gamma_{0}^{(S+1)}}{2}\left\Vert z_{0}^{(S+1)}-x^{*}\right\Vert ^{2}\right]\\
 & +\E\left[\sum_{s=1}^{S}\sum_{t=1}^{T_{s}}\frac{\gamma_{t}^{(s)}-\gamma_{t-1}^{(s)}}{2}\left\Vert x_{t}^{(s)}-x^{*}\right\Vert ^{2}+\left(\frac{\left(a^{(s)}\right)^{2}}{2\gamma_{t}^{(s)}}-\frac{\left(a^{(s)}\right)^{2}}{16\beta}\right)\left\Vert g_{t}^{(s)}-g_{t-1}^{(s)}\right\Vert ^{2}\right]\\
\le & A_{T_{0}}^{(0)}\left(F(u^{(0)})-F(x^{*})\right)+\frac{\gamma_{0}^{(1)}}{2}\left\Vert z_{0}^{(s)}-x^{*}\right\Vert ^{2}\\
 & +\E\left[\sum_{s=1}^{S}\sum_{t=1}^{T_{s}}\frac{\gamma_{t}^{(s)}-\gamma_{t-1}^{(s)}}{2}\left\Vert x_{t}^{(s)}-x^{*}\right\Vert ^{2}+\left(\frac{\left(a^{(s)}\right)^{2}}{2\gamma_{t}^{(s)}}-\frac{\left(a^{(s)}\right)^{2}}{16\beta}\right)\left\Vert g_{t}^{(s)}-g_{t-1}^{(s)}\right\Vert ^{2}\right]\\
= & A_{T_{0}}^{(0)}\left(F(u^{(0)})-F(x^{*})\right)+\frac{\gamma}{2}\left\Vert u^{(0)}-x^{*}\right\Vert ^{2}\\
 & +\E\left[\sum_{s=1}^{S}\sum_{t=1}^{T_{s}}\frac{\gamma_{t}^{(s)}-\gamma_{t-1}^{(s)}}{2}\left\Vert x_{t}^{(s)}-x^{*}\right\Vert ^{2}+\left(\frac{\left(a^{(s)}\right)^{2}}{2\gamma_{t}^{(s)}}-\frac{\left(a^{(s)}\right)^{2}}{16\beta}\right)\left\Vert g_{t}^{(s)}-g_{t-1}^{(s)}\right\Vert ^{2}\right],
\end{align*}
where we use $\gamma_{0}^{(1)}=\gamma$ and $z_{0}^{(1)}=u^{(0)}$.
\end{proof}

\subsection{Bound for the residual term}

We turn to bound the term
\[
\sum_{s=1}^{S}\sum_{t=1}^{T_{s}}\frac{\gamma_{t}^{(s)}-\gamma_{t-1}^{(s)}}{2}\left\Vert x_{t}^{(s)}-x^{*}\right\Vert ^{2}+\left(\frac{\left(a^{(s)}\right)^{2}}{2\gamma_{t}^{(s)}}-\frac{\left(a^{(s)}\right)^{2}}{16\beta}\right)\left\Vert g_{t}^{(s)}-g_{t-1}^{(s)}\right\Vert ^{2}
\]

This follows the standard analysis used to bound the residual term
in adaptive methods. We first admit Lemma \ref{lem:AdaVRAE-Reminder-Bound-II}
to give the final bound for this term.
\begin{lem}
\label{lem:AdaVRAE-Reminder-Bound-I}If $\dom$ is a compact convex
set with diameter $D$, we have
\[
\sum_{s=1}^{S}\sum_{t=1}^{T_{s}}\frac{\gamma_{t}^{(s)}-\gamma_{t-1}^{(s)}}{2}\left\Vert x_{t}^{(s)}-x^{*}\right\Vert ^{2}+\left(\frac{\left(a^{(s)}\right)^{2}}{2\gamma_{t}^{(s)}}-\frac{\left(a^{(s)}\right)^{2}}{16\beta}\right)\left\Vert g_{t}^{(s)}-g_{t-1}^{(s)}\right\Vert ^{2}\le\frac{8\beta\left(D^{4}+2\eta^{4}\right)}{\eta^{2}}
\]
\end{lem}
\begin{proof}
It follows that
\begin{align*}
 & \sum_{s=1}^{S}\sum_{t=1}^{T_{s}}\frac{\gamma_{t}^{(s)}-\gamma_{t-1}^{(s)}}{2}\left\Vert x_{t}^{(s)}-x^{*}\right\Vert ^{2}+\left(\frac{\left(a^{(s)}\right)^{2}}{2\gamma_{t}^{(s)}}-\frac{\left(a^{(s)}\right)^{2}}{16\beta}\right)\left\Vert g_{t}^{(s)}-g_{t-1}^{(s)}\right\Vert ^{2}\\
\overset{(a)}{\leq} & \sum_{s=1}^{S}\sum_{t=1}^{T_{s}}\frac{\gamma_{t}^{(s)}-\gamma_{t-1}^{(s)}}{2}D^{2}+\left(\frac{\left(a^{(s)}\right)^{2}}{2\gamma_{t}^{(s)}}-\frac{\left(a^{(s)}\right)^{2}}{16\beta}\right)\left\Vert g_{t}^{(s)}-g_{t-1}^{(s)}\right\Vert ^{2}\\
\overset{(b)}{=} & \frac{\gamma_{T_{s}}^{(s)}-\gamma_{0}^{(1)}}{2}D^{2}+\sum_{s=1}^{S}\sum_{t=1}^{T_{s}}\left(\frac{\left(a^{(s)}\right)^{2}}{2\gamma_{t}^{(s)}}-\frac{\left(a^{(s)}\right)^{2}}{16\beta}\right)\left\Vert g_{t}^{(s)}-g_{t-1}^{(s)}\right\Vert ^{2}\\
= & \frac{\gamma_{T_{s}}^{(s)}-\gamma_{0}^{(1)}}{2}D^{2}-\frac{D^{4}}{16\beta\left(D^{4}+2\eta^{4}\right)}\sum_{s=1}^{S}\sum_{t=1}^{T_{s}}\left(a^{(s)}\right)^{2}\left\Vert g_{t}^{(s)}-g_{t-1}^{(s)}\right\Vert ^{2}\\
 & +\sum_{s=1}^{S}\sum_{t=1}^{T_{s}}\text{\ensuremath{\left(\frac{1}{2\gamma_{t}^{(s)}}-\frac{\eta^{4}}{8\beta\left(D^{4}+2\eta^{4}\right)}\right)}}\left(a^{(s)}\right)^{2}\left\Vert g_{t}^{(s)}-g_{t-1}^{(s)}\right\Vert ^{2}\\
\overset{(c)}{\leq} & \frac{8\beta\left(D^{4}+2\eta^{4}\right)}{\eta^{2}}
\end{align*}
where $(a)$ is by $\gamma_{t}^{(s)}\geq\gamma_{t-1}^{(s)}$ and $\left\Vert x_{t}^{(s)}-x^{*}\right\Vert \leq D$,
$(b)$ is by noticing $\gamma_{0}^{(s+1)}=\gamma_{T_{s}}^{(s)}$,
$(c)$ is by Lemma \ref{lem:AdaVRAE-Reminder-Bound-II}.
\end{proof}

\begin{lem}
\label{lem:AdaVRAE-Reminder-Bound-II}Under our update rule of $\gamma_{t}^{(s)}$,
we have

\begin{align*}
\frac{D^{2}}{2}\left(\gamma_{T_{S}}^{(S)}-\gamma_{0}^{(1)}\right)-\frac{D^{4}}{16\beta\left(D^{4}+2\eta^{4}\right)}\sum_{s=1}^{S}\sum_{t=1}^{T_{s}}\left(a^{(s)}\right)^{2}\left\Vert g_{t}^{(s)}-g_{t-1}^{(s)}\right\Vert ^{2} & \le\frac{4\beta\left(D^{4}+2\eta^{4}\right)}{\eta^{2}}
\end{align*}

\begin{align*}
\sum_{s=1}^{S}\sum_{t=1}^{T_{s}}\text{\ensuremath{\left(\frac{1}{2\gamma_{t}^{(s)}}-\frac{\eta^{4}}{8\beta\left(D^{4}+2\eta^{4}\right)}\right)}}\left(a^{(s)}\right)^{2}\left\Vert g_{t}^{(s)}-g_{t-1}^{(s)}\right\Vert ^{2} & \le\frac{4\beta\left(D^{4}+2\eta^{4}\right)}{\eta^{2}}
\end{align*}
\end{lem}
\begin{proof}
For simplicity, $g_{0}^{(1)},g_{1}^{(1)},\dots,g_{T_{1}}^{(1)}=g_{0}^{(2)},g_{1}^{(2)},\dots,g_{T_{2}}^{(2)},\dots$
as $\left(g_{k}\right)_{k\ge0}$ and $\gamma_{0}^{(1)},\gamma_{1}^{(1)},\dots,\gamma_{T_{1}}^{(1)}=\gamma_{0}^{(2)},\gamma_{1}^{(2)},\dots,\gamma_{T_{2}}^{(2)},\dots$
as $\left(\gamma_{k}\right)_{k\ge0}$. For $k\ge1$, assume that $g_{t}^{(s)}$
is the element that correspond to $g_{k}$, and let $a_{k}=a^{(s)}$.
Then we can write $\gamma_{k}=\frac{1}{\eta}\sqrt{\eta^{2}\gamma_{k-1}^{2}+a_{k}^{2}\left\Vert g_{k}-g_{k-1}\right\Vert ^{2}}$.
By writing $\eta^{2}\gamma_{k}^{2}=\eta^{2}\gamma_{k-1}^{2}+a_{k}^{2}\left\Vert g_{k}-g_{k-1}\right\Vert ^{2}$
we obtain $\eta^{2}\gamma_{k}^{2}=\eta^{2}\gamma_{0}^{2}+\sum_{t=1}^{k}a_{t}^{2}\left\Vert g_{t}-g_{t-1}\right\Vert ^{2}$
and hence $\gamma_{k}=\frac{1}{\eta}\sqrt{\eta^{2}\gamma_{0}+\sum_{t=1}^{k}a_{t}^{2}\left\Vert g_{t}-g_{t-1}\right\Vert ^{2}}$.

For 1). Using $\sqrt{a+b}\le\text{\ensuremath{\sqrt{a}+\sqrt{b}}}$
we have $\gamma_{k}\le\gamma_{0}+\frac{1}{\eta}\sqrt{\sum_{t=1}^{k}a_{t}^{2}\left\Vert g_{t}-g_{t-1}\right\Vert ^{2}}$.
Therefore
\begin{align*}
 & \frac{D^{2}}{2}\left(\gamma_{k}-\gamma_{0}\right)-\frac{D^{4}}{16\beta\left(D^{4}+2\eta^{4}\right)}\sum_{t=1}^{k}a_{t}^{2}\left\Vert g_{t}-g_{t-1}\right\Vert ^{2}\\
\le & \frac{D^{2}}{2\eta}\sqrt{\sum_{t=1}^{k}a_{t}^{2}\left\Vert g_{t}-g_{t-1}\right\Vert ^{2}}-\frac{D^{4}}{16\beta\left(D^{4}+2\eta^{4}\right)}\sum_{t=1}^{k}a_{t}^{2}\left\Vert g_{t}-g_{t-1}\right\Vert ^{2}\\
\stackrel{(a)}{\le} & \frac{4\beta\left(D^{4}+2\eta^{4}\right)}{\eta^{2}}
\end{align*}
where for $(a)$ we use $ax-bx^{2}\le\frac{a^{2}}{4b}$.

For 2). Let $\tau$ be the last index such that $\gamma_{\tau}\le\frac{4\beta\left(D^{4}+2\eta^{4}\right)}{\eta^{4}}$
or $\tau=-1$ if $\gamma_{0}>\frac{4\beta\left(D^{4}+2\eta^{4}\right)}{\eta^{4}}$.
If $\tau\le0$ we have $\sum_{t=1}^{k}\left(\frac{1}{2\gamma_{t}}-\frac{\eta^{4}}{8\beta\left(D^{4}+2\eta^{4}\right)}\right)a_{t}^{2}\left\Vert g_{t}-g_{t-1}\right\Vert ^{2}\le0$
for all $k$. Assume $\tau>0$
\begin{align*}
 & \sum_{t=1}^{k}\left(\frac{1}{2\gamma_{t}}-\frac{\eta^{4}}{8\beta\left(D^{4}+2\eta^{4}\right)}\right)a_{t}^{2}\left\Vert g_{t}-g_{t-1}\right\Vert ^{2}\\
\le & \sum_{t=1}^{\tau}\frac{1}{2\gamma_{t}}a_{t}^{2}\left\Vert g_{t}-g_{t-1}\right\Vert ^{2}\\
= & \eta^{2}\sum_{t=1}^{\tau}\frac{\gamma_{t}^{2}-\gamma_{t-1}^{2}}{2\gamma_{t}}\\
= & \eta^{2}\sum_{t=1}^{\tau}\frac{\left(\gamma_{t}-\gamma_{t-1}\right)\left(\gamma_{t}+\gamma_{t-1}\right)}{2\gamma_{t}}\\
\stackrel{(b)}{\le} & \eta^{2}\sum_{t=1}^{\tau}\left(\gamma_{t}-\gamma_{t-1}\right)\\
\leq & \eta^{2}\gamma_{\tau}\\
\overset{(c)}{\leq} & \frac{4\beta\left(D^{4}+2\eta^{4}\right)}{\eta^{2}}
\end{align*}
where $(b)$ is due to $\gamma_{t-1}\le\gamma_{t}$, $(c)$ is by
the definition of $\tau$.
\end{proof}

Finally we give an explicit choice for the parameters to satisfy all
conditions and give the final necessary bound.

\subsection{Parameter choice and bound}

The following lemma states the bound for the coefficients.
\begin{lem}
\label{lem:AdaVRAE-Bound}Under the choice of parameters in Theorem
\ref{thm:AdaVRAE-Convergence-A}, $\forall s\geq1$, we have
\[
\left(a^{(s)}\right)^{2}<4A_{0}^{(s)}
\]
and
\[
A_{T_{s}}^{(s)}\ge\begin{cases}
n(4n)^{-0.5^{s}} & 1\leq s\le s_{0}\\
\frac{n}{4c}(s-s_{0})^{2} & s_{0}<s
\end{cases}
\]
\end{lem}
\begin{proof}
As a reminder, we choose the parameters as follows, where $c=\frac{3}{2}$
and $s=s_{0}=\left\lceil \log_{2}\log_{2}4n\right\rceil $

\begin{align*}
a^{(s)} & =\begin{cases}
(4n)^{-0.5^{s}} & 1\leq s\leq s_{0}\\
\frac{s-s_{0}-1+c}{2c} & s_{0}<s
\end{cases},\\
T_{s} & =n,\\
A_{T_{0}}^{(0)} & =\frac{5}{4}.
\end{align*}
The idea in this choice is that we divide the time into two phases
in which the convergence behaves differently. In the first phase,
$A_{T_{s}}^{(s)}$ quickly gets to $\Omega(n)$ and we can set the
coefficients for the checkpoint relatively small. In the second phase,
to achieve the optimal $\sqrt{\frac{n\beta}{\epsilon}}$ rate, $A_{T_{s}}^{(s)}=\Omega(n^{2})$.
In this phase, we need to be more conservative and set the coefficients
for the checkpoint large. We analyze the two phases separately.

First we show by induction that for $1\leq s\le s_{0}$,
\begin{align}
A_{0}^{(s)} & =1+n\sum_{k=0}^{s-2}(4n)^{-0.5^{k}},\label{eq:AdaVRAE-6}\\
A_{T_{s}}^{(s)} & =1+n\sum_{k=0}^{s}(4n)^{-0.5^{k}}.\label{eq:AdaVRAE-7}
\end{align}
Indeed, we have
\begin{align*}
A_{0}^{(1)} & =A_{T_{0}}^{(0)}-T_{1}\left(a^{(1)}\right)^{2}\overset{(a)}{=}\frac{5}{4}-n(4n)^{-1}=\frac{5}{4}-\frac{1}{4}=1,\\
A_{T_{1}}^{(1)} & =A_{0}^{(1)}+T_{1}\left(a^{(1)}+\left(a^{(1)}\right)^{2}\right)\overset{(b)}{=}1+n\left((4n)^{-0.5}+(4n)^{-1}\right),
\end{align*}
where $(a)$ and $(b)$ are both by plugging in $a^{(1)}=(4n)^{-0.5}$
and $T_{1}=n$. Supposed that \ref{eq:AdaVRAE-6} and \ref{eq:AdaVRAE-7}
hold for all $k\le s<s_{0}$. For $k=s+1\le s_{0}$, we have
\begin{align*}
A_{0}^{(s+1)} & =A_{T_{s}}^{(s)}-T_{s+1}\left(a^{(s+1)}\right)^{2}\\
 & \overset{(c)}{=}\left(1+n\sum_{k=0}^{s}(4n)^{-0.5^{k}}\right)-n(4n)^{-0.5^{s}}\\
 & =1+n\sum_{k=0}^{s-1}(4n)^{-0.5^{k}},\\
A_{T_{s+1}}^{(s+1)} & =A_{0}^{(s+1)}+T_{s+1}\left(a^{(s+1)}+\left(a^{(s+1)}\right)^{2}\right)\\
 & \overset{(d)}{=}\left(1+n\sum_{k=0}^{s-1}(4n)^{-0.5^{k}}\right)+n\left((4n)^{-0.5^{s+1}}+(4n)^{-0.5^{s}}\right)\\
 & =1+n\sum_{k=0}^{s+1}(4n)^{-0.5^{k}},
\end{align*}
where $(c)$ is by plugging $a^{(s+1)}=(4n)^{-0.5^{s+1}}$, $T_{s+1}=n$
and the assumption on $A_{T_{s}}^{(s)}$, $(d)$ is by plugging $a^{(s+1)}=(4n)^{-0.5^{s+1}}$
and $T_{s+1}=n$. Now the induction is completed. From this we can
see that $A_{0}^{(s)}\geq1>\frac{\left(a^{(s)}\right)^{2}}{4}$ and
$A_{T_{s}}^{(s)}>n(4n)^{-0.5^{s}}$.

Next, for $s>s_{0}$, we show by induction that
\begin{align}
A_{0}^{(s)} & >\frac{n}{2}+\frac{n}{4c}(s-s_{0}-2+2c)(s-s_{0}-1)-\frac{n}{4c^{2}}(s-s_{0}-1+c)^{2},\label{eq:AdaVRAE-8}\\
A_{T_{s}}^{(s)} & >\frac{n}{2}+\frac{n}{4c}(s-s_{0}-1+2c)(s-s_{0}).\label{eq:AdaVRAE-9}
\end{align}
Indeed we have $A_{T_{s_{0}}}^{(s_{0})}=1+n\sum_{k=0}^{s_{0}}(4n)^{-0.5^{k}}>n(4n)^{-0.5^{s_{0}}}\geq n(4n)^{-0.5^{\log_{2}\log_{2}4n}}=\frac{n}{2}.$
Hence
\begin{align*}
A_{0}^{(s_{0}+1)} & =A_{T_{s_{0}}}^{(s_{0})}-T_{s_{0}+1}\left(a^{(s_{0}+1)}\right)^{2}\\
 & \overset{(e)}{\geq}\frac{n}{2}-\frac{n}{4},\\
A_{T_{s_{0}+1}}^{(s_{0}+1)} & =A_{0}^{(s_{0}+1)}+T_{s_{0}+1}\left(\left(a^{(s_{0}+1)}\right)+\left(a^{(s_{0}+1)}\right)^{2}\right)\\
 & \overset{(f)}{\geq}\frac{n}{2}+n\left(\frac{1}{2}+\frac{1}{4}\right)\\
 & >\frac{n}{2}+\frac{n}{2},
\end{align*}
where $(e)$ and $(f)$ are both by $a^{(s_{0}+1)}=\frac{1}{2}$,
$T_{s_{0}+1}=n$. Supposed that \ref{eq:AdaVRAE-8} and \ref{eq:AdaVRAE-9}
hold for all $s_{0}<k\le s$. For $k=s+1$ we have
\begin{align*}
A_{0}^{(s+1)} & =A_{T_{s}}^{(s)}-T_{s+1}\left(a^{(s+1)}\right)^{2}\\
 & \overset{(g)}{>}\frac{n}{2}+\frac{n}{4c}(s-s_{0}-1+2c)(s-s_{0})-n\left(\frac{s-s_{0}+c}{2c}\right)^{2}\\
 & =\frac{n}{2}+\frac{n}{4c}(s-s_{0}-1+2c)(s-s_{0})-\frac{n}{4c^{2}}(s-s_{0}+c)^{2},\\
A_{T_{s+1}}^{(s+1)} & =A_{0}^{(s+1)}+T_{s+1}\left(a^{(s+1)}+\left(a^{(s+1)}\right)^{2}\right)\\
 & =A_{T_{s}}^{(s)}+T_{s+1}a^{(s+1)}\\
 & \overset{(h)}{>}\frac{n}{2}+\frac{n}{4c}(s-s_{0}-1+2c)(s-s_{0})+\frac{n}{2c}\left(s-s_{0}+c\right)\\
 & =\frac{n}{2}+\frac{n}{4c}(s-s_{0}+2c)(s-s_{0}+1),
\end{align*}
where $(g)$ and $(h)$ are both due to $T_{s+1}=n$, $a^{(s+1)}=\frac{s-s_{0}+c}{2c}$
and the assumption on $A_{T_{s}}^{(s)}$. Now the induction is completed.
We can see that if $c=\frac{3}{2}$. we have 
\begin{align*}
A_{0}^{(s)} & >n\left(\frac{1}{2}+\frac{(s-s_{0}-1+c)^{2}}{4c}\left(1-\frac{1}{c}\right)-\frac{s-s_{0}-1+c^{2}}{4c}\right)\\
 & =n\left(\frac{1}{2}+\frac{(s-s_{0}-1+c)^{2}}{12c}-\frac{s-s_{0}-1+c^{2}}{4c}\right)\\
 & =n\left(\frac{1}{2}+\frac{(s-s_{0}-1+c)^{2}}{16c^{2}}+\frac{(s-s_{0}-1+c)^{2}}{24c}-\frac{s-s_{0}-1+c^{2}}{4c}\right)\\
 & =n\left(\frac{(s-s_{0}-1+c)^{2}}{16c^{2}}+\frac{(s-s_{0}-1)^{2}-3(s-s_{0}-1)+(c^{2}-6c^{2}+12c)}{24c}\right)\\
 & >\frac{(s-s_{0}-1+c)^{2}}{16c^{2}}=\frac{\left(a^{(s)}\right)^{2}}{4}
\end{align*}
 and $A_{T_{s}}^{(s)}>\frac{n}{4c}(s-s_{0})^{2}.$
\end{proof}

\subsection{Putting all together}

We are now ready to put everything together and complete the proof
of Theorem \ref{thm:AdaVRAE-Convergence-A}.

\begin{proof}
(Theorem \ref{thm:AdaVRAE-Convergence-A}) From Lemma \ref{lem:AdaVRAE-Bound},
we know$\left(a^{(s)}\right)^{2}<4A_{0}^{(s)}$ for any $s\geq1$,
which implies for any $s\geq1$, $t\in\left[T_{s}\right]$
\[
\left(a^{(s)}\right)^{2}<4A_{t-1}^{(s)}.
\]
 Combining our parameters, we can find the requirements for Lemma
\ref{lem:AdaVRAE-Whole-Descent} are satisfied, which will give us
\begin{align*}
\E\left[A_{T_{S}}^{(S)}\left(F(u^{(S)})-F(x^{*})\right)\right] & \leq A_{T_{0}}^{(0)}\left(F(u^{(0)})-F(x^{*})\right)+\frac{\gamma}{2}\left\Vert u^{(0)}-x^{*}\right\Vert ^{2}\\
 & \quad+\E\left[\sum_{s=1}^{S}\sum_{t=1}^{T_{s}}\frac{\gamma_{t}^{(s)}-\gamma_{t-1}^{(s)}}{2}\left\Vert x_{t}^{(s)}-x^{*}\right\Vert ^{2}+\left(\frac{\left(a^{(s)}\right)^{2}}{2\gamma_{t}^{(s)}}-\frac{\left(a^{(s)}\right)^{2}}{16\beta}\right)\left\Vert g_{t}^{(s)}-g_{t-1}^{(s)}\right\Vert ^{2}\right].
\end{align*}
By using Lemma \ref{lem:AdaVRAE-Reminder-Bound-I}, we know
\begin{align*}
\E\left[A_{T_{S}}^{(S)}\left(F(u^{(S)})-F(x^{*})\right)\right] & \leq A_{T_{0}}^{(0)}\left(F(u^{(0)})-F(x^{*})\right)+\frac{\gamma}{2}\left\Vert u^{(0)}-x^{*}\right\Vert ^{2}+\frac{8\beta\left(D^{4}+2\eta^{4}\right)}{\eta^{2}}\\
 & \overset{(a)}{\leq}\frac{5}{4}\left(F(u^{(0)})-F(x^{*})\right)+\frac{\gamma}{2}\left\Vert u^{(0)}-x^{*}\right\Vert ^{2}+\frac{8\beta\left(D^{4}+2\eta^{4}\right)}{\eta^{2}}\\
\Rightarrow\E\left[F(u^{(S)})-F(x^{*})\right] & \leq\frac{V}{2A_{T_{S}}^{(S)}}\\
 & \overset{(b)}{\leq}\begin{cases}
\frac{2V}{\left(4n\right)^{1-0.5^{S}}} & 1\leq S\leq s_{0}\\
\frac{2cV}{n(S-s_{0})^{2}} & s_{0}<S
\end{cases}
\end{align*}
where $(a)$ is by plugging in $A_{T_{0}}^{(0)}=\frac{5}{4}$, $(b)$
is by \ref{lem:AdaVRAE-Bound}.
\end{proof}
\begin{itemize}
\item If $\epsilon\ge\frac{V}{n}$, we choose $S=\left\lceil \log_{2}\log_{2}\frac{4V}{\epsilon}\right\rceil \leq\left\lceil \log_{2}\log_{2}4n\right\rceil =s_{0}$,
so we have
\begin{align*}
\E\left[F(u^{(S)})-F(x^{*})\right] & \leq\frac{2V}{(4n)^{1-0.5^{S}}}\\
 & \overset{(c)}{\leq}\frac{2V}{\left(\frac{4V}{\epsilon}\right){}^{1-0.5^{S}}}\\
 & =\frac{\epsilon}{2\left(\frac{4V}{\epsilon}\right){}^{-0.5^{S}}}\\
 & \overset{(d)}{\leq}\epsilon,
\end{align*}
where $(c)$ is by $n\geq\frac{V}{\epsilon}$, $(d)$ is by $\left(\frac{4V}{\epsilon}\right){}^{-0.5^{S}}=\left(\frac{4V}{\epsilon}\right){}^{-0.5^{\left\lceil \log_{2}\log_{2}\frac{4V}{\epsilon}\right\rceil }}\geq\left(\frac{4V}{\epsilon}\right){}^{-0.5^{\log_{2}\log_{2}\frac{4V}{\epsilon}}}=\frac{1}{2}$.
Note that the final full gradient computation in the last epoch is
not needed, therefore the number of individual gradient evaluations
is
\begin{align*}
\#grads & =n+\sum_{s=1}^{S-1}\left(2(T_{s}-1)+n\right)+2(T_{S}-1)\\
 & <3nS\\
 & =3n\left\lceil \log_{2}\log_{2}\frac{4V}{\epsilon}\right\rceil \\
 & =\mathcal{O}\left(n\log\log\frac{V}{\epsilon}\right).
\end{align*}
\item If $\epsilon<\frac{V}{n}$, we choose $S=s_{0}+\left\lceil \sqrt{\frac{2cV}{n\epsilon}}\right\rceil \geq s_{0}+1$,
so we have
\begin{align*}
\E\left[F(u^{(S)})-F(x^{*})\right] & \leq\frac{2cV}{n(S-s_{0})^{2}}\\
 & =\frac{2cV}{n\left(\left\lceil \sqrt{\frac{2cV}{n\epsilon}}\right\rceil \right)^{2}}\\
 & \leq\frac{2cV}{n\left(\sqrt{\frac{2cV}{n\epsilon}}\right)^{2}}\\
 & =\epsilon.
\end{align*}
The number of individual gradient evaluations is
\begin{align*}
\#grads & =n+\sum_{s=1}^{S-1}\left(2(T_{s}-1)+n\right)+2(T_{S}-1)\\
 & <3nS\\
 & =3ns_{0}+3n(S-s_{0})\\
 & =3n\left\lceil \log_{2}\log_{2}4n\right\rceil +3n\left\lceil \sqrt{\frac{2cV}{n\epsilon}}\right\rceil \\
 & =\mathcal{O}\left(n\log\log n+\sqrt{\frac{nV(z)}{\epsilon}}\right).
\end{align*}
\end{itemize}

\section{Analysis of algorithm \ref{alg:AdaVRAG}}

\label{sec:AdaVRAG}

In this section, we analyze Algorithm \ref{alg:AdaVRAG} and prove
the following convergence guarantee:
\begin{thm}
\label{thm:AdaVRAG-Convergence-A}(Convergence of AdaVRAG) Define
$s_{0}=\lceil\log_{2}\log_{2}4n\rceil$, $c=\frac{3+\sqrt{33}}{4}$.
Suppose we set the parameters of Algorithm \ref{alg:AdaVRAG} as follows:
\begin{align*}
a^{(s)} & =\begin{cases}
1-\left(4n\right)^{-0.5^{s}} & 1\leq s\leq s_{0}\\
\frac{c}{s-s_{0}+2c} & s_{0}<s
\end{cases},\\
q^{(s)} & =\begin{cases}
\frac{1}{\left(1-a^{(s)}\right)a^{(s)}} & 1\leq s\leq s_{0}\\
\frac{8\left(2-a^{(s)}\right)a^{(s)}}{3(1-a^{(s)})} & s_{0}<s
\end{cases},\\
T_{s} & =n.
\end{align*}
Suppose that $\dom$ is a compact convex set with diameter $D$
and we set $\eta=\Theta(D)$. Addtionally, we assume that $2\eta^{2}>D^{2}$
if $\text{Option I}$ is used for setting the step size. The number
of individual gradient evaluations to achieve a solution $u^{(S)}$
such that $\E\left[F(u^{(S)})-F(x^{*})\right]\leq\epsilon$ for Algorithm
\ref{alg:AdaVRAG} is 
\[
\#grads=\begin{cases}
\mathcal{O}\left(n\log\log\frac{V}{\epsilon}\right) & \epsilon\geq\frac{V}{n}\\
\mathcal{O}\left(n\log\log n+\sqrt{\frac{nV}{\epsilon}}\right) & \epsilon<\frac{V}{n}
\end{cases},
\]
where
\[
V=\begin{cases}
\frac{1}{2}(F(u^{(0)})-F(x^{*}))+\gamma\left\Vert u^{(0)}-x^{*}\right\Vert ^{2}+\left[\beta-\left(1-\frac{D^{2}}{2\eta^{2}}\right)\gamma\right]^{+}\left(D^{2}+2(\eta^{2}+D^{2})\log\frac{\frac{2\eta^{2}\beta}{2\eta^{2}-D^{2}}}{\gamma}\right) & \text{for Option I}\\
\frac{1}{2}(F(u^{(0)})-F(x^{*}))+\gamma\left\Vert u^{(0)}-x^{*}\right\Vert ^{2}+\eta^{2}\left(\frac{D^{2}}{\eta^{2}}+\beta-\gamma\right)^{+}\left(\frac{2D^{2}}{\eta^{2}}+\beta-\gamma\right) & \text{for Option II}
\end{cases}.
\]
\end{thm}

\subsection{Single epoch progress and final output}

We first analyze the progress in function value made in a single iteration
of an epoch. The analysis is done in a standard way by combining the
smoothness and convexity of $f$, the convexity of $h$ and the optimality
condition of $x_{t}^{(s)}$. 
\begin{lem}
\label{lem:AdaVRAG-One-Step-Descent}For all epochs $s\geq1$ and
all iterations $t\in\left[T_{s}\right]$, we have 
\begin{align*}
\E\left[F(\avx_{t}^{(s)})-F(x^{*})\right] & \leq\E\left[\left(1-a^{(s)}\right)\left(F(u^{(s-1)})-F(x^{*})\right)\right]\\
 & \quad+\E\left[\frac{\gamma_{t-1}^{(s)}q^{(s)}a^{(s)}}{2}\left(\left\Vert x_{t-1}^{(s)}-x^{*}\right\Vert ^{2}-\left\Vert x_{t}^{(s)}-x^{*}\right\Vert ^{2}\right)\right]\\
 & \quad+\E\left[\left(\frac{\beta\left(2-a^{(s)}\right)\left(a^{(s)}\right)^{2}}{2\left(1-a^{(s)}\right)}-\frac{\gamma_{t-1}^{(s)}q^{(s)}a^{(s)}}{2}\right)\left\Vert x_{t}^{(s)}-x_{t-1}^{(s)}\right\Vert ^{2}\right].
\end{align*}
\end{lem}
\begin{proof}
We have
\begin{align*}
 & \E\left[f(\avx_{t}^{(s)})-f(\avx_{t-1}^{(s)})\right]\\
\overset{(a)}{\leq} & \E\left[\left\langle \nabla f(\avx_{t-1}^{(s)}),\avx_{t}^{(s)}-\avx_{t-1}^{(s)}\right\rangle +\frac{\beta}{2}\left\Vert \avx_{t}^{(s)}-\avx_{t-1}^{(s)}\right\Vert ^{2}\right]\\
= & \E\left[\left\langle g_{t}^{(s)},\avx_{t}^{(s)}-\avx_{t-1}^{(s)}\right\rangle +\left\langle \nabla f(\avx_{t-1}^{(s)})-g_{t}^{(s)},\avx_{t}^{(s)}-\avx_{t-1}^{(s)}\right\rangle +\frac{\beta}{2}\left\Vert \avx_{t}^{(s)}-\avx_{t-1}^{(s)}\right\Vert ^{2}\right],
\end{align*}
where $(a)$ is due to $f$ being $\beta$-smooth. Using Cauchy--Schwarz
inequality and Young's inequality ($ab\leq\frac{\lambda}{2}a^{2}+\frac{1}{2\lambda}b^{2}$
with $\lambda>0$) we have

\begin{align*}
 & \left\langle \nabla f(\avx_{t-1}^{(s)})-g_{t}^{(s)},\avx_{t}^{(s)}-\avx_{t-1}^{(s)}\right\rangle \\
\le & \left\Vert \nabla f(\avx_{t-1}^{(s)})-g_{t}^{(s)}\right\Vert \left\Vert \avx_{t}^{(s)}-\avx_{t-1}^{(s)}\right\Vert \\
\le & \frac{1-a^{(s)}}{2\beta}\left\Vert \nabla f(\avx_{t-1}^{(s)})-g_{t}^{(s)}\right\Vert ^{2}+\frac{\beta}{2(1-a^{(s)})}\left\Vert \avx_{t}^{(s)}-\avx_{t-1}^{(s)}\right\Vert ^{2},
\end{align*}
also note that 
\begin{align*}
\avx_{t}^{(s)}-\avx_{t-1}^{(s)} & =\left(a^{(s)}x_{t}^{(s)}+(1-a^{(s)})u^{(s-1)}\right)-\left(a^{(s)}x_{t-1}^{(s)}+(1-a^{(s)})u^{(s-1)}\right)=a^{(s)}\left(x_{t}^{(s)}-x_{t-1}^{(s)}\right).
\end{align*}
Hence, we obtain

\begin{align}
 & \E\left[f(\avx_{t}^{(s)})-f(\avx_{t-1}^{(s)})\right]\nonumber \\
\leq & \E\left[\left\langle g_{t}^{(s)},a^{(s)}\left(x_{t}^{(s)}-x_{t-1}^{(s)}\right)\right\rangle +\frac{1-a^{(s)}}{2\beta}\left\Vert \nabla f(\avx_{t-1}^{(s)})-g_{t}^{(s)}\right\Vert ^{2}+\frac{\beta\left(2-a^{(s)}\right)\left(a^{(s)}\right)^{2}}{2\left(1-a^{(s)}\right)}\left\Vert x_{t}^{(s)}-x_{t-1}^{(s)}\right\Vert ^{2}\right]\nonumber \\
\overset{(b)}{\leq} & \E\left[\left\langle g_{t}^{(s)},a^{(s)}\left(x_{t}^{(s)}-x_{t-1}^{(s)}\right)\right\rangle +\left(1-a^{(s)}\right)\left(f(u^{(s-1)})-f(\avx_{t-1}^{(s)})-\langle\nabla f(\avx_{t-1}^{(s)}),u^{(s-1)}-\avx_{t-1}^{(s)}\rangle\right)\right]\nonumber \\
 & +\E\left[\frac{\beta\left(2-a^{(s)}\right)\left(a^{(s)}\right)^{2}}{2\left(1-a^{(s)}\right)}\left\Vert x_{t}^{(s)}-x_{t-1}^{(s)}\right\Vert ^{2}\right]\nonumber \\
= & \E\left[\left\langle g_{t}^{(s)},a^{(s)}\left(x_{t}^{(s)}-x^{*}\right)\right\rangle +\left\langle g_{t}^{(s)},a^{(s)}\left(x^{*}-x_{t-1}^{(s)}\right)\right\rangle -\left\langle \nabla f(\avx_{t-1}^{(s)}),\left(1-a^{(s)}\right)\left(u^{(s-1)}-\avx_{t-1}^{(s)}\right)\right\rangle \right]\nonumber \\
 & +\E\left[\left(1-a^{(s)}\right)\left(f(u^{(s-1)})-f(\avx_{t-1}^{(s)})\right)\right]+\E\left[\frac{\beta\left(2-a^{(s)}\right)\left(a^{(s)}\right)^{2}}{2\left(1-a^{(s)}\right)}\left\Vert x_{t}^{(s)}-x_{t-1}^{(s)}\right\Vert ^{2}\right]\nonumber \\
\overset{(c)}{=} & \E\left[\left\langle g_{t}^{(s)},a^{(s)}\left(x_{t}^{(s)}-x^{*}\right)\right\rangle +\left\langle \nabla f(\avx_{t-1}^{(s)}),a^{(s)}\left(x^{*}-x_{t-1}^{(s)}\right)-\left(1-a^{(s)}\right)\left(u^{(s-1)}-\avx_{t-1}^{(s)}\right)\right\rangle \right]\nonumber \\
 & +\E\left[\left(1-a^{(s)}\right)\left(f(u^{(s-1)})-f(\avx_{t-1}^{(s)})\right)\right]+\E\left[\frac{\beta\left(2-a^{(s)}\right)\left(a^{(s)}\right)^{2}}{2\left(1-a^{(s)}\right)}\left\Vert x_{t}^{(s)}-x_{t-1}^{(s)}\right\Vert ^{2}\right]\nonumber \\
\overset{(d)}{=} & \E\left[\left\langle g_{t}^{(s)},a^{(s)}\left(x_{t}^{(s)}-x^{*}\right)\right\rangle +\left\langle \nabla f(\avx_{t-1}^{(s)}),a^{(s)}\left(x^{*}-\avx_{t-1}^{(s)}\right)\right\rangle +\left(1-a^{(s)}\right)\left(f(u^{(s-1)})-f(\avx_{t-1}^{(s)})\right)\right]\nonumber \\
 & +\E\left[\frac{\beta\left(2-a^{(s)}\right)\left(a^{(s)}\right)^{2}}{2\left(1-a^{(s)}\right)}\left\Vert x_{t}^{(s)}-x_{t-1}^{(s)}\right\Vert ^{2}\right]\nonumber \\
\overset{(e)}{\leq} & \E\left[\left\langle g_{t}^{(s)},a^{(s)}\left(x_{t}^{(s)}-x^{*}\right)\right\rangle +\frac{\beta\left(2-a^{(s)}\right)\left(a^{(s)}\right)^{2}}{2\left(1-a^{(s)}\right)}\left\Vert x_{t}^{(s)}-x_{t-1}^{(s)}\right\Vert ^{2}\right.\nonumber \\
 & \qquad\left.+\left(1-a^{(s)}\right)f(u^{(s-1)})+a^{(s)}f(x^{*})-f(\avx_{t-1}^{(s)})\right],\label{eq:varag-1}
\end{align}
where $(b)$ is by Lemma \ref{lem:Variance-Reduction}, $(c)$ is
because of
\[
\E\left[\left\langle g_{t}^{(s)},a^{(s)}\left(x^{*}-x_{t-1}^{(s)}\right)\right\rangle \right]=\E\left[\left\langle \nabla f(\avx_{t-1}^{(s)}),a^{(s)}\left(x^{*}-x_{t-1}^{(s)}\right)\right\rangle \right],
\]
$(d)$ is by $\avx_{t-1}^{(s)}=a^{(s)}x_{t-1}^{(s)}+\left(1-a^{(s)}\right)u^{(s-1)}$,
$(e)$ is due to the convexity of $f$ which implies 
\[
\left\langle \nabla f(\avx_{t-1}^{(s)}),a^{(s)}\left(x^{*}-\avx_{t-1}^{(s)}\right)\right\rangle \leq a^{(s)}\left(f(x^{*})-f(\avx_{t-1}^{(s)})\right).
\]
By adding $\E\left[f(\avx_{t-1}^{(s)})-f(x^{*})\right]$ to both sides
of (\ref{eq:varag-1}), we obtain
\begin{align}
 & \E\left[f(\avx_{t}^{(s)})-f(x^{*})\right]\nonumber \\
\text{\ensuremath{\leq}} & \E\left[\left\langle g_{t}^{(s)},a^{(s)}\left(x_{t}^{(s)}-x^{*}\right)\right\rangle +\frac{\beta\left(2-a^{(s)}\right)\left(a^{(s)}\right)^{2}}{2\left(1-a^{(s)}\right)}\left\Vert x_{t}^{(s)}-x_{t-1}^{(s)}\right\Vert ^{2}+\left(1-a^{(s)}\right)\left(f(u^{(s-1)})-f(x^{*})\right)\right].\label{eq:varag-2}
\end{align}
Next, we upper bound the inner product $\left\langle g_{t}^{(s)},a^{(s)}\left(x_{t}^{(s)}-x^{*}\right)\right\rangle $.
By the optimality condition of $x_{t}^{(s)}$, we have
\[
\left\langle g_{t}^{(s)}+h'(x_{t}^{(s)})+\gamma_{t-1}^{(s)}q^{(s)}\left(x_{t}^{(s)}-x_{t-1}^{(s)}\right),x_{t}^{(s)}-x^{*}\right\rangle \leq0,
\]
 where $h'(x_{t}^{(s)})\in\partial h(x_{t}^{(s)})$ is a subgradient
of $h$ at $x_{t}^{(s)}$. We rearrange the above inequality and obtain
\begin{align}
 & a^{(s)}\left\langle g_{t}^{(s)},x_{t}^{(s)}-x^{*}\right\rangle \nonumber \\
\leq & a^{(s)}\left\langle h'(x_{t}^{(s)})+\gamma_{t-1}^{(s)}q^{(s)}\left(x_{t}^{(s)}-x_{t-1}^{(s)}\right),x^{*}-x_{t}^{(s)}\right\rangle \nonumber \\
\overset{(f)}{\leq} & a^{(s)}\left(h(x^{*})-h(x_{t}^{(s)})\right)+a^{(s)}\gamma_{t-1}^{(s)}q^{(s)}\left\langle x_{t}^{(s)}-x_{t-1}^{(s)},x^{*}-x_{t}^{(s)}\right\rangle \nonumber \\
\overset{(g)}{=} & a^{(s)}\left(h(x^{*})-h(x_{t}^{(s)})\right)+\frac{a^{(s)}\gamma_{t-1}^{(s)}q^{(s)}}{2}\left(\left\Vert x_{t-1}^{(s)}-x^{*}\right\Vert ^{2}-\left\Vert x_{t}^{(s)}-x^{*}\right\Vert ^{2}-\left\Vert x_{t-1}^{(s)}-x_{t}^{(s)}\right\Vert ^{2}\right),\label{eq:varag-3}
\end{align}
 where $(f)$ follows from the convexity of $h$ and the fact that
$h'(x_{t}^{(s)})\in\partial h(x_{t}^{(s)})$, and $(g)$ is due to
the identity $\langle a,b\rangle=\frac{1}{2}\left(\left\Vert a+b\right\Vert ^{2}-\left\Vert a\right\Vert ^{2}-\left\Vert b\right\Vert ^{2}\right)$.

We plug in (\ref{eq:varag-3}) into (\ref{eq:varag-2}), and obtain
\begin{align*}
 & \E\left[f(\avx_{t}^{(s)})-f(x^{*})\right]\\
\leq & \E\left[\left(1-a^{(s)}\right)\left(f(u^{(s-1)})-f(x^{*})\right)+a^{(s)}\left(h(x^{*})-h(x_{t}^{(s)})\right)\right]\\
 & +\E\left[\frac{\gamma_{t-1}^{(s)}q^{(s)}a^{(s)}}{2}\left(\left\Vert x_{t-1}^{(s)}-x^{*}\right\Vert ^{2}-\left\Vert x_{t}^{(s)}-x^{*}\right\Vert ^{2}\right)+\left(\frac{\beta\left(2-a^{(s)}\right)\left(a^{(s)}\right)^{2}}{2\left(1-a^{(s)}\right)}-\frac{\gamma_{t-1}^{(s)}q^{(s)}a^{(s)}}{2}\right)\left\Vert x_{t}^{(s)}-x_{t-1}^{(s)}\right\Vert ^{2}\right]\\
\overset{(h)}{=} & \E\left[\left(1-a^{(s)}\right)\left(F(u^{(s-1)})-F(x^{*})\right)+h(x^{*})-a^{(s)}h(x_{t}^{(s)})-\left(1-a^{(s)}\right)h(u^{(s-1)})\right]\\
 & +\E\left[\frac{\gamma_{t-1}^{(s)}q^{(s)}a^{(s)}}{2}\left(\left\Vert x_{t-1}^{(s)}-x^{*}\right\Vert ^{2}-\left\Vert x_{t}^{(s)}-x^{*}\right\Vert ^{2}\right)+\left(\frac{\beta\left(2-a^{(s)}\right)\left(a^{(s)}\right)^{2}}{2\left(1-a^{(s)}\right)}-\frac{\gamma_{t-1}^{(s)}q^{(s)}a^{(s)}}{2}\right)\left\Vert x_{t}^{(s)}-x_{t-1}^{(s)}\right\Vert ^{2}\right]\\
\overset{(i)}{\leq} & \E\left[\left(1-a^{(s)}\right)\left(F(u^{(s-1)})-F(x^{*})\right)+h(x^{*})-h(\avx_{t}^{(s)})\right]\\
 & +\E\left[\frac{\gamma_{t-1}^{(s)}q^{(s)}a^{(s)}}{2}\left(\left\Vert x_{t-1}^{(s)}-x^{*}\right\Vert ^{2}-\left\Vert x_{t}^{(s)}-x^{*}\right\Vert ^{2}\right)+\left(\frac{\beta\left(2-a^{(s)}\right)\left(a^{(s)}\right)^{2}}{2\left(1-a^{(s)}\right)}-\frac{\gamma_{t-1}^{(s)}q^{(s)}a^{(s)}}{2}\right)\left\Vert x_{t}^{(s)}-x_{t-1}^{(s)}\right\Vert ^{2}\right],
\end{align*}
where $(h)$ is by the definition of $F=f+h$, and $(i)$ is by the
convexity of $h$ which implies 
\[
h(\avx_{t}^{(s)})=h\left(a^{(s)}x_{t}^{(s)}+(1-a^{(s)})u^{(s-1)}\right)\leq a^{(s)}h(x_{t}^{(s)})+(1-a^{(s)})h(u^{(s-1)}).
\]
Now we move the term $\E\left[h(x^{*})-h(\avx_{t}^{(s)})\right]$
to the LHS, and obtain
\begin{align*}
 & \E\left[F(\avx_{t}^{(s)})-F(x^{*})\right]\\
\leq & \E\left[\left(1-a^{(s)}\right)\left(F(u^{(s-1)})-F(x^{*})\right)+\frac{\gamma_{t-1}^{(s)}q^{(s)}a^{(s)}}{2}\left(\left\Vert x_{t-1}^{(s)}-x^{*}\right\Vert ^{2}-\left\Vert x_{t}^{(s)}-x^{*}\right\Vert ^{2}\right)\right]\\
 & +\E\left[\left(\frac{\beta\left(2-a^{(s)}\right)\left(a^{(s)}\right)^{2}}{2\left(1-a^{(s)}\right)}-\frac{\gamma_{t-1}^{(s)}q^{(s)}a^{(s)}}{2}\right)\left\Vert x_{t}^{(s)}-x_{t-1}^{(s)}\right\Vert ^{2}\right].
\end{align*}
\end{proof}

By Lemma \ref{lem:AdaVRAG-One-Step-Descent}, if $\frac{1}{T_{s}}\sum_{t=1}^{T_{s}}\avx_{t}^{(s)}$
is defined as a new chekpoint like what we do in Algorithm \ref{alg:AdaVRAG},
the following guarantee for the function value progress in one epoch
comes up immediately by the convexity of $F$.
\begin{lem}
\label{lem:AdaVRAG-One-Epoch-Descent}For all epochs $s\geq1$, we
have
\begin{align*}
\E\left[F(u^{(s)})-F(x^{*})\right] & \leq\E\left[\left(1-a^{(s)}\right)\left(F(u^{(s-1)})-F(x^{*})\right)\right]\\
 & \quad+\E\left[\frac{1}{T_{s}}\sum_{t=1}^{T_{s}}\frac{\gamma_{t-1}^{(s)}q^{(s)}a^{(s)}}{2}\left(\left\Vert x_{t-1}^{(s)}-x^{*}\right\Vert ^{2}-\left\Vert x_{t}^{(s)}-x^{*}\right\Vert ^{2}\right)\right]\\
 & \quad+\E\left[\frac{1}{T_{s}}\sum_{t=1}^{T_{s}}\left(\frac{\beta\left(2-a^{(s)}\right)\left(a^{(s)}\right)^{2}}{2\left(1-a^{(s)}\right)}-\frac{\gamma_{t-1}^{(s)}q^{(s)}a^{(s)}}{2}\right)\left\Vert x_{t}^{(s)}-x_{t-1}^{(s)}\right\Vert ^{2}\right].
\end{align*}
\end{lem}
\begin{proof}
We have
\begin{align*}
 & \E\left[F(u^{(s)})-F(x^{*})\right]\\
\overset{(a)}{\leq} & \E\left[\frac{1}{T_{s}}\sum_{t=1}^{T_{s}}\left(F(\avx_{t}^{(s)})-F(x^{*})\right)\right]\\
\overset{(b)}{\leq} & \E\left[\left(1-a^{(s)}\right)\left(F(u^{(s-1)})-F(x^{*})\right)\right]\\
 & +\E\left[\frac{1}{T_{s}}\sum_{t=1}^{T_{s}}\frac{\gamma_{t-1}^{(s)}q^{(s)}a^{(s)}}{2}\left(\left\Vert x_{t-1}^{(s)}-x^{*}\right\Vert ^{2}-\left\Vert x_{t}^{(s)}-x^{*}\right\Vert ^{2}\right)\right]\\
 & +\E\left[\frac{1}{T_{s}}\sum_{t=1}^{T_{s}}\left(\frac{\beta\left(2-a^{(s)}\right)\left(a^{(s)}\right)^{2}}{2\left(1-a^{(s)}\right)}-\frac{\gamma_{t-1}^{(s)}q^{(s)}a^{(s)}}{2}\right)\left\Vert x_{t}^{(s)}-x_{t-1}^{(s)}\right\Vert ^{2}\right],
\end{align*}
where $(a)$ is by the convexity of $F$ and the definition of $u^{(s)}=\frac{1}{T_{s}}\sum_{t=1}^{T_{s}}\avx_{t}^{(s)}$,
and $(b)$ is by Lemma \ref{lem:AdaVRAG-One-Step-Descent}.
\end{proof}

Lemma \ref{lem:AdaVRAE-Whole-Descent} is a quite general result without
any assumptions on any parameters. To ensure that we can make the
telescoping sum over the function value part, and also to simplify
the term besides the function value part, we need some specific conditions
on our parameters to be satisfied, which is stated in Lemma \ref{lem:AdaVRAG-Whole-Descent}.
With these extra conditions, we can finally find the following guarantee
for the function value gap of the final output $u^{(S)}$.
\begin{lem}
\label{lem:AdaVRAG-Whole-Descent}For all $S\geq1$, if the parameters
satisfy
\[
\frac{\left(2-a^{(s)}\right)a^{(s)}}{1-a^{(s)}}\leq q^{(s)},\forall s\in\left[S\right]
\]
and
\[
\frac{(1-a^{(s+1)})T_{s+1}}{q^{(s+1)}a^{(s+1)}}\leq\frac{T_{s}}{q^{(s)}a^{(s)}},\forall s\in\left[S-1\right].
\]
then we have
\begin{align*}
 & \E\left[\frac{T_{S}}{q^{(S)}a^{(S)}}(F(u^{(S)})-F(x^{*}))\right]\\
\leq & \frac{(1-a^{(1)})T_{1}}{q^{(1)}a^{(1)}}(F(u^{(0)})-F(x^{*}))\\
 & +\E\left[\sum_{s=1}^{S}\sum_{t=1}^{T_{s}}\frac{\gamma_{t-1}^{(s)}}{2}\left\Vert x_{t-1}^{(s)}-x^{*}\right\Vert ^{2}-\frac{\gamma_{t-1}^{(s)}}{2}\left\Vert x_{t}^{(s)}-x^{*}\right\Vert ^{2}+\frac{\beta-\gamma_{t-1}^{(s)}}{2}\left\Vert x_{t}^{(s)}-x_{t-1}^{(s)}\right\Vert ^{2}\right].
\end{align*}
\end{lem}
\begin{proof}
If $\frac{\left(2-a^{(s)}\right)a^{(s)}}{1-a^{(s)}}\leq q^{(s)}$
for any $s\in\left[S\right]$, by using Lemma \ref{lem:AdaVRAG-One-Epoch-Descent},
we know
\begin{align*}
 & \E\left[F(u^{(s)})-F(x^{*})\right]\\
\leq & \E\left[\left(1-a^{(s)}\right)\left(F(u^{(s-1)})-F(x^{*})\right)\right]\\
 & +\E\left[\frac{1}{T_{s}}\sum_{t=1}^{T_{s}}\frac{\gamma_{t-1}^{(s)}q^{(s)}a^{(s)}}{2}\left(\left\Vert x_{t-1}^{(s)}-x^{*}\right\Vert ^{2}-\left\Vert x_{t}^{(s)}-x^{*}\right\Vert ^{2}\right)\right]\\
 & +\E\left[\frac{1}{T_{s}}\sum_{t=1}^{T_{s}}\left(\frac{\beta\left(2-a^{(s)}\right)\left(a^{(s)}\right)^{2}}{2\left(1-a^{(s)}\right)}-\frac{\gamma_{t-1}^{(s)}q^{(s)}a^{(s)}}{2}\right)\left\Vert x_{t}^{(s)}-x_{t-1}^{(s)}\right\Vert ^{2}\right]\\
\leq & \E\left[(1-a^{(s)})(F(u^{(s-1)})-F(x^{*}))\right]\\
 & +\E\left[\frac{q^{(s)}a^{(s)}}{T_{s}}\left(\sum_{t=1}^{T_{s}}\frac{\gamma_{t-1}^{(s)}}{2}\left\Vert x_{t-1}^{(s)}-x^{*}\right\Vert ^{2}-\frac{\gamma_{t-1}^{(s)}}{2}\left\Vert x_{t}^{(s)}-x^{*}\right\Vert ^{2}+\frac{\beta-\gamma_{t-1}^{(s)}}{2}\left\Vert x_{t}^{(s)}-x_{t-1}^{(s)}\right\Vert ^{2}\right)\right]
\end{align*}
Now multiply both sides by $\frac{T_{s}}{q^{(s)}a^{(s)}}$, we have
\begin{align*}
 & \E\left[\frac{T_{s}}{q^{(s)}a^{(s)}}(F(u^{(s)})-F(x^{*}))\right]\\
\leq & \E\left[\frac{(1-a^{(s)})T_{s}}{q^{(s)}a^{(s)}}(F(u^{(s-1)})-F(x^{*}))\right]\\
 & +\E\left[\sum_{t=1}^{T_{s}}\frac{\gamma_{t-1}^{(s)}}{2}\left\Vert x_{t-1}^{(s)}-x^{*}\right\Vert ^{2}-\frac{\gamma_{t-1}^{(s)}}{2}\left\Vert x_{t}^{(s)}-x^{*}\right\Vert ^{2}+\frac{\beta-\gamma_{t-1}^{(s)}}{2}\left\Vert x_{t}^{(s)}-x_{t-1}^{(s)}\right\Vert ^{2}\right].
\end{align*}
If $\frac{(1-a^{(s+1)})T_{s+1}}{q^{(s+1)}a^{(s+1)}}\leq\frac{T_{s}}{q^{(s)}a^{(s)}}$
is satisfied for any $s\in\left[S-1\right]$, we can make the telescoping
sum from $s=1$ to $S$ to get
\begin{align*}
 & \E\left[\frac{T_{S}}{q^{(S)}a^{(S)}}(F(u^{(S)})-F(x^{*}))\right]\\
\leq & \frac{(1-a^{(1)})T_{1}}{q^{(1)}a^{(1)}}(F(u^{(0)})-F(x^{*}))\\
 & +\E\left[\sum_{s=1}^{S}\sum_{t=1}^{T_{s}}\frac{\gamma_{t-1}^{(s)}}{2}\left\Vert x_{t-1}^{(s)}-x^{*}\right\Vert ^{2}-\frac{\gamma_{t-1}^{(s)}}{2}\left\Vert x_{t}^{(s)}-x^{*}\right\Vert ^{2}+\frac{\beta-\gamma_{t-1}^{(s)}}{2}\left\Vert x_{t}^{(s)}-x_{t-1}^{(s)}\right\Vert ^{2}\right].
\end{align*}
\end{proof}

\subsection{Bound for the residual term}

By the analysis in the previous subsection, we get an upper bound
for the function value gap of $u^{(S)}$ involving $F(u^{(0)})-F(x^{*})$
and
\begin{equation}
\E\left[\sum_{s=1}^{S}\sum_{t=1}^{T_{s}}\frac{\gamma_{t-1}^{(s)}}{2}\left\Vert x_{t-1}^{(s)}-x^{*}\right\Vert ^{2}-\frac{\gamma_{t-1}^{(s)}}{2}\left\Vert x_{t}^{(s)}-x^{*}\right\Vert ^{2}+\frac{\beta-\gamma_{t-1}^{(s)}}{2}\left\Vert x_{t}^{(s)}-x_{t-1}^{(s)}\right\Vert ^{2}\right].\label{eq:residual}
\end{equation}
In this subsection we will show how to bound \ref{eq:residual} under
the compact assumption of $\dom$. Before giving the detailed analysis
of the two different update options, we first state the following
lemma to simplify \ref{eq:residual}.
\begin{lem}
\label{lem:AdaVRAG-Simplify-Reminder}If $\gamma_{t}^{(s)}\geq\gamma_{t-1}^{(s)}$
for any $s\in\left[S\right]$, $t\in\left[T_{s}\right]$ and $\dom$
is a compact convex set with diameter $D$, then we have
\begin{align*}
 & \E\left[\sum_{s=1}^{S}\sum_{t=1}^{T_{s}}\frac{\gamma_{t-1}^{(s)}}{2}\left\Vert x_{t-1}^{(s)}-x^{*}\right\Vert ^{2}-\frac{\gamma_{t-1}^{(s)}}{2}\left\Vert x_{t}^{(s)}-x^{*}\right\Vert ^{2}+\frac{\beta-\gamma_{t-1}^{(s)}}{2}\left\Vert x_{t}^{(s)}-x_{t-1}^{(s)}\right\Vert ^{2}\right]\\
\leq & \frac{\gamma}{2}\left\Vert u^{(0)}-x^{*}\right\Vert ^{2}+\E\left[\sum_{s=1}^{S}\sum_{t=1}^{T_{s}}\frac{\gamma_{t}^{(s)}-\gamma_{t-1}^{(s)}}{2}D^{2}+\frac{\beta-\gamma_{t-1}^{(s)}}{2}\left\Vert x_{t}^{(s)}-x_{t-1}^{(s)}\right\Vert ^{2}\right].
\end{align*}
\end{lem}
\begin{proof}
It follows that
\begin{align*}
 & \sum_{s=1}^{S}\sum_{t=1}^{T_{s}}\frac{\gamma_{t-1}^{(s)}}{2}\left\Vert x_{t-1}^{(s)}-x^{*}\right\Vert ^{2}-\frac{\gamma_{t-1}^{(s)}}{2}\left\Vert x_{t}^{(s)}-x^{*}\right\Vert ^{2}+\frac{\beta-\gamma_{t-1}^{(s)}}{2}\left\Vert x_{t}^{(s)}-x_{t-1}^{(s)}\right\Vert ^{2}\\
= & \sum_{s=1}^{S}\sum_{t=1}^{T_{s}}\frac{\gamma_{t-1}^{(s)}}{2}\left\Vert x_{t-1}^{(s)}-x^{*}\right\Vert ^{2}-\frac{\gamma_{t}^{(s)}}{2}\left\Vert x_{t}^{(s)}-x^{*}\right\Vert ^{2}+\frac{\gamma_{t}^{(s)}-\gamma_{t-1}^{(s)}}{2}\left\Vert x_{t}^{(s)}-x^{*}\right\Vert ^{2}+\frac{\beta-\gamma_{t-1}^{(s)}}{2}\left\Vert x_{t}^{(s)}-x_{t-1}^{(s)}\right\Vert ^{2}\\
\overset{(a)}{\leq} & \sum_{s=1}^{S}\sum_{t=1}^{T_{s}}\frac{\gamma_{t-1}^{(s)}}{2}\left\Vert x_{t-1}^{(s)}-x^{*}\right\Vert ^{2}-\frac{\gamma_{t}^{(s)}}{2}\left\Vert x_{t}^{(s)}-x^{*}\right\Vert ^{2}+\frac{\gamma_{t}^{(s)}-\gamma_{t-1}^{(s)}}{2}D^{2}+\frac{\beta-\gamma_{t-1}^{(s)}}{2}\left\Vert x_{t}^{(s)}-x_{t-1}^{(s)}\right\Vert ^{2}\\
= & \sum_{s=1}^{S}\left(\frac{\gamma_{0}^{(s)}}{2}\left\Vert x_{0}^{(s)}-x^{*}\right\Vert ^{2}-\frac{\gamma_{T_{s}}^{(s)}}{2}\left\Vert x_{T_{s}}^{(s)}-x^{*}\right\Vert ^{2}+\sum_{t=1}^{T_{s}}\frac{\gamma_{t}^{(s)}-\gamma_{t-1}^{(s)}}{2}D^{2}+\frac{\beta-\gamma_{t-1}^{(s)}}{2}\left\Vert x_{t}^{(s)}-x_{t-1}^{(s)}\right\Vert ^{2}\right)\\
\overset{(b)}{=} & \sum_{s=1}^{S}\left(\frac{\gamma_{0}^{(s)}}{2}\left\Vert x_{0}^{(s)}-x^{*}\right\Vert ^{2}-\frac{\gamma_{0}^{(s+1)}}{2}\left\Vert x_{0}^{(s+1)}-x^{*}\right\Vert ^{2}+\sum_{t=1}^{T_{s}}\frac{\gamma_{t}^{(s)}-\gamma_{t-1}^{(s)}}{2}D^{2}+\frac{\beta-\gamma_{t-1}^{(s)}}{2}\left\Vert x_{t}^{(s)}-x_{t-1}^{(s)}\right\Vert ^{2}\right)\\
= & \frac{\gamma_{0}^{(1)}}{2}\left\Vert x_{0}^{(1)}-x^{*}\right\Vert ^{2}-\frac{\gamma_{0}^{(S+1)}}{2}\left\Vert x_{0}^{(S+1)}-x^{*}\right\Vert ^{2}+\sum_{s=1}^{S}\sum_{t=1}^{T_{s}}\frac{\gamma_{t}^{(s)}-\gamma_{t-1}^{(s)}}{2}D^{2}+\frac{\beta-\gamma_{t-1}^{(s)}}{2}\left\Vert x_{t}^{(s)}-x_{t-1}^{(s)}\right\Vert ^{2}\\
\leq & \frac{\gamma_{0}^{(1)}}{2}\left\Vert x_{0}^{(1)}-x^{*}\right\Vert ^{2}+\sum_{s=1}^{S}\sum_{t=1}^{T_{s}}\frac{\gamma_{t}^{(s)}-\gamma_{t-1}^{(s)}}{2}D^{2}+\frac{\beta-\gamma_{t-1}^{(s)}}{2}\left\Vert x_{t}^{(s)}-x_{t-1}^{(s)}\right\Vert ^{2}\\
\overset{(c)}{=} & \frac{\gamma}{2}\left\Vert u^{(0)}-x^{*}\right\Vert ^{2}+\sum_{s=1}^{S}\sum_{t=1}^{T_{s}}\frac{\gamma_{t}^{(s)}-\gamma_{t-1}^{(s)}}{2}D^{2}+\frac{\beta-\gamma_{t-1}^{(s)}}{2}\left\Vert x_{t}^{(s)}-x_{t-1}^{(s)}\right\Vert ^{2},
\end{align*}
where $(a)$ is due to $\gamma_{t}^{(s)}\geq\gamma_{t-1}^{(s)}$ and
$\left\Vert x_{t}^{(s)}-x^{*}\right\Vert \leq D$, $(b)$ follows
from the definition of $x_{0}^{(s+1)}=x_{T_{s}}^{(s)}$ and $\gamma_{0}^{(s+1)}=\gamma_{T_{s}}^{(s)}$,
$(c)$ is by the definition of $x_{0}^{(1)}=u^{(0)}$ and $\gamma_{0}^{(1)}=\gamma$.
Now Taking expectations with both sides yields what we want.
\end{proof}

With the above result, we can show the bound of \ref{eq:residual}
under Option I and Option II respectively. There are two key common
parts in our analysis, the first one is to notice that we can reduce
the doubly indexed sequence $\left\{ x_{t}^{(s)}\right\} $ and $\left\{ \gamma_{t}^{(s)}\right\} $
into two singly indexed sequences, which are much easier to bound.
The second technique is to define a hitting time $\tau$ to upper
bound $\gamma_{t}^{(s)}$. Read our proof for the details.
\begin{lem}
\label{lem:AdaVRAG-Reminder-I}For $\text{Option I}$, if $\dom$
is a compact convex set with diameter $D$ and $2\eta^{2}>D^{2}$,
we have
\begin{align*}
 & \E\left[\sum_{s=1}^{S}\sum_{t=1}^{T_{s}}\frac{\gamma_{t-1}^{(s)}}{2}\left\Vert x_{t-1}^{(s)}-x^{*}\right\Vert ^{2}-\frac{\gamma_{t-1}^{(s)}}{2}\left\Vert x_{t}^{(s)}-x^{*}\right\Vert ^{2}+\frac{\beta-\gamma_{t-1}^{(s)}}{2}\left\Vert x_{t}^{(s)}-x_{t-1}^{(s)}\right\Vert ^{2}\right]\\
\leq & \frac{\gamma}{2}\left\Vert u^{(0)}-x^{*}\right\Vert ^{2}+\left[\frac{\beta}{2}-\left(\frac{1}{2}-\frac{D^{2}}{4\eta^{2}}\right)\gamma\right]^{+}\left(D^{2}+2(\eta^{2}+D^{2})\log\frac{\frac{2\eta^{2}\beta}{2\eta^{2}-D^{2}}}{\gamma}\right).
\end{align*}
\end{lem}
\begin{proof}
For Option I, by the definition of $\gamma_{t}^{(s)}$, we have
\[
\gamma_{t}^{(s)}\geq\gamma_{t-1}^{(s)},\forall s\in\left[S\right],t\in\left[T_{s}\right].
\]
By requiring that $\dom$ is a compact convex set with diameter $D$,
we can apply Lemma \ref{lem:AdaVRAG-Simplify-Reminder} and obtain
\begin{align}
 & \E\left[\sum_{s=1}^{S}\sum_{t=1}^{T_{s}}\frac{\gamma_{t-1}^{(s)}}{2}\left\Vert x_{t-1}^{(s)}-x^{*}\right\Vert ^{2}-\frac{\gamma_{t-1}^{(s)}}{2}\left\Vert x_{t}^{(s)}-x^{*}\right\Vert ^{2}+\frac{\beta-\gamma_{t-1}^{(s)}}{2}\left\Vert x_{t}^{(s)}-x_{t-1}^{(s)}\right\Vert ^{2}\right]\nonumber \\
\leq & \frac{\gamma}{2}\left\Vert u^{(0)}-x^{*}\right\Vert ^{2}+\E\left[\sum_{s=1}^{S}\sum_{t=1}^{T_{s}}\frac{\gamma_{t}^{(s)}-\gamma_{t-1}^{(s)}}{2}D^{2}+\frac{\beta-\gamma_{t-1}^{(s)}}{2}\left\Vert x_{t}^{(s)}-x_{t-1}^{(s)}\right\Vert ^{2}\right].\label{eq:varag-6}
\end{align}

Note that the last element $x_{T_{s}}^{(s)}$ (resp. $\gamma_{T_{s}}^{(s)}$)
in the $s$-th epoch is just the start element $x_{0}^{(s+1)}$ (resp.
$\gamma_{0}^{(s+1)}$) in the $(s+1)$-th epoch, which means we can
consider the doubly indexed sequences $\{x_{t}^{(s)}\}$ and $\{\gamma_{t}^{(s)}\}$
as two singly indexed sequences $\{x'_{t},t\geq0\}$ and $\{\gamma'_{t},t\geq0,\gamma'_{0}=\gamma\}$
with the reformulated update rule as follows
\[
\gamma'_{t}=\gamma'_{t-1}\sqrt{1+\frac{\left\Vert x'_{t}-x'_{t-1}\right\Vert ^{2}}{\eta^{2}}}.
\]
Besides, by defining $T'=\sum_{s=1}^{S}T_{s}$, we have
\[
\sum_{s=1}^{S}\sum_{t=1}^{T_{s}}\frac{\gamma_{t}^{(s)}-\gamma_{t-1}^{(s)}}{2}D^{2}+\frac{\beta-\gamma_{t-1}^{(s)}}{2}\left\Vert x_{t}^{(s)}-x_{t-1}^{(s)}\right\Vert ^{2}=\sum_{t=1}^{T'}\frac{\gamma'_{t}-\gamma'_{t-1}}{2}D^{2}+\frac{\beta-\gamma'_{t-1}}{2}\left\Vert x'_{t}-x'_{t-1}\right\Vert ^{2}.
\]

Note that we require $2\eta^{2}>D^{2}$, so if $\gamma\geq\frac{2\eta^{2}\beta}{2\eta^{2}-D^{2}}\Leftrightarrow\frac{\beta}{2}-\left(\frac{1}{2}-\frac{D^{2}}{4\eta^{2}}\right)\gamma\leq0\Rightarrow\frac{\beta}{2}-\left(\frac{1}{2}-\frac{D^{2}}{4\eta^{2}}\right)\gamma'_{t-1}\leq0$,
by using the reformulated update rule, we have
\begin{align*}
 & \sum_{t=1}^{T'}\frac{\gamma'_{t}-\gamma'_{t-1}}{2}D^{2}+\frac{\beta-\gamma'_{t-1}}{2}\left\Vert x'_{t}-x'_{t-1}\right\Vert ^{2}\\
= & \sum_{t=1}^{T'}\frac{(\gamma'_{t})^{2}-(\gamma'_{t-1})^{2}}{2(\gamma'_{t}+\gamma'_{t-1})}D^{2}+\frac{\beta-\gamma'_{t-1}}{2}\left\Vert x'_{t}-x'_{t-1}\right\Vert ^{2}\\
= & \sum_{t=1}^{T'}\frac{(\gamma'_{t-1})^{2}D^{2}}{2\eta^{2}(\gamma'_{t}+\gamma'_{t-1})}\left\Vert x'_{t}-x'_{t-1}\right\Vert ^{2}+\frac{\beta-\gamma'_{t-1}}{2}\left\Vert x'_{t}-x'_{t-1}\right\Vert ^{2}\\
\overset{(a)}{\leq} & \sum_{t=1}^{T'}\left(\frac{\gamma'_{t-1}}{4\eta^{2}}D^{2}+\frac{\beta-\gamma'_{t-1}}{2}\right)\left\Vert x'_{t}-x'_{t-1}\right\Vert ^{2}\\
= & \sum_{t=1}^{T'}\left[\frac{\beta}{2}-\left(\frac{1}{2}-\frac{D^{2}}{4\eta^{2}}\right)\gamma'_{t-1}\right]\left\Vert x'_{t}-x'_{t-1}\right\Vert ^{2}\\
\leq & 0,
\end{align*}
where $(a)$ is by $\gamma'_{t}\geq\gamma'_{t-1}$. Now we assume
$\gamma<\frac{2\eta^{2}\beta}{2\eta^{2}-D^{2}}$, define
\[
\tau=\max\left\{ t\in[T'],\gamma'_{t-1}<\frac{2\eta^{2}\beta}{2\eta^{2}-D^{2}}\right\} .
\]
By our assumption on $\gamma$, we know $\tau\geq1$, Combining the
reformulated update rule, we have
\begin{align*}
 & \sum_{t=1}^{T'}\frac{\gamma'_{t}-\gamma'_{t-1}}{2}D^{2}+\frac{\beta-\gamma'_{t-1}}{2}\left\Vert x'_{t}-x'_{t-1}\right\Vert ^{2}\\
\leq & \sum_{t=1}^{T'}\left[\frac{\beta}{2}-\left(\frac{1}{2}-\frac{D^{2}}{4\eta^{2}}\right)\gamma'_{t-1}\right]\left\Vert x'_{t}-x'_{t-1}\right\Vert ^{2}\\
\leq & \sum_{t=1}^{\tau}\left[\frac{\beta}{2}-\left(\frac{1}{2}-\frac{D^{2}}{4\eta^{2}}\right)\gamma'_{t-1}\right]\left\Vert x'_{t}-x'_{t-1}\right\Vert ^{2}\\
\overset{(b)}{\leq} & \left[\frac{\beta}{2}-\left(\frac{1}{2}-\frac{D^{2}}{4\eta^{2}}\right)\gamma\right]\sum_{t=1}^{\tau}\left\Vert x'_{t}-x'_{t-1}\right\Vert ^{2}\\
\overset{(c)}{\leq} & \left[\frac{\beta}{2}-\left(\frac{1}{2}-\frac{D^{2}}{4\eta^{2}}\right)\gamma\right]\left(D^{2}+\sum_{t=1}^{\tau-1}\left\Vert x'_{t}-x'_{t-1}\right\Vert ^{2}\right)\\
\overset{(d)}{=} & \left[\frac{\beta}{2}-\left(\frac{1}{2}-\frac{D^{2}}{4\eta^{2}}\right)\gamma\right]\left(D^{2}+\sum_{t=1}^{\tau-1}\eta^{2}\frac{(\gamma'_{t})^{2}-(\gamma'_{t-1})^{2}}{(\gamma'_{t-1})^{2}}\right)\\
\overset{(e)}{\leq} & \left[\frac{\beta}{2}-\left(\frac{1}{2}-\frac{D^{2}}{4\eta^{2}}\right)\gamma\right]\left(D^{2}+\left(\eta^{2}+D^{2}\right)\sum_{t=1}^{\tau-1}\frac{(\gamma'_{t})^{2}-(\gamma'_{t-1})^{2}}{(\gamma'_{t})^{2}}\right)\\
\overset{(f)}{\leq} & \left[\frac{\beta}{2}-\left(\frac{1}{2}-\frac{D^{2}}{4\eta^{2}}\right)\gamma\right]\left(D^{2}+2\left(\eta^{2}+D^{2}\right)\sum_{t=1}^{\tau-1}\log\frac{\gamma'_{t}}{\gamma'_{t-1}}\right)\\
= & \left[\frac{\beta}{2}-\left(\frac{1}{2}-\frac{D^{2}}{4\eta^{2}}\right)\gamma\right]\left(D^{2}+2\left(\eta^{2}+D^{2}\right)\log\frac{\gamma'_{\tau-1}}{\gamma}\right)\\
\overset{(g)}{\leq} & \left[\frac{\beta}{2}-\left(\frac{1}{2}-\frac{D^{2}}{4\eta^{2}}\right)\gamma\right]\left(D^{2}+2\left(\eta^{2}+D^{2}\right)\log\frac{\frac{2\eta^{2}\beta}{2\eta^{2}-D^{2}}}{\gamma}\right),
\end{align*}
where $(b)$ is by $\gamma'_{t-1}\geq\gamma$, $(c)$ is by $\left\Vert x'_{\tau}-x'_{\tau-1}\right\Vert \leq D$,
$(d)$ is by the reformulated update rule, $(e)$ is due to 
\[
\gamma'_{t}=\gamma'_{t-1}\sqrt{1+\frac{\left\Vert x'_{t}-x'_{t-1}\right\Vert ^{2}}{\eta^{2}}}\leq\gamma'_{t-1}\sqrt{1+\frac{D^{2}}{\eta^{2}}},
\]
$(f)$ is by the inequality $1-\frac{1}{x^{2}}\leq\log x^{2}=2\log x$,
$(g)$ is by the definition of $\tau$.

Combining two cases of $\gamma$, we obtain the bound
\begin{align}
\sum_{s=1}^{S}\sum_{t=1}^{T_{s}}\frac{\gamma_{t}^{(s)}-\gamma_{t-1}^{(s)}}{2}D^{2}+\frac{\beta-\gamma_{t-1}^{(s)}}{2}\left\Vert x_{t}^{(s)}-x_{t-1}^{(s)}\right\Vert ^{2} & =\sum_{t=1}^{T'}\frac{\gamma'_{t}-\gamma'_{t-1}}{2}D^{2}+\frac{\beta-\gamma'_{t-1}}{2}\left\Vert x'_{t}-x'_{t-1}\right\Vert ^{2}\nonumber \\
 & \leq\left[\frac{\beta}{2}-\left(\frac{1}{2}-\frac{D^{2}}{4\eta^{2}}\right)\gamma\right]^{+}\left(D^{2}+2\left(\eta^{2}+D^{2}\right)\log\frac{\frac{2\eta^{2}\beta}{2\eta^{2}-D^{2}}}{\gamma}\right).\label{eq:varag-7}
\end{align}
By plugging in (\ref{eq:varag-7}) into (\ref{eq:varag-6}), we have
\begin{align*}
 & \E\left[\sum_{s=1}^{S}\sum_{t=1}^{T_{s}}\frac{\gamma_{t-1}^{(s)}}{2}\left\Vert x_{t-1}^{(s)}-x^{*}\right\Vert ^{2}-\frac{\gamma_{t-1}^{(s)}}{2}\left\Vert x_{t}^{(s)}-x^{*}\right\Vert ^{2}+\frac{\beta-\gamma_{t-1}^{(s)}}{2}\left\Vert x_{t}^{(s)}-x_{t-1}^{(s)}\right\Vert ^{2}\right]\\
\leq & \frac{\gamma}{2}\left\Vert u^{(0)}-x^{*}\right\Vert ^{2}+\left[\frac{\beta}{2}-\left(\frac{1}{2}-\frac{D^{2}}{4\eta^{2}}\right)\gamma\right]^{+}\left(D^{2}+2\left(\eta^{2}+D^{2}\right)\log\frac{\frac{2\eta^{2}\beta}{2\eta^{2}-D^{2}}}{\gamma}\right).
\end{align*}
\end{proof}

\begin{lem}
\label{lem:AdaVRAG-Reminder-II}For $\text{Option II}$, if $\dom$
is a compact set with diameter $D$, we have
\begin{align*}
 & \E\left[\sum_{s=1}^{S}\sum_{t=1}^{T_{s}}\frac{\gamma_{t-1}^{(s)}}{2}\left\Vert x_{t-1}^{(s)}-x^{*}\right\Vert ^{2}-\frac{\gamma_{t-1}^{(s)}}{2}\left\Vert x_{t}^{(s)}-x^{*}\right\Vert ^{2}+\frac{\beta-\gamma_{t-1}^{(s)}}{2}\left\Vert x_{t}^{(s)}-x_{t-1}^{(s)}\right\Vert ^{2}\right]\\
\leq & \frac{\gamma}{2}\left\Vert u^{(0)}-x^{*}\right\Vert ^{2}+\frac{\eta^{2}}{2}\left(\frac{D^{2}}{\eta^{2}}+\beta-\gamma\right)^{+}\left(\frac{2D^{2}}{\eta^{2}}+\beta-\gamma\right).
\end{align*}
\end{lem}
\begin{proof}
For Option II, by the definition of $\gamma_{t}^{(s)}$, we have
\[
\gamma_{t}^{(s)}\geq\gamma_{t-1}^{(s)},\forall s\in\left[S\right],t\in\left[T_{s}\right].
\]
By requiring that $\dom$ is a compact convex set with diameter $D$,
we can apply Lemma \ref{lem:AdaVRAG-Simplify-Reminder} and obtain
\begin{align}
 & \E\left[\sum_{s=1}^{S}\sum_{t=1}^{T_{s}}\frac{\gamma_{t-1}^{(s)}}{2}\left\Vert x_{t-1}^{(s)}-x^{*}\right\Vert ^{2}-\frac{\gamma_{t-1}^{(s)}}{2}\left\Vert x_{t}^{(s)}-x^{*}\right\Vert ^{2}+\frac{\beta-\gamma_{t-1}^{(s)}}{2}\left\Vert x_{t}^{(s)}-x_{t-1}^{(s)}\right\Vert ^{2}\right]\nonumber \\
\leq & \frac{\gamma}{2}\left\Vert u^{(0)}-x^{*}\right\Vert ^{2}+\E\left[\sum_{s=1}^{S}\sum_{t=1}^{T_{s}}\frac{\gamma_{t}^{(s)}-\gamma_{t-1}^{(s)}}{2}D^{2}+\frac{\beta-\gamma_{t-1}^{(s)}}{2}\left\Vert x_{t}^{(s)}-x_{t-1}^{(s)}\right\Vert ^{2}\right].\label{eq:varag-4}
\end{align}

Note that the last element $x_{T_{s}}^{(s)}$ (resp. $\gamma_{T_{s}}^{(s)}$)
in the $s$-th epoch is just the starting element $x_{0}^{(s+1)}$
(resp. $\gamma_{0}^{(s+1)}$) in the $(s+1)$-th epoch, which means
we can consider the doubly indexed sequences $\{x_{t}^{(s)}\}$ and
$\{\gamma_{t}^{(s)}\}$ as two singly indexed sequences $\{x'_{t},t\geq0\}$
and $\{\gamma'_{t},t\geq0,\gamma'_{0}=\gamma\}$ with the reformulated
update rule as follows
\[
\gamma'_{t}=\gamma'_{t-1}+\frac{\left\Vert x'_{t}-x'_{t-1}\right\Vert ^{2}}{\eta^{2}}.
\]
Besides, by defining $T'=\sum_{s=1}^{S}T_{s}$, we have
\[
\sum_{s=1}^{S}\sum_{t=1}^{T_{s}}\frac{\gamma_{t}^{(s)}-\gamma_{t-1}^{(s)}}{2}D^{2}+\frac{\beta-\gamma_{t-1}^{(s)}}{2}\left\Vert x_{t}^{(s)}-x_{t-1}^{(s)}\right\Vert ^{2}=\sum_{t=1}^{T'}\frac{\gamma'_{t}-\gamma'_{t-1}}{2}D^{2}+\frac{\beta-\gamma'_{t-1}}{2}\left\Vert x'_{t}-x'_{t-1}\right\Vert ^{2}.
\]

If $\gamma\geq\frac{D^{2}}{\eta^{2}}+\beta\Leftrightarrow\frac{D^{2}}{2\eta^{2}}+\frac{\beta-\gamma}{2}\leq0\Rightarrow\frac{D^{2}}{2\eta^{2}}+\frac{\beta-\gamma'_{t-1}}{2}\leq0$,
by using the reformulated update rule, we have
\begin{align*}
\sum_{t=1}^{T'}\frac{\gamma'_{t}-\gamma'_{t-1}}{2}D^{2}+\frac{\beta-\gamma'_{t-1}}{2}\left\Vert x'_{t}-x'_{t-1}\right\Vert ^{2} & =\sum_{t=1}^{T'}\left(\frac{D^{2}}{2\eta^{2}}+\frac{\beta-\gamma'_{t-1}}{2}\right)\left\Vert x'_{t}-x'_{t-1}\right\Vert ^{2}\\
 & \leq0.
\end{align*}
Now we assume $\gamma<\frac{D^{2}}{\eta^{2}}+\beta$. Define
\[
\tau=\max\left\{ t\in[T'],\gamma'_{t-1}<\frac{D^{2}}{\eta^{2}}+\beta\right\} .
\]
By our assumption on $\gamma$, we know $\tau\geq1$. Combining the
reformulated update rule, we have
\begin{align*}
 & \sum_{t=1}^{T'}\frac{\gamma'_{t}-\gamma'_{t-1}}{2}D^{2}+\frac{\beta-\gamma'_{t-1}}{2}\left\Vert x'_{t}-x'_{t-1}\right\Vert ^{2}\\
= & \sum_{t=1}^{T'}\left(\frac{D^{2}}{2\eta^{2}}+\frac{\beta-\gamma'_{t-1}}{2}\right)\left\Vert x'_{t}-x'_{t-1}\right\Vert ^{2}\\
\leq & \sum_{t=1}^{\tau}\left(\frac{D^{2}}{2\eta^{2}}+\frac{\beta-\gamma'_{t-1}}{2}\right)\left\Vert x'_{t}-x'_{t-1}\right\Vert ^{2}\\
\overset{(a)}{\leq} & \sum_{t=1}^{\tau}\left(\frac{D^{2}}{2\eta^{2}}+\frac{\beta-\gamma}{2}\right)\left\Vert x'_{t}-x'_{t-1}\right\Vert ^{2}\\
\overset{(b)}{=} & \sum_{t=1}^{\tau}\left(\frac{D^{2}}{2\eta^{2}}+\frac{\beta-\gamma}{2}\right)\eta^{2}\left(\gamma'_{t}-\gamma'_{t-1}\right)\\
= & \left(\frac{D^{2}}{2\eta^{2}}+\frac{\beta-\gamma}{2}\right)\eta^{2}\left(\gamma'_{\tau}-\gamma\right)\\
\overset{(c)}{=} & \left(\frac{D^{2}}{2\eta^{2}}+\frac{\beta-\gamma}{2}\right)\eta^{2}\left(\gamma'_{\tau-1}+\frac{\left\Vert x'_{\tau}-x'_{\tau-1}\right\Vert ^{2}}{\eta^{2}}-\gamma\right)\\
\overset{(d)}{\leq} & \left(\frac{D^{2}}{2\eta^{2}}+\frac{\beta-\gamma}{2}\right)\eta^{2}\left(2\frac{D^{2}}{\eta^{2}}+\beta-\gamma\right)\\
= & \frac{\eta^{2}}{2}\left(\frac{D^{2}}{\eta^{2}}+\beta-\gamma\right)\left(\frac{2D^{2}}{\eta^{2}}+\beta-\gamma\right),
\end{align*}
where $(a)$ is by the fact $\gamma'_{t-1}\geq\gamma$, $(b)$ and
$(c)$ are by the reformulated update rule, $(d)$ is by the definition
of $\tau$ and $\left\Vert x'_{\tau}-x'_{\tau-1}\right\Vert \leq D.$

Combining two cases of $\gamma$, we obtain the bound
\begin{align}
 & \sum_{s=1}^{S}\sum_{t=1}^{T_{s}}\frac{\gamma_{t}^{(s)}-\gamma_{t-1}^{(s)}}{2}D^{2}+\frac{\beta-\gamma_{t-1}^{(s)}}{2}\left\Vert x_{t}^{(s)}-x_{t-1}^{(s)}\right\Vert ^{2}\nonumber \\
= & \sum_{t=1}^{T'}\frac{\gamma'_{t}-\gamma'_{t-1}}{2}D^{2}+\frac{\beta-\gamma'_{t-1}}{2}\left\Vert x'_{t}-x'_{t-1}\right\Vert ^{2}\nonumber \\
\leq & \frac{\eta^{2}}{2}\left(\frac{D^{2}}{\eta^{2}}+\beta-\gamma\right)^{+}\left(\frac{2D^{2}}{\eta^{2}}+\beta-\gamma\right).\label{eq:varag-5}
\end{align}
By plugging in (\ref{eq:varag-5}) into (\ref{eq:varag-4}), we have
\begin{align*}
 & \E\left[\sum_{s=1}^{S}\sum_{t=1}^{T_{s}}\frac{\gamma_{t-1}^{(s)}}{2}\left\Vert x_{t-1}^{(s)}-x^{*}\right\Vert ^{2}-\frac{\gamma_{t-1}^{(s)}}{2}\left\Vert x_{t}^{(s)}-x^{*}\right\Vert ^{2}+\frac{\beta-\gamma_{t-1}^{(s)}}{2}\left\Vert x_{t}^{(s)}-x_{t-1}^{(s)}\right\Vert ^{2}\right]\\
\leq & \frac{\gamma}{2}\left\Vert u^{(0)}-x^{*}\right\Vert ^{2}+\frac{\eta^{2}}{2}\left(\frac{D^{2}}{\eta^{2}}+\beta-\gamma\right)^{+}\left(\frac{2D^{2}}{\eta^{2}}+\beta-\gamma\right).
\end{align*}
\end{proof}

\subsection{Parameter bound}

Combining the previous two parts analysis on the function value gap
and the residual term, we already can see the bound for $F(u^{(S)})-F(x^{*})$.
However, we need to make sure that our choice stated in Theorem \ref{thm:AdaVRAG-Convergence-A}
indeed satisfies the conditions used in previous lemmas, besides,
we also need to give the bounds for our choice explicitly. The following
two lemmas can help us to do this.
\begin{lem}
\label{lem:AdaVRAG-Condition}Under the choice of parameters in Theorem
\ref{thm:AdaVRAG-Convergence-A}, $\forall s\geq1$, we have the following
facts
\begin{align*}
a^{(s_{0})} & \leq\frac{1}{2},\\
\frac{\left(2-a^{(s)}\right)a^{(s)}}{1-a^{(s)}} & \leq q^{(s)},\\
\frac{(1-a^{(s+1)})T_{s+1}}{q^{(s+1)}a^{(s+1)}} & \leq\frac{T_{s}}{q^{(s)}a^{(s)}}.
\end{align*}
\end{lem}
\begin{proof}
Under the choice of parameters in Theorem \ref{thm:AdaVRAG-Convergence-A},
the first inequality follows that
\[
a^{(s_{0})}=1-\left(4n\right)^{-0.5^{s_{0}}}\leq1-\left(4n\right)^{-0.5^{\log_{2}\log_{2}4n}}=\frac{1}{2}.
\]

For the second inequality, note that
\[
\frac{\left(2-a^{(s)}\right)a^{(s)}}{\left(1-a^{(s)}\right)q^{(s)}}=\begin{cases}
\left(2-a^{(s)}\right)\left(a^{(s)}\right)^{2} & 1\leq s\leq s_{0}\\
\text{\ensuremath{\frac{3}{8}}} & s_{0}<s
\end{cases}.
\]
By noticing $\left(2-a^{(s)}\right)\left(a^{(s)}\right)^{2}\leq a^{(s)}\leq1$,
the inequality $\frac{\left(2-a^{(s)}\right)a^{(s)}}{1-a^{(s)}}\leq q^{(s)}$
becomes true immediately.

For the third inequality, note that we have $T_{s}\equiv n$, we only
need to prove for any $s\geq1$, there is
\[
\frac{1-a^{(s+1)}}{q^{(s+1)}a^{(s+1)}}\leq\frac{1}{q^{(s)}a^{(s)}}.
\]
We consider the following three cases:
\begin{itemize}
\item For $1\leq s\leq s_{0}-1$, note that $\left(1-a^{(s+1)}\right)^{2}=\left(4n\right)^{-0.5^{s}}=1-a^{(s)}$,
$q^{(s)}=\frac{1}{\left(1-a^{(s)}\right)a^{(s)}}$. We know
\begin{align*}
\frac{1-a^{(s+1)}}{q^{(s+1)}a^{(s+1)}} & =(1-a^{(s+1)})^{2}\\
 & =1-a^{(s)}\\
 & =\frac{1}{q^{(s)}a^{(s)}}.
\end{align*}
\item For $s=s_{0}$, note that $a^{(s_{0}+1)}=\frac{c}{1+2c}=\frac{9-\sqrt{33}}{8}$,
$q^{(s_{0}+1)}=\frac{8(2-a^{(s_{0}+1)})a^{(s_{0}+1)}}{3(1-a^{(s_{0}+1)})}$
we have
\begin{align*}
\frac{1-a^{(s_{0}+1)}}{q^{(s_{0}+1)}a^{(s_{0}+1)}} & =\frac{3(1-a^{(s_{0}+1)})^{2}}{8(2-a^{(s_{0}+1)})\left(a^{(s_{0}+1)}\right)^{2}}\\
 & =\frac{1}{2}\\
 & \overset{(a)}{\leq}1-a^{(s_{0})}\\
 & \overset{(b)}{=}\frac{1}{q^{(s_{0})}a^{(s_{0})}},
\end{align*}
where $(a)$ is by $a^{(s_{0})}\leq\frac{1}{2}$, $(b)$ is by $q^{(s_{0})}=\frac{1}{\left(1-a^{(s_{0})}\right)a^{(s_{0})}}$.
\item For $s\geq s_{0}+1$, note that $q^{(s)}=\frac{8(2-a^{(s)})a^{(s)}}{3\left(1-a^{(s)}\right)}$,
by plugging in $q^{(s)}$, we only need to show
\[
\frac{(1-a^{(s+1)})^{2}}{(a^{(s+1)})^{2}(2-a^{(s+1)})}\leq\frac{1-a^{(s)}}{(a^{(s)})^{2}(2-a^{(s)})}.
\]
Plug in $a^{(s)}=\frac{c}{s-s_{0}+2c}$, the above inequality is equivalent
to
\[
(2(s-s_{0})+3c)(s-s_{0}+1+2c)(s-s_{0}+1+c)^{2}\leq(2(s-s_{0})+2+3c)(s-s_{0}+c)(s-s_{0}+2c)^{2}.
\]
Let $y=s-s_{0}\geq1$, we need to show
\[
(2y+3c)(y+1+2c)(y+1+c)^{2}\leq(2y+2+3c)(y+c)(y+2c)^{2}
\]
is true for $y\geq1$. People can check when $c=\frac{3+\sqrt{33}}{4}$,
the above inequality is right for $y\ge1$.
\end{itemize}
\end{proof}

\begin{lem}
\label{lem:AdaVRAG-Bound}Under the choice of parameters in Theorem
\ref{thm:AdaVRAG-Convergence-A}, $\forall s\geq1$, we have the following
bounds
\[
\frac{(1-a^{(1)})T_{1}}{q^{(1)}a^{(1)}}=\frac{1}{4}
\]
and
\[
\frac{q^{(s)}a^{(s)}}{T_{s}}\leq\begin{cases}
\frac{4}{(4n)^{1-0.5^{s}}} & 1\leq s\leq s_{0}\\
\frac{2(5+\sqrt{33})c^{2}}{3n(s-s_{0}+2c)^{2}} & s_{0}<s
\end{cases}.
\]
\end{lem}
\begin{proof}
Note that $a^{(1)}=1-\frac{1}{2\sqrt{n}}$, $T_{1}=n$, $q^{(1)}=\frac{1}{(1-a^{(1)})a^{(1)}}$,
plugging in these values, we obtain
\begin{align*}
\frac{(1-a^{(1)})T_{1}}{q^{(1)}a^{(1)}} & =(1-a^{(1)})^{2}T_{1}\\
 & =\frac{1}{4}
\end{align*}
\begin{itemize}
\item For $1\leq s\leq s_{0}$, note that $q^{(s)}=\frac{1}{(1-a^{(s)})a^{(s)}}$
in our choice, so we know
\begin{align*}
\frac{q^{(s)}a^{(s)}}{T_{s}} & =\frac{1}{T_{s}(1-a^{(s)})}\\
 & \overset{(a)}{=}\frac{4}{(4n)^{1-0.5^{s}}}
\end{align*}
where $(a)$ is by plugging in $T_{s}=n$ and $a^{(s)}=1-\left(4n\right){}^{-0.5^{s}}$.
\item For $s>s_{0}$, note that $q^{(s)}=\frac{8(2-a^{(s)})a^{(s)}}{3(1-a^{(s)})}$
we have
\begin{align*}
\frac{q^{(s)}a^{(s)}}{T_{s}} & =\frac{8(2-a^{(s)})(a^{(s)})^{2}}{3T_{s}(1-a^{(s)})}\\
 & \overset{(b)}{=}\frac{8(2-a^{(s)})(a^{(s)})^{2}}{3n(1-a^{(s)})}\\
 & \overset{(c)}{\leq}\frac{2(5+\sqrt{33})c^{2}}{3n(s-s_{0}+2c)^{2}},
\end{align*}
where $(b)$ is by plugging in$T_{s}=n$, $(c)$ is by noticing $\frac{2-a^{(s)}}{1-a^{(s)}}\leq\frac{2-a^{(s_{0}+1)}}{1-a^{(s_{0}+1)}}=\frac{5+\sqrt{33}}{4}$
for $s>s_{0}$, and plug in $a^{(s)}=\frac{c}{s-s_{0}+2c}$.
\end{itemize}
\end{proof}

\subsection{Putting all together}

We are now ready to put everything together and complete the proof
of Theorem \ref{thm:AdaVRAG-Convergence-A}.

\begin{proof}
(Theorem \ref{thm:AdaVRAG-Convergence-A}) By Lemma \ref{lem:AdaVRAG-Condition},
$\forall s\geq1$, we have
\begin{align*}
\frac{\left(2-a^{(s)}\right)a^{(s)}}{1-a^{(s)}} & \leq q^{(s)},\\
\frac{(1-a^{(s+1)})T_{s+1}}{q^{(s+1)}a^{(s+1)}} & \leq\frac{T_{s}}{q^{(s)}a^{(s)}}.
\end{align*}
Hence all the conditions for Lemma \ref{lem:AdaVRAG-Whole-Descent}
are satisfied. Besides, we assume $\dom$ is a compact convex set
with diameter $D$, which satisfies the requirements for Lemma \ref{lem:AdaVRAG-Reminder-II}
and \ref{lem:AdaVRAG-Reminder-I}.
\begin{enumerate}
\item For Option I, by Lemma \ref{lem:AdaVRAG-Whole-Descent} and \ref{lem:AdaVRAG-Reminder-I}
\begin{align*}
\E\left[\frac{T_{S}}{q^{(S)}a^{(S)}}(F(u^{(S)})-F(x^{*}))\right] & \leq\frac{(1-a^{(1)})T_{1}}{q^{(1)}a^{(1)}}(F(u^{(0)})-F(x^{*}))+\frac{\gamma}{2}\left\Vert u^{(0)}-x^{*}\right\Vert ^{2}\\
 & \quad+\left[\frac{\beta}{2}-\left(\frac{1}{2}-\frac{D^{2}}{4\eta^{2}}\right)\gamma\right]^{+}\left(D^{2}+2(\eta^{2}+D^{2})\log\frac{\frac{2\eta^{2}\beta}{2\eta^{2}-D^{2}}}{\gamma}\right).
\end{align*}
\item For Option II, by Lemma \ref{lem:AdaVRAG-Whole-Descent} and \ref{lem:AdaVRAG-Reminder-II}
\begin{align*}
\E\left[\frac{T_{S}}{q^{(S)}a^{(S)}}(F(u^{(S)})-F(x^{*}))\right] & \leq\frac{(1-a^{(1)})T_{1}}{q^{(1)}a^{(1)}}(F(u^{(0)})-F(x^{*}))+\frac{\gamma}{2}\left\Vert u^{(0)}-x^{*}\right\Vert ^{2}\\
 & \quad+\frac{\eta^{2}}{2}\left(\frac{D^{2}}{\eta^{2}}+\beta-\gamma\right)^{+}\left(\frac{2D^{2}}{\eta^{2}}+\beta-\gamma\right).
\end{align*}
\end{enumerate}
Plugging in the bound $\frac{(1-a^{(1)})T_{1}}{q^{(1)}a^{(1)}}=\frac{1}{4}$
from Lemma \ref{lem:AdaVRAG-Bound}, we have
\begin{align*}
\E\left[\frac{T_{S}}{q^{(S)}a^{(S)}}(F(u^{(S)})-F(x^{*}))\right] & \leq\frac{V}{2}\\
\Rightarrow\E\left[F(u^{(S)})-F(x^{*})\right] & \leq\frac{q^{(S)}a^{(S)}V}{2T_{S}}\\
 & \overset{(a)}{\leq}\begin{cases}
\frac{2V}{(4n)^{1-0.5^{S}}} & 1\leq S\leq s_{0}\\
\frac{(5+\sqrt{33})c^{2}V}{3n(S-s_{0}+2c)^{2}} & s_{0}<S
\end{cases},
\end{align*}
where $(a)$ is by the bound for $\frac{q^{(S)}a^{(S)}}{T_{S}}$ from
Lemma \ref{lem:AdaVRAG-Bound}.
\begin{itemize}
\item If $\epsilon\ge\frac{V}{n}$, we choose $S=\left\lceil \log_{2}\log_{2}\frac{4V}{\epsilon}\right\rceil \leq\left\lceil \log_{2}\log_{2}4n\right\rceil =s_{0}$,
so we have
\begin{align*}
\E\left[F(u^{(S)})-F(x^{*})\right] & \leq\frac{2V}{(4n)^{1-0.5^{S}}}\\
 & \overset{(b)}{\leq}\frac{2V}{\left(\frac{4V}{\epsilon}\right){}^{1-0.5^{S}}}\\
 & =\frac{\epsilon}{2\left(\frac{4V}{\epsilon}\right){}^{-0.5^{S}}}\\
 & \overset{(c)}{\leq}\epsilon,
\end{align*}
where $(b)$ is by $n\geq\frac{V}{\epsilon}$, $(c)$ is by $\left(\frac{4V}{\epsilon}\right){}^{-0.5^{S}}=\left(\frac{4V}{\epsilon}\right){}^{-0.5^{\left\lceil \log_{2}\log_{2}\frac{4V}{\epsilon}\right\rceil }}\geq\left(\frac{4V}{\epsilon}\right){}^{-0.5^{\log_{2}\log_{2}\frac{4V}{\epsilon}}}=\frac{1}{2}$.
The number of individual gradient evaluations is
\begin{align*}
\#grads & =nS+\sum_{s=1}^{S}2T_{s}\\
 & =3nS\\
 & =3n\left\lceil \log_{2}\log_{2}\frac{4V}{\epsilon}\right\rceil \\
 & =\mathcal{O}\left(n\log\log\frac{V}{\epsilon}\right).
\end{align*}
\item If $\epsilon<\frac{V}{n}$, we choose $S=s_{0}+\left\lceil c\left(\sqrt{\frac{(5+\sqrt{33})V}{3n\epsilon}}-\frac{15}{8}\right)\right\rceil \geq s_{0}+\left\lceil c\left(\sqrt{\frac{5+\sqrt{33}}{3}}-\frac{15}{8}\right)\right\rceil =s_{0}+1$,
so we have
\begin{align*}
\E\left[F(u^{(S)})-F(x^{*})\right] & \leq\frac{(5+\sqrt{33})c^{2}V}{3n(S-s_{0}+2c)^{2}}\\
 & =\frac{(5+\sqrt{33})c^{2}V}{3n\left(s_{0}+\left\lceil c\left(\sqrt{\frac{(5+\sqrt{33})V}{3n\epsilon}}-\frac{15}{8}\right)\right\rceil +2c\right)^{2}}\\
 & \leq\frac{(5+\sqrt{33})c^{2}V}{3n\left(c\sqrt{\frac{(5+\sqrt{33})V}{3n\epsilon}}+\frac{c}{8}\right)^{2}}\\
 & \leq\frac{(5+\sqrt{33})c^{2}V}{3n\left(c\sqrt{\frac{(5+\sqrt{33})V}{3n\epsilon}}\right)^{2}}\\
 & =\epsilon.
\end{align*}
The number of individual gradient evaluations is
\begin{align*}
\#grads & =nS+\sum_{s=1}^{S}2T_{s}\\
 & =3nS\\
 & =3ns_{0}+3n(S-s_{0})\\
 & =3n\left\lceil \log_{2}\log_{2}4n\right\rceil +3n\left\lceil c\left(\sqrt{\frac{(5+\sqrt{33})V}{3n\epsilon}}-\frac{15}{8}\right)\right\rceil \\
 & =\mathcal{O}\left(n\log\log n+\sqrt{\frac{nV}{\epsilon}}\right).
\end{align*}
\end{itemize}
\end{proof}

\section{AdaVRAE for known $\beta$}

\label{sec:VRAE}

\begin{algorithm}
\caption{VRAE\label{alg:VRAE}}

\textbf{Input}: initial point $u^{(0)}$, smoothness parameter $\beta$.

\textbf{Parameters}: $\{a^{(s)}\}$, $\left\{ T_{s}\right\} $, $A_{T_{0}}^{(0)}>0$

$\avx_{0}^{(1)}=z_{0}^{(1)}=u^{(0)}$, compute $\nabla f(u^{(0)})$

\textbf{for $s=1$ to $S$:}

$\quad$$A_{0}^{(s)}=A_{T_{s-1}}^{(s-1)}-T_{s}\left(a^{(s)}\right)^{2}$

\textbf{$\quad$for $t=1$ to $T_{s}$:}

$\quad\quad$$x_{t}^{(s)}=\arg\min_{x\in\dom}\left\{ a^{(s)}\left\langle g_{t-1}^{(s)},x\right\rangle +a^{(s)}h(x)+4\beta\left\Vert x-z_{t-1}^{(s)}\right\Vert ^{2}\right\} $

$\quad\quad$Let $A_{t}^{(s)}=A_{t-1}^{(s)}+a^{(s)}+\left(a^{(s)}\right)^{2}$

$\quad\quad$$\avx_{t}^{(s)}=\frac{1}{A_{t}^{(s)}}\left(A_{t-1}^{(s)}\avx_{t-1}^{(s)}+a^{(s)}x_{t}^{(s)}+\left(a^{(s)}\right)^{2}u^{(s-1)}\right)$

$\quad\quad$ \textbf{if $t\neq T_{s}$:}

$\quad\quad\quad$Pick $i_{t}^{(s)}\sim\mathrm{Uniform}\left(\left[n\right]\right)$

$\quad\quad\quad$$g_{t}^{(s)}=\nabla f_{i_{t}^{(s)}}(\avx_{t}^{(s)})-\nabla f_{i_{t}^{(s)}}(u^{(s-1)})+\nabla f(u^{(s-1)})$

$\quad\quad$ \textbf{else:}

$\quad\quad\quad$$g_{t}^{(s)}=\nabla f(\avx_{t}^{(s)})$

$\quad\quad$$z_{t}^{(s)}=\arg\min_{z\in\dom}\left\{ a^{(s)}\left\langle g_{t}^{(s)},z\right\rangle +a^{(s)}h(z)+4\beta\left\Vert z-z_{t-1}^{(s)}\right\Vert ^{2}\right\} $

$\quad$$u^{(s)}=\avx_{0}^{(s+1)}=\avx_{T_{s}}^{(s)}$, $z_{0}^{(s+1)}=z_{T_{s}}^{(s)}$,
$g_{0}^{(s+1)}=g_{T_{s}}^{(s)}$

\textbf{return} $u^{(S)}$
\end{algorithm}

In this section, we give a non-adaptive version of our algorithm AdaVRAE.
The algorithm is shown in Algorithm \ref{alg:VRAE}. The only change
is in the step size: we set $\gamma_{t}^{(s)}=8\beta$ for all epochs
$s$ and iterations $t$. The analysis readily extends to show the
following convergence guarantee:
\begin{thm}
\label{thm:AdaVRAE-Convergence-NA} Let $s_{0}=\left\lceil \log_{2}\log_{2}4n\right\rceil $,
$c=\frac{3}{2}$. If we choose parameters as follows
\begin{align*}
a^{(s)} & =\begin{cases}
(4n)^{-0.5^{s}} & 1\leq s\leq s_{0}\\
\frac{s-s_{0}-1+c}{2c} & s_{0}<s
\end{cases},\\
T_{s} & =n,\\
A_{T_{0}}^{(0)} & =\frac{5}{4}.
\end{align*}
The number of gradient evaluations to achieve a solution $u^{(S)}$
such that $\E\left[F(u^{(S)})-F(x^{*})\right]\le\epsilon$ for Algorithm
\ref{alg:VRAE} is

\begin{align*}
\#grads & =\begin{cases}
\mathcal{O}\text{\ensuremath{\left(n\log\log\frac{V}{\epsilon}\right)}} & \mbox{if }\epsilon\geq\frac{V}{n}\\
\mathcal{O}\left(n\log\log n+\sqrt{\frac{Vn}{\epsilon}}\right) & \mbox{if }\epsilon<\frac{V}{n}
\end{cases}
\end{align*}

where $V=\frac{5}{2}\left(F(u^{(0)})-F(x^{*})\right)+\text{\ensuremath{8\beta}}\left\Vert u^{(0)}-x^{*}\right\Vert ^{2}$.
\end{thm}
\begin{proof}
Note that Algorithm \ref{alg:VRAE} is essentially the same as Algorithm
\ref{alg:AdaVRAE} by choosing $\gamma_{t}^{(s)}\equiv8\beta$ with
no other changes. Hence the requirements for Lemma \ref{lem:AdaVRAE-Whole-Descent}
still hold. So we can obtain
\begin{align*}
\E\left[A_{T_{S}}^{(S)}\left(F(u^{(S)})-F(x^{*})\right)\right] & \leq A_{T_{0}}^{(0)}\left(F(u^{(0)})-F(x^{*})\right)+4\beta\left\Vert u^{(0)}-x^{*}\right\Vert ^{2}.
\end{align*}
Then by the similar proof in Theorem \ref{thm:AdaVRAE-Convergence-A},
we get the desired result.
\end{proof}

\section{AdaVRAG for known $\beta$}

\label{sec:VRAG}

\begin{algorithm}
\caption{VRAG\label{alg:VRAG}}

\textbf{Input:} initial point $u^{(0)}$, smoothness parameter $\beta$

\textbf{Parameters:} $\{a^{(s)}\}$ where $a^{(s)}\in(0,1)$, $\left\{ T_{s}\right\} $

$x_{0}^{(1)}=u^{(0)}$

\textbf{for $s=1$ to $S$:}

$\quad$$\avx_{0}^{(s)}=a^{(s)}x_{0}^{(s)}+(1-a^{(s)})u^{(s-1)},\text{calculate }\nabla f(u^{(s-1)})$

\textbf{$\quad$for $t=1$ to $T_{s}$:}

$\quad\quad$Pick $i_{t}^{(s)}\sim\mathrm{Uniform}\left(\left[n\right]\right)$

$\quad\quad$$g_{t}^{(s)}=\nabla f_{i_{t}^{(s)}}(\avx_{t-1}^{(s)})-\nabla f_{i_{t}^{(s)}}(u^{(s-1)})+\nabla f(u^{(s-1)})$

$\quad\quad$$x_{t}^{(s)}=\arg\min_{x\in\dom}\left\{ \left\langle g_{t}^{(s)},x\right\rangle +h(x)+\frac{\beta\left(2-a^{(s)}\right)a^{(s)}}{2(1-a^{(s)})}\left\Vert x-x_{t-1}^{(s)}\right\Vert ^{2}\right\} $

$\quad\quad$$\avx_{t}^{(s)}=a^{(s)}x_{t}^{(s)}+(1-a^{(s)})u^{(s-1)}$

$\quad$$u^{(s)}=\frac{1}{T_{s}}\sum_{t=1}^{T_{s}}\avx_{t}^{(s)}$,
$x_{0}^{(s+1)}=x_{T_{s}}^{(s)}$

\textbf{return} $u^{(S)}$
\end{algorithm}

In this section, we give a non-adaptive version of our algorithm AdaVRAG.
The algorithm is shown in Algorithm \ref{alg:VRAG}. VRAG admits the
following convergence guarantee:
\begin{thm}
\label{thm:AdaVRAG-Convergence-NA}(Convergence of VRAG) Define $s_{0}=\lceil\log_{2}\log_{2}4n\rceil$,
$c=\frac{3+\sqrt{33}}{4}$. If we choose the parameters as follows
\begin{align*}
a^{(s)} & =\begin{cases}
1-\left(4n\right)^{-0.5^{s}} & 1\leq s\leq s_{0}\\
\frac{c}{s-s_{0}+2c} & s_{0}<s
\end{cases},\\
T_{s} & =n.
\end{align*}
The number of individual gradient evaluations to achieve a solution
$u^{(S)}$ such that $\E\left[F(u^{(S)})-F(x^{*})\right]\leq\epsilon$
for Algorithm \ref{alg:VRAG} is 
\[
\#grads=\begin{cases}
\mathcal{O}\left(n\log\log\frac{V}{\epsilon}\right) & \epsilon\geq\frac{V}{n}\\
\mathcal{O}\left(n\log\log n+\sqrt{\frac{nV}{\epsilon}}\right) & \epsilon<\frac{V}{n}
\end{cases},
\]
where
\[
V=\frac{1}{2}(F(u^{(0)})-F(x^{*}))+\beta\left\Vert u^{(0)}-x^{*}\right\Vert ^{2}.
\]
\end{thm}
Before giving the proof of Theorem \ref{thm:AdaVRAE-Convergence-NA},
we state some intuition on our parameter choice. Note that by defining
the following two auxiliary sequences

\begin{align*}
q^{(s)} & =\begin{cases}
\frac{1}{\left(1-a^{(s)}\right)a^{(s)}} & 1\leq s\leq s_{0}\\
\frac{\left(2-a^{(s)}\right)a^{(s)}}{1-a^{(s)}} & s_{0}<s
\end{cases},\\
\gamma_{t-1}^{(s)} & =\frac{\beta\left(2-a^{(s)}\right)a^{(s)}}{(1-a^{(s)})q^{(s)}},\forall t\in\left[T_{s}\right],
\end{align*}
the update rule of $x_{t}^{(s)}$ in every epoch in Algorithm \ref{alg:VRAG}
is equivalent to the update rule of $x_{t}^{(s)}$ in every epoch
in Algorithm \ref{alg:AdaVRAG}. Since $\gamma_{t-1}^{(s)}$ is a
constant in the corresponding epoch now, we will use $\gamma^{(s)}$
without the subscript to simplify the notation. The above argument
means that we can apply Lemma \ref{lem:AdaVRAG-One-Epoch-Descent}
directly to obtain the following lemma.
\begin{lem}
\label{lem:AdaVRAG-One-Epoch-Descent-NA}For all epochs $s\geq1$,
we have
\begin{align*}
\E\left[F(u^{(s)})-F(x^{*})\right] & \leq\E\left[\left(1-a^{(s)}\right)\left(F(u^{(s-1)})-F(x^{*})\right)+\frac{\gamma^{(s)}q^{(s)}a^{(s)}}{2T_{s}}\left(\left\Vert x_{0}^{(s)}-x^{*}\right\Vert ^{2}-\left\Vert x_{0}^{(s+1)}-x^{*}\right\Vert ^{2}\right)\right].
\end{align*}
\end{lem}
\begin{proof}
By applying Lemma \ref{lem:AdaVRAG-One-Epoch-Descent}, we know
\begin{align*}
 & \E\left[F(u^{(s)})-F(x^{*})\right]\\
\leq & \E\left[\left(1-a^{(s)}\right)\left(F(u^{(s-1)})-F(x^{*})\right)\right]\\
 & +\E\left[\frac{1}{T_{s}}\sum_{t=1}^{T_{s}}\frac{\gamma_{t-1}^{(s)}q^{(s)}a^{(s)}}{2}\left(\left\Vert x_{t-1}^{(s)}-x^{*}\right\Vert ^{2}-\left\Vert x_{t}^{(s)}-x^{*}\right\Vert ^{2}\right)\right]\\
 & +\E\left[\frac{1}{T_{s}}\sum_{t=1}^{T_{s}}\left(\frac{\beta\left(2-a^{(s)}\right)\left(a^{(s)}\right)^{2}}{2\left(1-a^{(s)}\right)}-\frac{\gamma_{t-1}^{(s)}q^{(s)}a^{(s)}}{2}\right)\left\Vert x_{t}^{(s)}-x_{t-1}^{(s)}\right\Vert ^{2}\right]\\
\overset{(a)}{=} & \E\left[\left(1-a^{(s)}\right)\left(F(u^{(s-1)})-F(x^{*})\right)+\frac{\gamma^{(s)}q^{(s)}a^{(s)}}{2T_{s}}\sum_{t=1}^{T_{s}}\left\Vert x_{t-1}^{(s)}-x^{*}\right\Vert ^{2}-\left\Vert x_{t}^{(s)}-x^{*}\right\Vert ^{2}\right]\\
= & \E\left[\left(1-a^{(s)}\right)\left(F(u^{(s-1)})-F(x^{*})\right)+\frac{\gamma^{(s)}q^{(s)}a^{(s)}}{2T_{s}}\left(\left\Vert x_{0}^{(s)}-x^{*}\right\Vert ^{2}-\left\Vert x_{T_{s}}^{(s)}-x^{*}\right\Vert ^{2}\right)\right]\\
\overset{(b)}{=} & \E\left[\left(1-a^{(s)}\right)\left(F(u^{(s-1)})-F(x^{*})\right)+\frac{\gamma^{(s)}q^{(s)}a^{(s)}}{2T_{s}}\left(\left\Vert x_{0}^{(s)}-x^{*}\right\Vert ^{2}-\left\Vert x_{0}^{(s+1)}-x^{*}\right\Vert ^{2}\right)\right],
\end{align*}
where $(a)$ is by $\gamma_{t-1}^{(s)}q^{(s)}=\frac{\beta\left(2-a^{(s)}\right)a^{(s)}}{1-a^{(s)}}$
and $\gamma^{(s)}=\gamma_{t-1}^{(s)},\forall t\in\left[T_{s}\right]$,
$(b)$ is by $x_{0}^{(s+1)}=x_{T_{s}}^{(s)}$.
\end{proof}

Now if we still multiply both sides by $\frac{T_{s}}{q^{(s)}a^{(s)}}$,
we need to ensure that $\gamma^{(s)}$ can help us to make a telescoping
sum. However, this is not always true. So we need some different conditions
as stated in the following lemma to obtain a bound for the function
value gap of $u^{(S)}$. The new bound for the function value gap
of $u^{(S)}$ for Algorithm \ref{alg:VRAG} is as follows.
\begin{lem}
\label{lem:AdaVRAG-Whole-Descent-NA}If  $\forall s\neq s_{0}$, we
have
\begin{align*}
a^{(s+1)} & \leq a^{(s)},\\
\frac{(1-a^{(s+1)})T_{s+1}}{q^{(s+1)}a^{(s+1)}} & \leq\frac{T_{s}}{q^{(s)}a^{(s)}}.
\end{align*}
Additionally, for $s_{0}$, assume we have
\[
\frac{\left(1-a^{(s_{0}+1)}\right)^{2}T_{s_{0}+1}}{\left(2-a^{(s_{0}+1)}\right)\left(a^{(s_{0}+1)}\right)^{2}}\leq\frac{(1-a^{(s_{0})})T_{s_{0}}}{\left(2-a^{(s_{0})}\right)\left(a^{(s_{0})}\right)^{2}}.
\]
Then for $S\leq s_{0}$,
\begin{align*}
\E\left[\frac{T_{S}}{q^{(S)}a^{(S)}}\left(F(u^{(S)})-F(x^{*})\right)\right] & \leq\frac{\left(1-a^{(1)}\right)T_{1}}{q^{(1)}a^{(1)}}\left(F(u^{(0)})-F(x^{*})\right)+\frac{\beta}{2}\left\Vert u^{(0)}-x^{*}\right\Vert ^{2}.
\end{align*}
For $S>s_{0}$,
\[
\E\left[\frac{\left(2-a^{(s_{0})}\right)\left(a^{(s_{0})}\right)^{2}T_{S}}{q^{(S)}a^{(S)}}\left(F(u^{(S)})-F(x^{*})\right)\right]\leq\frac{\left(1-a^{(1)}\right)T_{1}}{q^{(1)}a^{(1)}}\left(F(u^{(0)})-F(x^{*})\right)+\frac{\beta}{2}\left\Vert u^{(0)}-x^{*}\right\Vert ^{2}
\]
\end{lem}
\begin{proof}
By applying Lemma \ref{lem:AdaVRAG-One-Epoch-Descent-NA} and multiply
both sides by $\frac{T_{s}}{q^{(s)}a^{(s)}},$we have
\begin{align*}
\E\left[\frac{T_{s}}{q^{(s)}a^{(s)}}\left(F(u^{(s)})-F(x^{*})\right)\right] & \leq\E\left[\frac{\left(1-a^{(s)}\right)T_{s}}{q^{(s)}a^{(s)}}\left(F(u^{(s-1)})-F(x^{*})\right)+\frac{\gamma^{(s)}}{2}\left(\left\Vert x_{0}^{(s)}-x^{*}\right\Vert ^{2}-\left\Vert x_{0}^{(s+1)}-x^{*}\right\Vert ^{2}\right)\right].
\end{align*}

For $S\leq s_{0}$
\begin{align*}
 & \E\left[\frac{T_{S}}{q^{(S)}a^{(S)}}\left(F(u^{(S)})-F(x^{*})\right)\right]\\
\leq & \E\left[\frac{\left(1-a^{(1)}\right)T_{1}}{q^{(1)}a^{(1)}}\left(F(u^{(0)})-F(x^{*})\right)+\sum_{s=1}^{S}\frac{\gamma^{(s)}}{2}\left(\left\Vert x_{0}^{(s)}-x^{*}\right\Vert ^{2}-\left\Vert x_{0}^{(s+1)}-x^{*}\right\Vert ^{2}\right)\right]\\
\overset{(a)}{=} & \frac{\left(1-a^{(1)}\right)T_{1}}{q^{(1)}a^{(1)}}\left(F(u^{(0)})-F(x^{*})\right)+\E\left[\sum_{s=1}^{S}\frac{\beta\left(2-a^{(s)}\right)\left(a^{(s)}\right)^{2}}{2}\left(\left\Vert x_{0}^{(s)}-x^{*}\right\Vert ^{2}-\left\Vert x_{0}^{(s+1)}-x^{*}\right\Vert ^{2}\right)\right]\\
= & \frac{\left(1-a^{(1)}\right)T_{1}}{q^{(1)}a^{(1)}}\left(F(u^{(0)})-F(x^{*})\right)+\frac{\beta\left(2-a^{(1)}\right)\left(a^{(1)}\right)^{2}}{2}\left\Vert x_{0}^{(1)}-x^{*}\right\Vert ^{2}\\
 & +\E\left[\sum_{s=1}^{S-1}\frac{\beta\left[\left(2-a^{(s+1)}\right)\left(a^{(s+1)}\right)^{2}-\left(2-a^{(s)}\right)\left(a^{(s)}\right)^{2}\right]}{2}\left(\left\Vert x_{0}^{(s+1)}-x^{*}\right\Vert ^{2}\right)\right]\\
 & -\E\left[\frac{\beta\left(2-a^{(S)}\right)\left(a^{(S)}\right)^{2}}{2}\left(\left\Vert x_{0}^{(S+1)}-x^{*}\right\Vert ^{2}\right)\right]\\
\overset{(b)}{\leq} & \frac{\left(1-a^{(1)}\right)T_{1}}{q^{(1)}a^{(1)}}\left(F(u^{(0)})-F(x^{*})\right)+\frac{\beta}{2}\left\Vert u^{(0)}-x^{*}\right\Vert ^{2}\\
 & -\E\left[\frac{\beta\left(2-a^{(S)}\right)\left(a^{(S)}\right)^{2}}{2}\left(\left\Vert x_{0}^{(S+1)}-x^{*}\right\Vert ^{2}\right)\right],
\end{align*}
where $(a)$ is by the definition of $\gamma^{(s)}$ when $s\leq s_{0}$,
$(b)$ is by $\left(2-a^{(1)}\right)\left(a^{(1)}\right)^{2}\leq1$
and $x_{0}^{(1)}=u^{(0)}$, additionally, note that our assumption
$a^{(s+1)}\leq a^{(s)}\Rightarrow\left(2-a^{(s+1)}\right)\left(a^{(s+1)}\right)^{2}\leq\left(2-a^{(s)}\right)\left(a^{(s)}\right)^{2}$.

For $S>s_{0}$, we can also make the telescoping sum from $s=s_{0}+1$
to $S$ by a similar argument to get
\begin{align*}
\E\left[\frac{T_{S}}{q^{(S)}a^{(S)}}\left(F(u^{(S)})-F(x^{*})\right)\right] & \leq\E\left[\frac{\left(1-a^{(s_{0}+1)}\right)T_{s_{0}+1}}{q^{(s_{0}+1)}a^{(s_{0}+1)}}\left(F(u^{(s_{0})})-F(x^{*})\right)+\frac{\beta}{2}\left\Vert x_{0}^{(s_{0}+1)}-x^{*}\right\Vert ^{2}\right].
\end{align*}
Multiplying both sides by $\left(2-a^{(s_{0})}\right)\left(a^{(s_{0})}\right)^{2}$,
we have
\begin{align*}
 & \E\left[\frac{\left(2-a^{(s_{0})}\right)\left(a^{(s_{0})}\right)^{2}T_{S}}{q^{(S)}a^{(S)}}\left(F(u^{(S)})-F(x^{*})\right)\right]\\
\leq & \E\left[\frac{\left(2-a^{(s_{0})}\right)\left(a^{(s_{0})}\right)^{2}\left(1-a^{(s_{0}+1)}\right)T_{s_{0}+1}}{q^{(s_{0}+1)}a^{(s_{0}+1)}}\left(F(u^{(s_{0})})-F(x^{*})\right)+\frac{\beta\left(2-a^{(s_{0})}\right)\left(a^{(s_{0})}\right)^{2}}{2}\left\Vert x_{0}^{(s_{0}+1)}-x^{*}\right\Vert ^{2}\right]\\
\overset{(c)}{=} & \E\left[\frac{\left(2-a^{(s_{0})}\right)\left(a^{(s_{0})}\right)^{2}\left(1-a^{(s_{0}+1)}\right)^{2}T_{s_{0}+1}}{(2-a^{(s_{0}+1)})\left(a^{(s_{0}+1)}\right)^{2}}\left(F(u^{(s_{0})})-F(x^{*})\right)+\frac{\beta\left(2-a^{(s_{0})}\right)\left(a^{(s_{0})}\right)^{2}}{2}\left\Vert x_{0}^{(s_{0}+1)}-x^{*}\right\Vert ^{2}\right],
\end{align*}
where $(c)$ is by the definition $q^{(s_{0}+1)}=\frac{\left(2-a^{(s_{0}+1)}\right)a^{(s_{0}+1)}}{1-a^{(s_{0}+1)}}$.
Note that by our assumption
\begin{align*}
\frac{\left(2-a^{(s_{0})}\right)\left(a^{(s_{0})}\right)^{2}\left(1-a^{(s_{0}+1)}\right)^{2}T_{s_{0}+1}}{(2-a^{(s_{0}+1)})\left(a^{(s_{0}+1)}\right)^{2}} & \leq(1-a^{(s_{0})})T_{s_{0}}\\
 & =\frac{T_{s_{0}}}{q^{(s_{0})}a^{(s_{0})}},
\end{align*}
so we know
\begin{align*}
 & \E\left[\frac{\left(2-a^{(s_{0})}\right)\left(a^{(s_{0})}\right)^{2}T_{S}}{q^{(S)}a^{(S)}}\left(F(u^{(S)})-F(x^{*})\right)\right]\\
\leq & \E\left[\frac{T_{s_{0}}}{q^{(s_{0})}a^{(s_{0})}}\left(F(u^{(s_{0})})-F(x^{*})\right)+\frac{\beta\left(2-a^{(s_{0})}\right)\left(a^{(s_{0})}\right)^{2}}{2}\left\Vert x_{0}^{(s_{0}+1)}-x^{*}\right\Vert ^{2}\right].
\end{align*}
Now combining
\begin{align*}
\E\left[\frac{T_{s_{0}}}{q^{(s_{0})}a^{(s_{0})}}\left(F(u^{(s_{0})})-F(x^{*})\right)\right] & \leq\frac{\left(1-a^{(1)}\right)T_{1}}{q^{(1)}a^{(1)}}\left(F(u^{(0)})-F(x^{*})\right)+\frac{\beta}{2}\left\Vert u^{(0)}-x^{*}\right\Vert ^{2}\\
 & \quad-\E\left[\frac{\beta\left(2-a^{(s_{0})}\right)\left(a^{(s_{0})}\right)^{2}}{2}\left(\left\Vert x_{0}^{(s_{0}+1)}-x^{*}\right\Vert ^{2}\right)\right],
\end{align*}
we have
\begin{align*}
\E\left[\frac{\left(2-a^{(s_{0})}\right)\left(a^{(s_{0})}\right)^{2}T_{S}}{q^{(S)}a^{(S)}}\left(F(u^{(S)})-F(x^{*})\right)\right] & \leq\frac{\left(1-a^{(1)}\right)T_{1}}{q^{(1)}a^{(1)}}\left(F(u^{(0)})-F(x^{*})\right)+\frac{\beta}{2}\left\Vert u^{(0)}-x^{*}\right\Vert ^{2}.
\end{align*}
\end{proof}

Using the above new lemma w.r.t. the function value gap of $u^{(S)}$,
we finally can give the proof of Theorem \ref{thm:AdaVRAG-Convergence-NA}.

\begin{proof}
(Theorem $\ref{thm:AdaVRAG-Convergence-NA}$) Note that by our choice
$a^{(s+1)}\leq a^{(s)}$ is true for any $s\neq s_{0}$. Besides,
our parameters $\left\{ a^{(s)}\right\} $ and $\left\{ q^{(s)}\right\} $
are totally the same as the choice in Theorem \ref{thm:AdaVRAG-Convergence-A}
when $s\leq s_{0}$. Hence we know
\[
\frac{(1-a^{(s+1)})T_{s+1}}{q^{(s+1)}a^{(s+1)}}\leq\frac{T_{s}}{q^{(s)}a^{(s)}}
\]
is still true for $s\leq s_{0}-1$. For $s\geq s_{0}+1$, note that
our new $\left\{ q^{(s)}\right\} $ are only different from the choice
in Theorem \ref{thm:AdaVRAG-Convergence-A} by a constant, which implies
\[
\frac{(1-a^{(s+1)})T_{s+1}}{q^{(s+1)}a^{(s+1)}}\leq\frac{T_{s}}{q^{(s)}a^{(s)}}
\]
also holds for $s\geq s_{0}+1$. Besides, we can show 
\[
\frac{\left(1-a^{(s_{0}+1)}\right)^{2}T_{s_{0}+1}}{\left(2-a^{(s_{0}+1)}\right)\left(a^{(s_{0}+1)}\right)^{2}}\leq\frac{(1-a^{(s_{0})})T_{s_{0}}}{\left(2-a^{(s_{0})}\right)\left(a^{(s_{0})}\right)^{2}}
\]
is true by plugging in the value of $a^{(s_{0}+1)}=\frac{c}{1+2c}$
and noticing that $a^{(s_{0})}\leq\frac{1}{2}$. Hence all the conditions
for Lemma \ref{lem:AdaVRAG-Whole-Descent-NA} are satisfied, then
we know for $S\leq s_{0}$,
\[
\E\left[\frac{T_{S}}{q^{(S)}a^{(S)}}\left(F(u^{(S)})-F(x^{*})\right)\right]\leq\frac{\left(1-a^{(1)}\right)T_{1}}{q^{(1)}a^{(1)}}\left(F(u^{(0)})-F(x^{*})\right)+\frac{\beta}{2}\left\Vert u^{(0)}-x^{*}\right\Vert ^{2}.
\]
For $S>s_{0}$,
\[
\E\left[\frac{\left(2-a^{(s_{0})}\right)\left(a^{(s_{0})}\right)^{2}T_{S}}{q^{(S)}a^{(S)}}\left(F(u^{(S)})-F(x^{*})\right)\right]\leq\frac{\left(1-a^{(1)}\right)T_{1}}{q^{(1)}a^{(1)}}\left(F(u^{(0)})-F(x^{*})\right)+\frac{\beta}{2}\left\Vert u^{(0)}-x^{*}\right\Vert ^{2}.
\]
By noticing
\begin{align*}
a^{(s_{0})} & =1-\left(4n\right)^{-0.5^{s_{0}}}\\
 & \geq1-\left(4n\right)^{-0.5^{\left(\log_{2}\log_{2}4n\right)+1}}\\
 & =1-\frac{1}{\sqrt{2}}\\
\Rightarrow\left(2-a^{(s_{0})}\right)\left(a^{(s_{0})}\right)^{2} & \geq\frac{2-\sqrt{2}}{4}
\end{align*}
and
\[
\frac{\left(1-a^{(1)}\right)T_{1}}{q^{(1)}a^{(1)}}=\frac{1}{4},
\]
combining the fact that our new $\left\{ q^{(s)}\right\} $ for $S>s_{0}$
have the same order of the choice in Theorem \ref{thm:AdaVRAG-Convergence-A}.
Following a similar proof, we can arrive the desired result.
\end{proof}

\section{Hyperparameter choices and additional results}

\label{sec:Additional-experiment-details}

Table \ref{tab:Hyperparameters} reports the hyperparameter choices
used in the experiments. VRAG and VRAE are the non-adaptive versions
our algorithms (Algorithms \ref{alg:VRAE} and \ref{alg:VRAG}). We
set their step sizes via a hyperparameter search as described in Section
\ref{sec:experiments}. Figures \ref{fig:apdx-experimental-results-a1a},
\ref{fig:apdx-experimental-results-mushrooms}, \ref{fig:apdx-experimental-results-w8a},
\ref{fig:apdx-experimental-results-phishing} give the experimental
evaluation of our non-adaptive algorithms. 

\begin{table}[h]
\centering{}{\small{}\caption{\label{tab:Hyperparameters}Hyperparameters used in the experiments}
\vspace{0.1in}
}%
\begin{tabular}{cccccccc}
\toprule 
{\small{}Dataset} & {\small{}Loss} & {\small{}SVRG} & {\small{}$\text{SVRG}^{++}$} & {\small{}VARAG} & {\small{}VRADA} & {\small{}VRAG} & {\small{}VRAE}\tabularnewline
\midrule
\midrule 
\multirow{3}{*}{{\small{}a1a}} & {\small{}logistic} & {\small{}0.5} & {\small{}0.5} & {\small{}1} & {\small{}1} & {\small{}1} & {\small{}1}\tabularnewline
\cmidrule{2-8} \cmidrule{3-8} \cmidrule{4-8} \cmidrule{5-8} \cmidrule{6-8} \cmidrule{7-8} \cmidrule{8-8} 
 & {\small{}squared} & {\small{}0.01} & {\small{}0.05} & {\small{}0.05} & {\small{}0.1} & {\small{}0.1} & {\small{}0.05}\tabularnewline
\cmidrule{2-8} \cmidrule{3-8} \cmidrule{4-8} \cmidrule{5-8} \cmidrule{6-8} \cmidrule{7-8} \cmidrule{8-8} 
 & {\small{}huber} & {\small{}0.05} & {\small{}0.1} & {\small{}0.1} & {\small{}0.5} & {\small{}0.1} & {\small{}0.1}\tabularnewline
\midrule
\midrule 
\multirow{3}{*}{{\small{}mushrooms}} & {\small{}logistic} & {\small{}0.5} & {\small{}1} & {\small{}1} & {\small{}1} & {\small{}1} & {\small{}1}\tabularnewline
\cmidrule{2-8} \cmidrule{3-8} \cmidrule{4-8} \cmidrule{5-8} \cmidrule{6-8} \cmidrule{7-8} \cmidrule{8-8} 
 & {\small{}squared} & {\small{}0.01} & {\small{}0.01} & {\small{}0.05} & {\small{}0.1} & {\small{}0.05} & {\small{}0.01}\tabularnewline
\cmidrule{2-8} \cmidrule{3-8} \cmidrule{4-8} \cmidrule{5-8} \cmidrule{6-8} \cmidrule{7-8} \cmidrule{8-8} 
 & {\small{}huber} & {\small{}0.05} & {\small{}0.1} & {\small{}0.1} & {\small{}0.1} & {\small{}0.1} & {\small{}0.05}\tabularnewline
\midrule
\midrule 
\multirow{3}{*}{{\small{}w8a}} & {\small{}logistic} & {\small{}0.1} & {\small{}1} & {\small{}1} & {\small{}100} & {\small{}1} & {\small{}5}\tabularnewline
\cmidrule{2-8} \cmidrule{3-8} \cmidrule{4-8} \cmidrule{5-8} \cmidrule{6-8} \cmidrule{7-8} \cmidrule{8-8} 
 & {\small{}squared} & {\small{}0.01} & {\small{}0.01} & {\small{}0.01} & {\small{}100} & {\small{}0.05} & {\small{}0.05}\tabularnewline
\cmidrule{2-8} \cmidrule{3-8} \cmidrule{4-8} \cmidrule{5-8} \cmidrule{6-8} \cmidrule{7-8} \cmidrule{8-8} 
 & {\small{}huber} & {\small{}0.01} & {\small{}0.1} & {\small{}0.1} & {\small{}100} & {\small{}0.1} & {\small{}0.5}\tabularnewline
\midrule
\midrule 
\multirow{3}{*}{{\small{}phishing}} & {\small{}logistic} & {\small{}50} & {\small{}100} & {\small{}100} & {\small{}100} & {\small{}100} & {\small{}100}\tabularnewline
\cmidrule{2-8} \cmidrule{3-8} \cmidrule{4-8} \cmidrule{5-8} \cmidrule{6-8} \cmidrule{7-8} \cmidrule{8-8} 
 & {\small{}squared} & {\small{}0.05} & {\small{}0.5} & {\small{}1} & {\small{}1} & {\small{}1} & {\small{}1}\tabularnewline
\cmidrule{2-8} \cmidrule{3-8} \cmidrule{4-8} \cmidrule{5-8} \cmidrule{6-8} \cmidrule{7-8} \cmidrule{8-8} 
 & {\small{}huber} & {\small{}0.5} & {\small{}1} & {\small{}1} & {\small{}5} & {\small{}5} & {\small{}5}\tabularnewline
\bottomrule
\end{tabular}
\end{table}

\clearpage

\begin{figure*}
\subfloat[Logistic loss]{\includegraphics[width=0.33\textwidth]{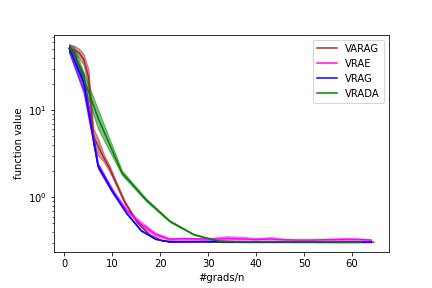}}\subfloat[Squared loss]{\includegraphics[width=0.33\textwidth]{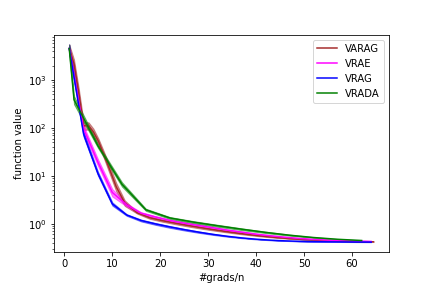}}\subfloat[Huber loss]{\includegraphics[width=0.33\textwidth]{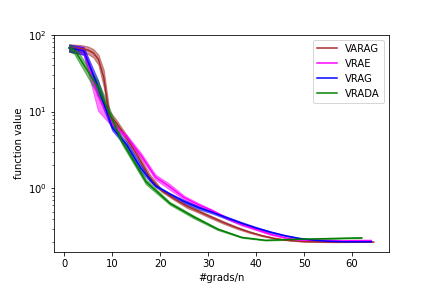}}\hfill{}\caption{a1a}

\label{fig:apdx-experimental-results-a1a}
\end{figure*}
\begin{figure*}
\subfloat[Logistic loss]{\includegraphics[width=0.33\textwidth]{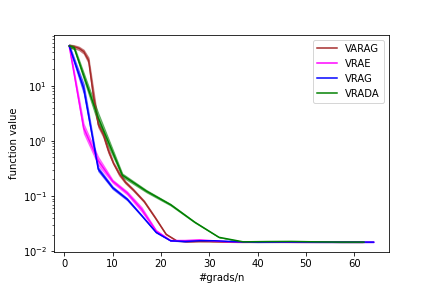}}\subfloat[Squared loss]{\includegraphics[width=0.33\textwidth]{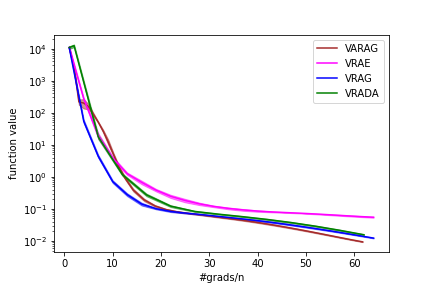}}\subfloat[Huber loss]{\includegraphics[width=0.33\textwidth]{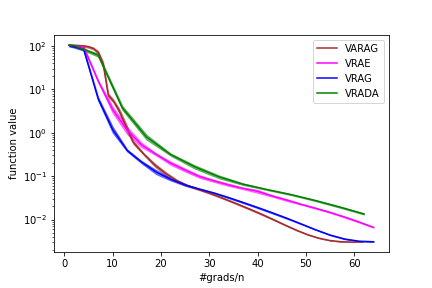}}\hfill{}\caption{mushrooms}

\label{fig:apdx-experimental-results-mushrooms}
\end{figure*}
\begin{figure*}
\subfloat[Logistic loss]{\includegraphics[width=0.33\textwidth]{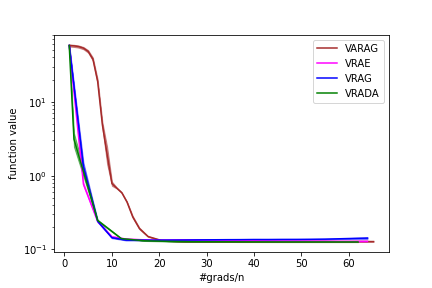}}\subfloat[Squared loss]{\includegraphics[width=0.33\textwidth]{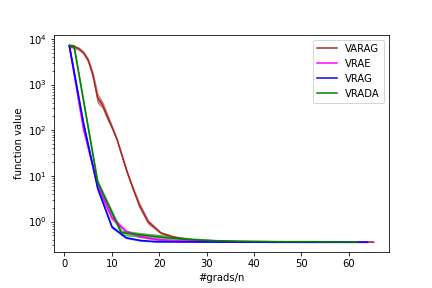}}\subfloat[Huber loss]{\includegraphics[width=0.33\textwidth]{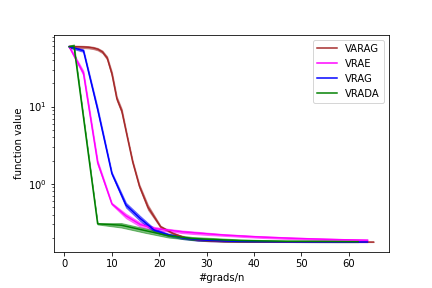}}\hfill{}\caption{w8a}

\label{fig:apdx-experimental-results-w8a}
\end{figure*}
\begin{figure*}
\subfloat[Logistic loss]{\includegraphics[width=0.33\textwidth]{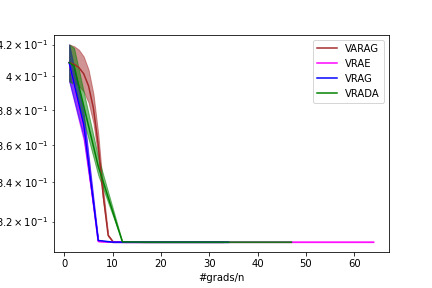}}\subfloat[Squared loss]{\includegraphics[width=0.33\textwidth]{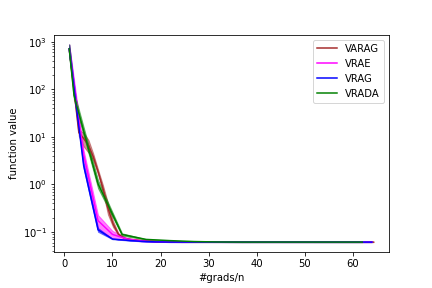}}\subfloat[Huber loss]{\includegraphics[width=0.33\textwidth]{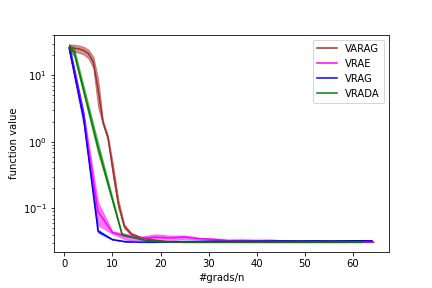}}\hfill{}\caption{phishing}

\label{fig:apdx-experimental-results-phishing}
\end{figure*}

\end{document}